%% file: main.tex
\pdfoutput=1
\documentclass[11pt]{article}

\input{preamble}

\ifpdf
\hypersetup{
	pdftitle={Distributed Gradient-Regularized Newton Method},
	pdfauthor={W. Hu, P. Xie, Y.-X. Yuan, and L. Zhang}
}
\fi

\begin{document}
	
	\maketitle
	
	\begin{abstract}
		We propose \textsc{DisGrem}, a fully decentralized second-order method
		for convex consensus optimization over networks.
		Each agent solves a local Newton system with vanishing
		gradient-norm regularization $\lambda_{i,k}=\sqrt{M\|\tilde g_{i,k}\|}$
		and an eigenvalue-shift stabilizer, communicating through a two-stage
		gossip-mixing mechanism.
		We introduce a reference-step framework that reduces the
		network-wide update to an inexact centralized regularized Newton step,
		replacing the static Hessian-heterogeneity assumptions of prior work
		with an increment-based dispersion analysis that imposes no
		irreducible accuracy floor.
		Under a bounded-iterates assumption, after a burn-in phase whose
		order is controlled by the scheduled consensus accuracy, the
		post-burn-in phase achieves
		$\|\nabla f(\bar x_k)\|\le\varepsilon$ in $\cO(\varepsilon^{-1})$
		iterations---matching the centralized regularized Newton
		rate---without line search or stepsize tuning.
		For a logarithmic schedule with $p\ge3$, the total iteration
		complexity remains $\cO(\varepsilon^{-1})$. For a fixed connected
		network, this yields $\cO(\varepsilon^{-1}\log(1/\varepsilon))$
		neighbor communication rounds; more explicitly, the dependence on
		the mixing rate is
		$\cO((1-\rho)^{-1}\varepsilon^{-1}\log(1/\varepsilon))$ as
		$\rho\to1$.
		Under strong convexity and a relative tracking-accuracy condition,
		we further establish conditional local superlinear convergence
		of order $3/2$.
		In our nine-problem benchmark suite, the \textsc{DisGrem} family attains
		$\mathrm{relF}\le 10^{-6}$ on every test instance, while the tested
		baselines stagnate or diverge on at least one problem.
	\end{abstract}

	

    \section{Introduction}\label{sec:intro}
	We consider the decentralized consensus optimization problem
	\begin{equation}\label{eq:intro_prob}
		\min_{x\in\R^d}\; f(x)\;:=\;\frac{1}{N}\sum_{i=1}^N f_i(x),
	\end{equation}
	where each local cost $f_i$ is known only to agent~$i$ of a
	connected network and agents communicate exclusively with their
	immediate neighbors.
	Throughout the paper, unless otherwise stated, unqualified vector and
	stacked-vector norms are Euclidean norms; for matrices, $\|\cdot\|_2$
	and $\|\cdot\|_F$ denote the spectral and Frobenius norms.
	Each outer iteration involves one or more gossip rounds, so communication is a primary bottleneck in decentralized implementations; consequently, reducing the iteration count typically lowers the overall communication budget.
	First-order decentralized methods may need prohibitively many
	iterations on ill-conditioned problems.
	Second-order information can cut this count substantially, yet
	decentralized second-order methods face additional difficulties:
	curvature data are expensive to transmit, and global convergence
	typically requires stepsize tuning or line search.
	{\color{black}
	These issues are part of a broader theme in modern optimization: how to exploit curvature, local approximation structure, and model information while limiting expensive oracle calls, communication, or computation. Related perspectives appear in derivative-free trust-region modeling, underdetermined quadratic interpolation, least-norm and least-$H^2$ model updating, transformed-objective optimization, subspace optimization, and approximation-based acceleration; see, for example, \citet{xie2023linesearch,xie2023dfoto,xieyuannew,10.1093/imanum/drae106,xie2023twodimensional,xie2025remuregionalminimalupdating,he2025modeldrivensubspaceslargescaleoptimization}.}
	
	In the centralized setting, the regularized Newton method of
	\citet{mishchenko2023rn} takes a different approach, updating
	the iterate via an explicitly regularized system:
	\[
	x_{k+1} = x_k - \bigl(\nabla^2 f(x_k) + \lambda_k I\bigr)^{-1} \nabla f(x_k).
	\]
	With a vanishing regularization parameter
	$\lambda_k\asymp\sqrt{\|\nabla f(x_k)\|}$, this method achieves
	the optimal $\cO(1/k^2)$ functional rate under convexity
	and Lipschitz Hessians---equivalently,
	$\|\nabla f(x_k)\|\le\varepsilon$ in $\cO(\varepsilon^{-1})$
	iterations---requiring no line search or stepsize tuning.
	\citet{doikov2024gradient} show that gradient-norm regularization
	can replace the more expensive cubic model
	\citep{nesterov2006cubic,cartis2011arc} while preserving the
	same complexity guarantees;
	\citet{gratton2023nonconvex} extend the approach to nonconvex
	objectives via negative-curvature exploitation.
	{\color{black}
	The use of regularization and curvature modification is also connected to other numerical optimization mechanisms, including trust-region methods on non-Euclidean or constrained geometries \citep{xie2023ellipsoid,XIE2025116146}, and practical derivative-free or mixed-integer model-based solvers.}
	A natural question is whether this $\cO(\varepsilon^{-1})$ iteration complexity can be preserved in a fully decentralized setting.
	The difficulty is that the vanishing regularization
	$\lambda_k\asymp\sqrt{\|\nabla f(x_k)\|}$ relies on globally
	consistent gradient and curvature information.
	When each agent sees only a local approximation corrupted
	by consensus error, the balance between regularization and curvature
	breaks down, and existing methods resort to line search or static
	heterogeneity assumptions to recover convergence.
	
	Two specific challenges arise.
	First, each agent inverts a local regularized system whose
	average generally differs from the global Newton step; this gap
	depends on inter-agent agreement and must be controlled without
	static heterogeneity constants.
	Second, Hessian trackers can become transiently indefinite due to
	imperfect consensus, threatening the well-posedness of local solves.
	Our algorithm \textsc{DisGrem} (Distributed Gradient-Regularized
	Newton Method) addresses the first challenge through an increment-based
	dispersion recursion that bounds tracker mismatch via Lipschitz
	differences, and the second through an eigenvalue-shift
	stabilizer that ensures well-posedness at every iteration.
	\textsc{DisGrem} requires only a single scaling parameter $M\ge L_2$
	and does not need stepsize tuning or line search.
	To our knowledge, no prior fully decentralized method achieves the
	same post-burn-in centralized-order iteration bound to arbitrary
	accuracy with gradient-norm regularization under these conditions.
	{\color{black}
	In this sense, our work complements recent efforts that use carefully constructed local models or surrogate information to reduce expensive optimization costs, including regional minimal updating \citep{xie2025remuregionalminimalupdating}, local approximation strategies in large-scale subspaces \citep{he2025modeldrivensubspaceslargescaleoptimization}, and neural-network approximation diagnostics based on objective-value or shape changes \citep{xie2025objectivevaluechangeshapebased}.}
	
	We summarize our main contributions as follows:
	\begin{itemize}[leftmargin=1.5em]
		\item \textbf{A fully decentralized gradient-regularized Newton framework.}
		We propose \textsc{DisGrem}, in which each agent solves a local Newton system
		with vanishing regularization $\lambda_{i,k}=\sqrt{M\|\tilde g_{i,k}\|}$.
		A two-stage gossip mechanism and an eigenvalue-shift stabilizer ensure that
		all local systems remain well posed without line search or explicit positive-semidefinite (PSD) projection.
		
		\item \textbf{An analytical reduction to an inexact centralized Newton step.}
		We introduce a reference-step construction that quantifies the gap between the
		averaged local step and the ideal global Newton step, thereby reducing the
		decentralized dynamics to an inexact centralized regularized Newton iteration.
		This framework also replaces the usual static heterogeneity assumption by an
		increment-based dispersion recursion, avoiding an irreducible
		heterogeneity floor in the final accuracy bound.
		
		\item \textbf{Global guarantees and a conditional local superlinear result.}
		Under convexity, Lipschitz Hessians, and the bounded-iterates
		assumption (Assumption~\ref{ass:bounded_iterates}),
		after a burn-in phase controlled by the scheduled consensus accuracy,
		the post-burn-in phase requires $\cO(\varepsilon^{-1})$ iterations for
		achieving $\|\nabla f(\bar x_k)\|\le\varepsilon$.
		The $\cO(\varepsilon^{-1})$ bound is conditional on
		Assumption~\ref{ass:bounded_iterates}: this trajectory boundedness is
		not proved from the problem data but is empirically verified on all
		tested instances (Section~\ref{sec:experiments}).
		For a logarithmic mixing schedule with $p\ge3$, the burn-in estimate
		is compatible with the $\cO(\varepsilon^{-1})$ post-burn-in term,
		so the total iteration complexity matches the centralized regularized
		Newton rate. For a fixed connected
		network, the communication rounds scale as
		$\cO(\varepsilon^{-1}\log(1/\varepsilon))$; explicitly, the
		spectral-gap dependence is
		$\cO((1-\rho)^{-1}\varepsilon^{-1}\log(1/\varepsilon))$ as
		$\rho\to1$.
		Under strong convexity and a relative tracking-accuracy condition,
		we further establish conditional local $Q$-superlinear convergence
		of order $3/2$.
		
		\item \textbf{Practical variants and empirical robustness.}
		We develop communication-efficient and adaptive variants, namely
		\textsc{CeDisGrem} and \textsc{AdaDisGrem}. Across nine benchmark problems,
	the \textsc{DisGrem} family achieves a strong accuracy--robustness
	balance among the tested methods, while \textsc{AdaDisGrem} achieves the
		highest robustness in the multi-start experiments.
	\end{itemize}
	
	Section~\ref{sec:related} reviews related work.
	Sections~\ref{sec:prelim}--\ref{sec:theory} present the notation, assumptions, algorithmic variants, and convergence theory.
	Section~\ref{sec:experiments} reports numerical results.
	Appendices~\ref{app:dispersion} and~\ref{app:proofs} collect the longer proofs, while Appendix~\ref{app:supp_exp} reports supplementary experiments.
	
	\section{Related work}\label{sec:related}
	
	\noindent\textbf{First-order decentralized methods.}
	Decentralized (sub)gradient descent originates with
	\citet{nedic2009subgradient}.
	EXTRA~\citep{shi2015extra} and exact
	diffusion~\citep{yuan2019exactdiffusion} remove the bias of
	constant-stepsize methods through correction steps or primal-dual
	reformulations.
	Gradient tracking
	\citep{nedic2017diging,xin2020distributed,pu2021distributed}
	achieves exact convergence with a single doubly stochastic matrix
	by accumulating gradient increments;
	\citet{alghunaim2021decentralized} unify these two perspectives.
	These methods converge at $\cO(1/k)$ or linearly for strongly convex
	objectives, but the iteration complexity scales with
	$(1-\rho)^{-1}$ or $(1-\rho)^{-2}$, a severe penalty on poorly
	connected graphs.
	Acceleration \citep{jakovetic2014fast,li2024decentralized} and
	communication compression
	\citep{koloskova2019decentralized,beznosikov2022biased} are
	orthogonal improvements; we adopt the latter for Hessian data
	in Section~\ref{ssec:variants}.
	
	\noindent\textbf{Second-order and quasi-Newton decentralized methods.}
	Network Newton~\citep{mokhtari2017networknewton} approximates
	the global Newton direction via a truncated Hessian power series;
	Newton tracking~\citep{zhang2021newtontracking} embeds Newton-type
	updates into gradient tracking.
	DQM~\citep{eisen2017dqn} replaces exact Hessians with BFGS
	surrogates; \citet{bajovic2017newton} studies distributed
	Newton-type corrections with diagonal approximations, and
	\citet{li2020communication} propose communication-efficient
	approximate Newton and variance-reduced methods for networked optimization.
	ESOM~\citep{mokhtari2016esom} combines exact second-order information
	with an alternating direction method of multipliers (ADMM) consensus step.
	SONATA~\citep{sun2022sonata} solves a sequence of strongly convex
	local surrogates and handles nonconvex objectives, while
	Network-GIANT~\citep{maritan2023networkgiant} constructs a global Newton
	direction through harmonic-mean Hessian consensus.
	All of these methods require explicit stepsize tuning, penalty
	parameter selection, or line search.
	Furthermore, to the best of our knowledge, none attains the optimal
	centralized Newton iteration complexity of $\cO(1/k^2)$ for general
	convex problems in a rigorous global sense.
	While exact methods like ESOM and Newton tracking can reach
	arbitrary precision for strongly convex objectives, their
	convergence heavily relies on static Hessian-heterogeneity
	constants (e.g., $\sigma_H$) to dictate conservative stepsizes.
	Other approximate decentralized Newton methods simply ignore this
	heterogeneity or suffer from an irreducible accuracy floor
	$\|\nabla f(\bar x)\|\le\cO(\sigma_H)$.
	\citet{daneshmand2021newton} pursue a different approach, combining
	gradient tracking with cubic regularization and local Hessian
	subsampling to achieve an iteration count comparable to centralized
	cubic Newton.
	However, their convergence guarantee holds only up to the statistical
	precision of the Hessian estimator---an inherent accuracy floor from
	subsampling---and the method still requires a stepsize parameter.
	Our increment-based dispersion analysis, combined with
	exact Hessian tracking and gradient-norm regularization, eliminates
	both the heterogeneity floor and the statistical precision floor,
	yielding convergence to arbitrary $\varepsilon$ without any stepsize.
	
	\noindent\textbf{Communication-efficient and inversion-free approaches.}
	The $\cO(d^2)$ per-round cost of transmitting Hessian data and
	the $\cO(d^3)$ cost of solving Newton systems are the two main
	bottlenecks of decentralized second-order methods.
	\citet{zhang2024cubic} combine lazy Hessian updates with
	compression for distributed cubic Newton but rely on a central
	parameter server; our \textsc{CeDisGrem} achieves analogous
	savings in a fully peer-to-peer topology.
	On the computation side, DINAS~\citep{jakovetic2025dinas} avoids
	Hessian inversion through iterative linear solvers, and
	INDO~\citep{yuan2023indo} proposes an inversion-free method for
	consensus optimization.
	Incorporating such inexact solvers into the \textsc{DisGrem}
	framework is a promising direction, as discussed in
	Section~\ref{sec:conclusion}.
	
	\noindent\textbf{Regularized Newton and cubic regularization.}
	Our starting point is the centralized regularized Newton method of
	\citet{mishchenko2023rn}: setting
	$\lambda_k\asymp\sqrt{\|\nabla f(x_k)\|}$ yields the tight
	$\cO(1/k^2)$ functional rate under convexity and Lipschitz Hessians,
	with no stepsize to choose.
	The idea has roots in cubic regularization
	\citep{nesterov2006cubic,cartis2011arc}, which achieves the same rate
	through an adaptive cubic model.
	\citet{doikov2024gradient} establish that gradient-norm
	regularization can replace the cubic model across a broad class of
	problems.
	\citet{gratton2023nonconvex} extend it to nonconvex objectives via
	negative-curvature exploitation.
	\citet{doikov2022tensor} study local convergence of higher-order
	tensor methods, and \citet{doikov2023lazy} analyze lazy Hessian
	updates in the centralized setting, an idea we also explore in
	Section~\ref{ssec:commexp}.
	As far as we know, there is no prior work that combines gradient-norm
	regularization (as opposed to cubic regularization) with full
	gradient and Hessian tracking in a decentralized setting,
	achieving a centralized-order post-burn-in iteration bound without
	stepsize tuning, line search, or static heterogeneity constants.
	
	{\color{black}
	\noindent\textbf{Model-based and approximation-driven optimization.} 
    In some (derivative-free) trust-region methods, quadratic models built from interpolation or underdetermined interpolation play a central role; see, for example, the line-search/trust-region hybrid method of \citet{xie2023linesearch}, transformed-objective derivative-free optimization \citep{xie2023dfoto}, the optimality-aware underdetermined interpolation model of \citet{xieyuannew}, and least-$H^2$ norm updating of quadratic interpolation models \citep{10.1093/imanum/drae106}. Related model-update and model-selection ideas include barycentric weight-region analysis \citep{xie2024bary}, regional minimal updating \citep{xie2025remuregionalminimalupdating}, and the relationship between geometric poisedness and outlier detection \citep{zhang2024relationshiplambdapoisednessderivativefreeoptimization}. Large-scale settings motivate subspace and local-approximation strategies, including two-dimensional model-based subspace methods \citep{xie2023twodimensional}, model-driven subspaces \citep{he2025modeldrivensubspaceslargescaleoptimization}, and numerical methods tailored for unconstrained optimization \citep{li2025novelnumericalmethodtailored}. Further related applications and extensions include privacy-preserving black-box optimization \citep{xie2025privacypreservingblackboxoptimizationpbbo}, neural-network approximation and objective-shape diagnostics \citep{xie2024lchange,xie2025objectivevaluechangeshapebased}, inverse problems under uncertainty \citep{postionpaper2025optimization}. These works differ from the decentralized Newton framework studied here, but they share the common goal of designing optimization algorithms whose local models, regularization mechanisms, or surrogate information improve robustness and reduce the dominant computational or communication cost.}

	\section{Preliminaries and assumptions}\label{sec:prelim}
	\subsection{Network model and notation}
	Unless otherwise stated, $\|\cdot\|$ denotes the Euclidean norm for vectors
	and stacked vectors. For matrices, $\|\cdot\|_2$ and $\|\cdot\|_F$ denote
	the spectral and Frobenius norms, respectively.
	
	We model the communication network as an undirected, connected
	graph $\mathcal{G} = (\mathcal{V}, \mathcal{E})$, where
	$\mathcal{V} = \{1, \dots, N\}$ is the set of agents and
	$\mathcal{E}$ is the set of communication links.
	Agents $i$ and $j$ can exchange information if and only if
	$(i,j) \in \mathcal{E}$.
	Information mixing over this network is represented by a symmetric,
	doubly stochastic matrix $W \in \R^{N \times N}$ that respects
	the graph topology, meaning $W_{ij} > 0$ only if
	$(i,j) \in \mathcal{E}$ or $i=j$.
	We assume that $W$ has positive diagonal entries and satisfies
	$W_{ij}\ge0$, $W\one=\one$, $\one^\top W=\one^\top$.
	Let $\J:=\frac{1}{N}\one\one^\top$ denote the averaging projector,
	$\Pmat:=I-\J$ the complementary projector, and
	\[
	\rho := \normtwo{W-\J}\in[0,1).
	\]
	The positive diagonal condition is satisfied by the
	Metropolis--Hastings weights used in our experiments and rules out
	the periodic case in which an eigenvalue~$-1$ would give $\rho=1$.
	For stacked vectors $Z=[z_1;\dots;z_N]\in\R^{Nd}$ we use the
	Euclidean norm $\|Z\|$ and write $\bar z:=\frac1N\sum_{i=1}^N z_i$
	for the average.
	Stacked primal variables are denoted
	$X=[x_1;\dots;x_N]\in\R^{Nd}$, and we define the separable
	network objective
	\[
	F(X):=\sum_{i=1}^N f_i(x_i).
	\]
	Then
	\begin{align*}
		\nabla F(X)&=[\nabla f_1(x_1);\dots;\nabla f_N(x_N)]\in\R^{Nd},\\
		\nabla^2 F(X)&=\blkdiag\bigl(\nabla^2 f_1(x_1),\dots,\nabla^2 f_N(x_N)\bigr),
	\end{align*}
	where $\blkdiag(\cdot)$ denotes the block-diagonal matrix.
	We also introduce the stacked Hessian vectorization
	\[
	\mathcal H(X):=\bigl[\operatorname{vec}(\nabla^2 f_1(x_1));\dots;\operatorname{vec}(\nabla^2 f_N(x_N))\bigr]\in\R^{Nd^2},
	\]
	where $\operatorname{vec}(\cdot)$ stacks the columns of a matrix into a single vector.
	When evaluated at a consensus point $x$, we write $\mathcal H(\one\otimes x)$.
	
	Table~\ref{tab:notation} collects the main symbols used
	throughout the paper.
	
	\begin{table}[htbp]
		\centering
		\caption{Notation summary.}
		\label{tab:notation}
		\setlength{\tabcolsep}{4pt}
		\footnotesize
		\begin{tabular}{@{}p{0.23\textwidth}p{0.68\textwidth}@{}}
			\toprule
			Symbol & Meaning \\
			\midrule
			$N$, $d$ & number of agents; problem dimension \\
			$x_{i,k}$ & primal iterate of agent~$i$ at iteration~$k$ \\
			$\bar x_k$ & average iterate $\frac1N\sum_{i=1}^N x_{i,k}$ \\
			$X_k$ & stacked vector $[x_{1,k};\dots;x_{N,k}]\in\R^{Nd}$ \\
			$\tilde x_{i,k}$, $\tilde g_{i,k}$, $\tilde H_{i,k}$ & pre-mixed quantities (after $\tau_k$ gossip rounds) \\
			$g_{i,k}$, $H_{i,k}$ & gradient and Hessian trackers \\
			$\tilde G_k$ & stacked pre-mixed gradient trackers $[\tilde g_{1,k};\dots;\tilde g_{N,k}]$ \\
			$W$ & doubly stochastic mixing matrix \\
			$\rho:=\|W-\J\|_2$ & mixing rate ($1{-}\rho$ is the spectral gap) \\
			$\tau_k$, $t_k$ & pre-mixing and post-mixing depths \\
			$D(Z)$ & RMS dispersion $\bigl(\frac1N\sum_{i=1}^N\|z_i-\bar z\|^2\bigr)^{1/2}$ \\
			$M$ & regularization scaling ($M\ge L_2$) \\
			$\lambda_{i,k}$ & $\sqrt{M\|\tilde g_{i,k}\|}$ (vanishing regularizer) \\
			$\delta_{i,k}$ & eigenvalue-shift stabilizer \\
			$L_1$, $L_2$ & gradient and Hessian Lipschitz constants \\
			$\|\cdot\|$, $\|\cdot\|_2$, $\|\cdot\|_F$
			& Euclidean norm for vectors; spectral and Frobenius norms for matrices \\
			\bottomrule
		\end{tabular}
	\end{table}
	
	To quantify disagreement among agents we define the root-mean-square (RMS) dispersion
	(also called consensus error or disagreement in the
	decentralized optimization literature, see e.g.\
	\citet{nedic2009subgradient,nedic2017diging})
	\(
	D(Z):=\bigl(\tfrac1N\sum_{i=1}^N\|z_i-\bar z\|^2\bigr)^{1/2}.
	\)
	In particular, writing $X_k:=[x_{1,k};\dots;x_{N,k}]$ for the stacked iterates and
	$\tilde G_k:=[\tilde g_{1,k};\dots;\tilde g_{N,k}]$ for the stacked gradient trackers,
	$D(X_k)$ and $D(\tilde G_k)$ measure the primal and gradient-tracker disagreements at iteration $k$.
	
	\subsection{Assumptions}
	\begin{assumption}\label{ass:problem}
		\leavevmode\par\nopagebreak\vspace{-0.6\baselineskip}
		\begin{enumerate}[label=(\roman*),leftmargin=2em,itemsep=4pt,topsep=4pt]
			\item Each $f_i$ is convex and twice continuously differentiable
			with $L_1$-Lipschitz gradient and $L_2$-Lipschitz Hessian:
			\[
			\norm{\nabla f_i(x)-\nabla f_i(y)}\le L_1\norm{x-y},
			\qquad
			\normtwo{\nabla^2 f_i(x)-\nabla^2 f_i(y)}\le L_2\norm{x-y}.
			\]
			\item $W$ is symmetric, doubly stochastic, has positive diagonal
			entries, and $\rho=\normtwo{W-\J}<1$.
			\item The average objective $f$ is coercive: $f(x)\to+\infty$ as $\|x\|\to\infty$.
			Equivalently, for every $c\in\R$ the sublevel set $\{x:f(x)\le c\}$ is bounded.
		\end{enumerate}
	\end{assumption}
	
	\begin{remark}\label{rem:coercivity}
		Coercivity guarantees existence of minimizers and bounded sublevel
		sets. It is satisfied, for instance, whenever the objective contains
		an~$\ell_2$ regularizer.
	\end{remark}
	
	
	\section{The \textsc{DisGrem} algorithm}\label{sec:alg}
	
	\subsection{Design motivation}\label{ssec:design}
	
	Lifting the regularized Newton method to a decentralized setting
	requires mimicking the global Newton step without a central
	coordinator.
	This presents three coupled mathematical and algorithmic
	difficulties, which motivate the design of \textsc{DisGrem}.
	
	\noindent\textbf{1.\ The gap between averaging and inversion.}
	The ideal centralized step solves the global average system
	\[
	s_k^{\star}
	\;=\;
	-\Bigl(\frac{1}{N}\sum_{i=1}^{N}\nabla^2 f_i(\bar x_k)
	+\lambda_k I\Bigr)^{\!-1}
	\frac{1}{N}\sum_{i=1}^{N}\nabla f_i(\bar x_k).
	\]
	In a decentralized network, each agent~$i$ solves a local regularized
	system and produces a local step
	$s_{i,k}=-(\tilde H_{i,k}+\tilde\lambda_{i,k}I)^{-1}\tilde g_{i,k}$.
	Due to the non-commutativity of averaging and matrix inversion,
	the average step $\bar s_k$ generally differs from the step obtained
	by inverting the averaged system:
	\[
	\bar s_k
	=
	-\frac{1}{N}\sum_{i=1}^{N}
	(\tilde H_{i,k}+\tilde\lambda_{i,k}I)^{-1}\tilde g_{i,k}
	\neq
	-\Bigl(\frac{1}{N}\sum_{i=1}^{N}
	(\tilde H_{i,k}+\tilde\lambda_{i,k}I)\Bigr)^{-1}
	\frac{1}{N}\sum_{i=1}^{N}\tilde g_{i,k},
	\]
	in general.
	This discrepancy is governed by the inter-agent disagreement in
	$(x_{i,k},g_{i,k},H_{i,k})$.
	\textsc{DisGrem} mitigates it by introducing a multi-round
	pre-mixing stage that drives the local inputs closer to their
	network-wide means, tightly controlling the step dispersion.
	
	\noindent\textbf{2.\ Transient indefiniteness from imperfect consensus.}
	Decentralized tracking propagates Hessian increments
	$\nabla^2 f_j(x_{j,k+1})-\nabla^2 f_j(x_{j,k})$ across
	the network.
	Although the true local Hessians satisfy $\nabla^2 f_i(x)\succeq 0$
	for convex objectives, the tracked surrogate decomposes as
	\[
	\tilde H_{i,k}
	\;=\;
	\underbrace{\nabla^2 f(\bar x_k)}_{\succeq\,0}
	\;+\;
	\underbrace{\bigl(\tilde H_{i,k}-\nabla^2 f(\bar x_k)\bigr)}_{\text{tracking error}},
	\]
	where the tracking error accumulates gossip-mixed Hessian increments
	and can have negative eigenvalues.
	During the transient phase before consensus is reached,
	$\lambda_{\min}(\tilde H_{i,k})$ may be negative, rendering
	the local system ill-posed.
	We introduce an eigenvalue-shift stabilizer
	$\delta_{i,k}:=\max\{0,-\lambda_{\min}(\tilde H_{i,k})\}$.
	It makes the local coefficient matrix positive definite whenever
	$\tilde g_{i,k}\ne0$. When $\tilde g_{i,k}=0$, the local step is
	set to zero.
	
	\noindent\textbf{3.\ Stepsize-free local regularization.}
	The vanishing choice
	$\lambda_k\asymp\sqrt{\|\nabla f(\bar x_k)\|}$ is responsible
	for the optimal $\cO(1/k^2)$ centralized rate.
	In the absence of global gradient knowledge, \textsc{DisGrem}
	sets a local vanishing regularizer
	$\lambda_{i,k}=\sqrt{M\|\tilde g_{i,k}\|}$.
	The scaling parameter $M\ge L_2$ acts as a damping
	mechanism: it bounds the local step magnitude
	($\|s_{i,k}\|\le\sqrt{\|\tilde g_{i,k}\|/M}$) analogous to a
	trust-region radius, removing the need for
	per-iteration stepsize tuning or line search.
	
	\subsection{\textsc{DisGrem}: full algorithm}\label{ssec:disgrem}
	
	Each node $i$ stores a primal variable $x_{i,k}\in\R^d$, a gradient tracker $g_{i,k}\in\R^d$,
	and a Hessian tracker $H_{i,k}\in\R^{d\times d}$.
	For compactness, the result of $t$ successive neighbor-gossip rounds is written as
	$\sum_{j=1}^N [W^t]_{ij}z_j$. Operationally, with
	$\mathcal N_i^+:=\mathcal N_i\cup\{i\}$, this quantity is obtained by
	$t$ local updates $z_i^{(\ell+1)}=\sum_{j\in\mathcal N_i^+}W_{ij}z_j^{(\ell)}$;
	no direct all-to-all communication is required.
	
	\begin{algorithm}[H]
		\caption{\textsc{DisGrem}}
		\label{alg:disgrem}
		\small
		\begin{algorithmic}[1]
			\State \textbf{Input:} $\{x_{i,0}\}$, $W$, $M>0$, mixing schedule $\{\tau_k, t_k\}$.
			\State \textbf{Initialize:} $g_{i,0}\gets\nabla f_i(x_{i,0})$, $H_{i,0}\gets\nabla^2 f_i(x_{i,0})$ for all $i$.
			\For{$k=0,1,2,\dots$}
			\State \textbf{(A) Pre-mixing} ($\tau_k\ge 1$ neighbor-gossip rounds):
			\Statex\qquad Initialize $z_i^{(0)}\gets(x_{i,k},g_{i,k},H_{i,k})$.
			\Statex\qquad For $\ell=0,\ldots,\tau_k-1$, update
			\[
			z_i^{(\ell+1)}
			\gets
			\sum_{j\in\mathcal N_i^+}W_{ij}z_j^{(\ell)}.
			\]
			\Statex\qquad Set $(\tilde x_{i,k},\tilde g_{i,k},\tilde H_{i,k})\gets z_i^{(\tau_k)}$.
			\State \textbf{(B) Local Newton step} (each agent $i$ in parallel):
			\Statex\qquad $\lambda_{i,k}\gets\sqrt{M\|\tilde g_{i,k}\|}$,\quad
			$\delta_{i,k}\gets\max\{0,\,-\lambda_{\min}(\tilde H_{i,k})\}$.
			\Statex\qquad If $\tilde g_{i,k}=0$, set $s_{i,k}\gets0$; otherwise solve
			\[
			\bigl(\tilde H_{i,k}+(\lambda_{i,k}+\delta_{i,k})I\bigr)s_{i,k}
			=-\tilde g_{i,k}.
			\]
			\Statex\qquad $y_{i,k+1}\gets\tilde x_{i,k}+s_{i,k}$.
			\State \textbf{(C) Post-mixing} ($t_k\ge 1$ neighbor-gossip rounds):
			\Statex\qquad Initialize $u_i^{(0)}\gets y_{i,k+1}$.
			\Statex\qquad For $\ell=0,\ldots,t_k-1$, update
			\[
			u_i^{(\ell+1)}
			\gets
			\sum_{j\in\mathcal N_i^+}W_{ij}u_j^{(\ell)}.
			\]
			\Statex\qquad Set $x_{i,k+1}\gets u_i^{(t_k)}$.
			\State \textbf{(D) Tracker updates:}
			\Statex\qquad Initialize $v_i^{(0)}\gets\tilde g_{i,k}+\nabla f_i(x_{i,k+1})-\nabla f_i(x_{i,k})$.
			\Statex\qquad Initialize $R_i^{(0)}\gets\tilde H_{i,k}+\nabla^2 f_i(x_{i,k+1})-\nabla^2 f_i(x_{i,k})$.
			\Statex\qquad For $\ell=0,\ldots,t_k-1$, update
			\[
			\begin{aligned}
			v_i^{(\ell+1)}
			&\gets
			\sum_{j\in\mathcal N_i^+}W_{ij}v_j^{(\ell)},\\
			R_i^{(\ell+1)}
			&\gets
			\sum_{j\in\mathcal N_i^+}W_{ij}R_j^{(\ell)}.
			\end{aligned}
			\]
			\Statex\qquad Set $g_{i,k+1}\gets v_i^{(t_k)}$ and $H_{i,k+1}\gets R_i^{(t_k)}$.
			\EndFor
		\end{algorithmic}
	\end{algorithm}
	
	\subsection{Communication cost per iteration}\label{ssec:comm}
	
	Each outer iteration of \textsc{DisGrem} involves $(\tau_k + 2t_k)$ rounds of neighbor communication.
	In each round every agent $i$ sends messages to and receives messages from its $|\mathcal N_i|$ neighbors.
	The payload differs across the three communication stages.
	Pre-mixing exchanges $(x,g,H)$, namely two $d$-dimensional
	vectors and one symmetric Hessian matrix.
	Post-mixing exchanges only the trial variable $y$.
	The tracker update exchanges one gradient-tracker input and one
	Hessian-tracker input.
	The post-mixing step~(C) and the tracker-update step~(D) must be performed sequentially.
	Agent $i$ must first receive the mixed $y$ variables to compute $x_{i,k+1}$; only then can it evaluate the exact gradients and Hessians at $x_{i,k+1}$ to form the tracker inputs.
	Consequently, step~(C) requires $t_k$ rounds and step~(D) requires an additional $t_k$ rounds, giving a total of $(\tau_k+2t_k)$ communication rounds per outer iteration.
	
	With double precision (8 bytes per float), the sent message volume
	per agent per outer iteration in the uncompressed algorithm is
	\begin{equation}\label{eq:comm_cost_iter}
	\begin{aligned}
		C_{\mathrm{iter}}^{\mathrm{send}}(k)
		&=8|\mathcal N_i|\Bigl[
		\tau_k\Bigl(2d+\frac{d(d+1)}{2}\Bigr)
		+t_kd
		+t_k\Bigl(d+\frac{d(d+1)}{2}\Bigr)
		\Bigr] \\
		&=8|\mathcal N_i|(\tau_k+t_k)
		\Bigl(2d+\frac{d(d+1)}{2}\Bigr)
		\quad \text{bytes}.
	\end{aligned}
	\end{equation}
	All communication figures reported in Section~\ref{sec:experiments}
	use the same stage-wise payload accounting, summed over the directed
	off-diagonal entries of the mixing matrix and reported in cumulative MB.
	
	\subsection{Practical variants}\label{ssec:variants}
	
	We develop three practical variants of Algorithm~\ref{alg:disgrem}:
	two address communication cost and parameter selection individually,
	while a third combines both mechanisms.
	All variants share the same four-step structure and differ only in the
	Hessian-tracker and regularization-parameter modules
	(see Table~\ref{tab:family} at the end of this section for a summary).
	
	\textsc{CeDisGrem} (the prefix ``Ce'' stands for communication-efficient)
	reduces communication by compressing Hessian data.
	The dominant communication cost of \textsc{DisGrem} is the $\frac{d(d+1)}{2}$-dimensional
	Hessian increment $\nabla^2 f_j(x_{j,k+1})-\nabla^2 f_j(x_{j,k})$.
	\textsc{CeDisGrem} replaces each symmetric Hessian increment
	$\Delta H_{j,k}:=\nabla^2 f_j(x_{j,k+1})-\nabla^2 f_j(x_{j,k})$
	by a compressed approximation $\mathcal C(\Delta H_{j,k})$.
	The implementation supports both element-wise Top-$k$ sparsification
	and low-rank symmetric truncation; the main experiments use Top-$k$
	with a 10\% element budget, while Appendix~\ref{app:supp_exp}
	compares the two choices.
	For the low-rank option, let
	$\Delta H_{j,k}=V\operatorname{diag}(\mu_1,\dots,\mu_d)V^\top$ with
	$|\mu_1|\ge \cdots \ge |\mu_d|$ and $V=[v_1,\dots,v_d]$ orthogonal.
	Define the truncated approximation
	\[
	\hat\Delta H_{j,k}:=\sum_{\ell=1}^r \mu_\ell v_\ell v_\ell^\top
	= V_r\Lambda_r V_r^\top,
	\]
	where $V_r=[v_1,\dots,v_r]\in\R^{d\times r}$ and $\Lambda_r=\operatorname{diag}(\mu_1,\dots,\mu_r)\in\R^{r\times r}$.
	This $\hat\Delta H_{j,k}$ is the best rank-$r$ symmetric approximation of $\Delta H_{j,k}$ in Frobenius norm,
	since the best rank-$r$ approximation of a symmetric matrix retains the $r$ eigenpairs with largest absolute eigenvalues.
	Thus, transmitting $(V_r,\Lambda_r)$ costs $r(d+1)$ floats instead of $\frac{d(d+1)}{2}$ floats per neighbor per round.
	
	The truncation error satisfies
	\begin{align*}
		\normF{\Delta H_{j,k}\!-\!\hat\Delta H_{j,k}}
		&=\Bigl(\textstyle\sum_{\ell=r+1}^d \mu_\ell^2\Bigr)^{\!1/2}\\
		&\le |\mu_{r+1}(\Delta H_{j,k})|\sqrt{d\!-\!r},
	\end{align*}
	where $\mu_{r+1}(\Delta H_{j,k})$ denotes the $(r+1)$-th largest absolute eigenvalue.
	For Top-$k$, $\mathcal C(\Delta H_{j,k})$ retains the largest-magnitude
	entries and symmetrizes the result.
	The compression error can be viewed as an additional Hessian-tracker
	perturbation. In the experiments, both Top-$k$ and low-rank compression reduce communication
	with modest iteration overhead when the compression level is not too
	aggressive (Section~\ref{ssec:commexp}).
	
	\begin{remark}
		The exact eigendecomposition of $\Delta H_{j,k}\in\R^{d\times d}$
		costs $\cO(d^3)$ per agent per iteration---the same order as the
		Newton-system solve.
		When $d$ is large, randomized singular value decomposition (SVD)
		(cost $\cO(d^2 r)$) or Nystr\"om-type approximations can replace
		the exact decomposition.
		The additional approximation error enters the tracker dispersion
	additively.
	\end{remark}
	
	\textsc{AdaDisGrem} introduces an adaptive scaling factor.
	When $M$ is fixed, choosing a good value typically relies on
	problem-specific curvature information (see Section~\ref{ssec:practical}).
	\textsc{AdaDisGrem} instead updates a local scaling value
	$\hat M_{i,k}$ from a secant ratio:
	\begin{gather*}
		\hat L_{i,k}\gets
		\frac{\normtwo{\nabla^2 f_i(x_{i,k})-\nabla^2 f_i(x_{i,k-1})}}
		{\norm{x_{i,k}-x_{i,k-1}}},\\
		\hat M_{i,k}\gets \max\!\bigl(\gamma\,\hat M_{i,k-1},\;
		\zeta\,\min(\hat L_{i,k},\,\eta_c\,\hat M_{i,0})\bigr),
	\end{gather*}
	where $\zeta\ge1$ is a safety inflation factor, $\gamma\in(0,1)$ is
	the decay factor, and $\eta_c>0$ is the upper-bound parameter.
	If $x_{i,k}=x_{i,k-1}$, we set $\hat L_{i,k}=0$.
	The vanishing regularizer becomes $\lambda_{i,k}=\sqrt{\hat M_{i,k}\|\tilde g_{i,k}\|}$.
	The decay factor allows $\hat M_{i,k}$ to decrease when the local
	curvature scale is moderate, enabling larger Newton steps near the
	solution, while the upper bound damps isolated large secant ratios.
	The quantity $\hat L_{i,k}$ is a per-iterate secant
	indicator, distinct from the static proxy
	$H_{\max}^0 := \max_i\|\nabla^2 f_i(x_0)\|_2$ used to
	scale the baseline $M$ (Section~\ref{ssec:practical}).
	In the experiments, \textsc{AdaDisGrem} serves as an empirical
	parameter-selection variant of the fixed-$M$ method.
	Section~\ref{ssec:ada} studies its robustness across a $100\times$
	range of initial
	$\hat M_{i,0}$.
	
	\textsc{CeAdaDisGrem} applies both Hessian compression (Ce) and
	adaptive $M$ (Ada) simultaneously.
	Because the compressed Hessian increments perturb the secant-based
	scaling rule, Section~\ref{ssec:commexp} reports this combined variant
	separately from the fixed-$M$ compression results.
	
	Table~\ref{tab:family} summarizes the full \textsc{DisGrem} family.
	All members share the core structure of Algorithm~\ref{alg:disgrem};
	differences are confined to the Hessian-tracker and
	regularization-parameter modules.
	
	\begin{table}[htbp]
		\centering
		\caption{The \textsc{DisGrem} algorithm family.
			All variants use the same four-step structure (Algorithm~\ref{alg:disgrem});
			differences are confined to the Hessian communication and the choice of~$M$.}
		\label{tab:family}
		\setlength{\tabcolsep}{3pt}
		\footnotesize
		\begin{tabular}{lccc}
			\toprule
			Variant & Hessian & $M$ & Extra parameter(s) \\
			\midrule
			\textsc{DisGrem}       & exact tracking       & fixed  & $M$ \\
			\textsc{CeDisGrem}     & compressed + lazy     & fixed  & $M$, comp., budget, $K_{\mathrm{lazy}}$ \\
		\textsc{AdaDisGrem}    & exact tracking       & adaptive & $\hat M_{i,0}$, $\gamma$, $\zeta$, $\eta_c$ \\
		\textsc{CeAdaDisGrem}  & compressed + lazy     & adaptive & Ada params, comp., budget, $K_{\mathrm{lazy}}$ \\
			\bottomrule
		\end{tabular}
	\end{table}
	
	\subsection{Practical guidelines for parameter selection}\label{ssec:practical}
	
	The theoretical convergence requires $M\ge L_2$
	(Theorem~\ref{thm:main}).
	When the Hessian Lipschitz constant $L_2$ is unknown, $M$ acts as a
	robustness parameter: larger values increase damping and ensure
	well-conditioning, though potentially slowing asymptotic convergence.
	A practical baseline choice is
	$M = M_{\mathrm{fac}} \cdot H_{\max}^{0}$, where
	$H_{\max}^{0} := \max_i\|\nabla^2 f_i(x_0)\|_2$ is a readily
	computable baseline curvature proxy and
	$M_{\mathrm{fac}} \in [0.1, 15]$ is a tuning factor that scales
	with the problem's ill-conditioning.
	\textsc{AdaDisGrem} replaces this fixed choice by the online secant
	rule described above.
	
	The communication depths are chosen by the logarithmic schedule
	\[
	\tau_k=t_k=
	\Bigl\lceil \frac{p\log(k+2)+c_{\mathrm{mix}}}{-\log\rho}\Bigr\rceil,
	\]
	with a fixed constant $c_{\mathrm{mix}}\ge0$.
	In the numerical experiments we impose a maximum communication depth
	once the tested accuracy range has been reached.
	
	For Hessian compression in \textsc{CeDisGrem}, Top-$k$ with a 10\%
	element budget is used in the main experiments.
	For low-rank compression, $r=\lceil d/5\rceil$ serves as a robust
	nominal value, while $r=1$ often suffices for diagonally dominant objectives.
	To further reduce communication, the Hessian tracker update
	(step~(D) of Algorithm~\ref{alg:disgrem}) can be executed every
	$K_{\mathrm{lazy}}$-th iteration.
	Reusing the previous Hessian for
	$K_{\mathrm{lazy}}\in\{5,10\}$ consecutive steps reduces
	payload bytes with minimal impact on convergence for slowly varying
	problems.
	
	
	\section{Convergence analysis}\label{sec:theory}
	Throughout this section we assume Assumptions~\ref{ass:problem}
	and~\ref{ass:bounded_iterates}, the parameter requirement $M\ge L_2$,
	and that
	Algorithm~\ref{alg:disgrem} is run with the initialization
	$g_{i,0}=\nabla f_i(x_{i,0})$, $H_{i,0}=\nabla^2 f_i(x_{i,0})$.
	Complete proofs are collected in Appendix~\ref{app:proofs}, and the
	dispersion-decay and burn-in construction is detailed in
	Appendix~\ref{app:dispersion}.
	
	The convergence argument proceeds through the following chain of
	reductions.
	Averaging identities (\S\ref{ssec:avg}) show that the average iterate evolves as
	$\bar x_{k+1}=\bar x_k+\bar s_k$ and that tracked averages
	equal the averages of exact local gradients and Hessians evaluated
	at the current local iterates.
	Reference step and bridge bounds
	(\S\ref{ssec:refstep}--\S\ref{ssec:bridge}) introduce a
	``centralized reference'' step $s_k^{\mathrm{ref}}$ and bound
	its gap from $\bar s_k$ in terms of tracker dispersions.
	Stabilizer and step-dispersion control
	(\S\ref{ssec:stab}--\S\ref{ssec:stepdisp}) relate the
	eigenvalue shift $\delta_{i,k}$ and the step dispersion to
	gradient/Hessian tracker dispersions.
	These three ingredients are combined in \S\ref{ssec:inexact}
	to verify that the
	average iterate satisfies an inexact Newton residual bound,
	from which a $3/2$-recursion on the
	optimality gap (\S\ref{ssec:descent}) yields the $\cO(\varepsilon^{-1})$ rate.
	
	In summary, the logical dependency chain is:
	averaging identities $\to$ reference step $\to$ stabilizer/dispersion bounds
	$\to$ inexact Newton condition (Proposition~\ref{prop:inexact-rn})
	$\to$ $3/2$-recursion on the steady subsequence
	(Lemma~\ref{lem:solve32})
	$\to$ global complexity (Theorem~\ref{thm:main}).
	Appendix~\ref{app:dispersion} supplies the burn-in index~$K_0(\varepsilon)$.
	
	The proof combines two independent stages (detailed in
	Remark~\ref{rem:two_stages} after the main theorem).
	The reader interested only in the final result may skip directly to
	Theorem~\ref{thm:main}.
	
	\subsection{Bounded trajectory and local constants}\label{ssec:iterate_bound}
	
	\begin{assumption}[Bounded iterates]\label{ass:bounded_iterates}
		There exists a solution $x_\star\in\arg\min f$ and a finite constant
		$D>0$ such that the iterates generated by Algorithm~\ref{alg:disgrem}
		satisfy
		\begin{equation}\label{eq:bounded_levelset}
			\norm{\bar x_k\!-\!x_\star}\!\le\! D,\;
			\norm{x_{i,k}\!-\!x_\star}\!\le\! D,\;
			\norm{\tilde x_{i,k}\!-\!x_\star}\!\le\! D,
		\end{equation}
		for all $k\ge 0$ and all $i\in\{1,\dots,N\}$.
	\end{assumption}
	
	Under Assumption~\ref{ass:bounded_iterates}, we define all subsequent
	constants on the compact set $\{x:\|x-x_\star\|\le D\}$.
	Since each $\nabla f_i$ is globally $L_1$-Lipschitz, the Hessian bound
	\begin{equation}\label{eq:MHmax_def}
		M_{H,\max}:=L_1
	\end{equation}
	is valid for all iterates and is independent of the compact set.
	
	\begin{remark}
		Since each $f_i$ is twice continuously differentiable (Assumption~\ref{ass:problem}(i))
		and the iterates are confined to a compact set by
		Assumption~\ref{ass:bounded_iterates},
		the quantity $\sup_{x\in\mathcal B}\max_{1\le i\le N}\|\nabla^2 f_i(x)\|_2$
		is automatically finite for any bounded $\mathcal B\subset\R^d$.
		No separate assumption is needed.
	\end{remark}
	
	\begin{remark}
		Prior decentralized second-order analyses
		\citep{mokhtari2016esom,sun2022sonata,maritan2023networkgiant} posit a bounded
		static heterogeneity constant
		$\sigma_H:=\sup_{\|x-x_\star\|\le D}
		\frac{1}{\sqrt N}\|(\Pmat\otimes I_{d^2})\mathcal H(\one\otimes x)\|$.
		We avoid this entirely: Hessian-tracker dispersion is controlled
		by an increment-based recursion
		(Appendix~\ref{app:dispersion}, Lemma~\ref{lem:DH_rec})
		that tracks mismatch through Lipschitz-continuous
		differences $\nabla^2 F(X_{k+1})-\nabla^2 F(X_k)$,
		so no accuracy floor linked to~$\sigma_H$ appears in the
		final bounds.
	\end{remark}
	
	\subsection{Averaging identities}\label{ssec:avg}
	The starting point of the analysis is that doubly stochastic mixing
	preserves averages, so the mean iterate and mean trackers evolve as if
	they were computed by a single ``virtual agent.''
	
	Define averages $\bar x_k:=\frac1N\sum_{i=1}^N x_{i,k}$ and similarly for other variables.
	Define pre-mixed averages
	\begin{gather*}
		\tilde{\bar x}_k:=\tfrac1N\!\sum_{i=1}^N \tilde x_{i,k},\quad
		\tilde{\bar g}_k:=\tfrac1N\!\sum_{i=1}^N \tilde g_{i,k},\\
		\tilde{\bar H}_k:=\tfrac1N\!\sum_{i=1}^N \tilde H_{i,k}.
	\end{gather*}
	Define the average step $\bar s_k:=\frac1N\sum_{i=1}^N s_{i,k}$.
	
	\begin{lemma}\label{lem:avg-pres}
		For all $k\ge0$, $\tilde{\bar x}_k=\bar x_k$ and $\bar x_{k+1}=\bar x_k+\bar s_k$.
	\end{lemma}
	\begin{proof}
		Immediate from $W\one=\one$ (doubly stochastic),
		which gives $\frac1N\one^\top(W^t\otimes I_d)=\frac1N\one^\top\otimes I_d$
		for every integer $t\ge 1$.
	\end{proof}
	
	\begin{lemma}\label{lem:avg-track-g}
		If Algorithm~\ref{alg:disgrem} is initialized by
		$g_{i,0}=\nabla f_i(x_{i,0})$, then for all $k\ge0$,
		\[
		\bar g_k:=\frac1N\sum_{i=1}^N g_{i,k}=\frac1N\sum_{i=1}^N\nabla f_i(x_{i,k}),
		\qquad
		\tilde{\bar g}_k=\bar g_k.
		\]
	\end{lemma}
	\begin{proof}
		By induction using $W\one=\one$ and the telescoping form of the tracker update (step~(D)).
	\end{proof}
	
	\begin{lemma}\label{lem:avg-track-h}
		If Algorithm~\ref{alg:disgrem} is initialized by
		$H_{i,0}=\nabla^2 f_i(x_{i,0})$, then for all $k\ge0$,
		\[
		\bar H_k:=\frac1N\sum_{i=1}^N H_{i,k}=\frac1N\sum_{i=1}^N\nabla^2 f_i(x_{i,k}),
		\qquad
		\tilde{\bar H}_k=\bar H_k.
		\]
	\end{lemma}
	\begin{proof}
		The Hessian tracker has the same telescoping form as the
		gradient tracker in Lemma~\ref{lem:avg-track-g}: multiplication
		by $W$ preserves the block average, and the Hessian increment
		$\nabla^2 f_i(x_{i,k+1})-\nabla^2 f_i(x_{i,k})$ telescopes in
		the averaged update.
	\end{proof}
	
	\subsection{Well-posed local solves and a reference step}\label{ssec:refstep}
	With the averaging identities in hand, we next introduce the
	reference step $s_k^{\mathrm{ref}}$: the Newton step that a
	centralized agent would compute using the averaged Hessian and gradient.
	The gap $\|\bar s_k - s_k^{\mathrm{ref}}\|$ then quantifies how much the
	decentralized updates deviate from this ideal step.
	
	Algorithm~\ref{alg:disgrem} sets
	\begin{align*}
		\lambda_{i,k}&:=\sqrt{M\|\tilde g_{i,k}\|},\\
		\delta_{i,k}&:=\max\{0,-\lambda_{\min}(\tilde H_{i,k})\},\\
		\tilde\lambda_{i,k}&:=\lambda_{i,k}+\delta_{i,k}.
	\end{align*}
	Define
	\[
	A_{i,k}:=\tilde H_{i,k}+\tilde\lambda_{i,k}I,
	\]
	When $\tilde g_{i,k}\ne0$, the coefficient matrix is positive
	definite and Algorithm~\ref{alg:disgrem} sets
	$s_{i,k}:=-A_{i,k}^{-1}\tilde g_{i,k}$.
	When $\tilde g_{i,k}=0$, Algorithm~\ref{alg:disgrem} uses the
	convention $s_{i,k}=0$.
	All inverse-based estimates below are stated on indices where the
	corresponding regularization lower bound is positive; in particular,
	Lemma~\ref{lem:step-disp} assumes $\underline\lambda_k>0$.
	Define averaged quantities
	\begin{gather*}
		\tilde{\bar\lambda}_k:=\tfrac1N\textstyle\sum_{i=1}^N \tilde\lambda_{i,k},\quad
		A_k^{\mathrm{ref}}:=\tilde{\bar H}_k+\tilde{\bar\lambda}_k I.
	\end{gather*}
	Whenever $A_k^{\mathrm{ref}}$ is nonsingular, the reference step is
	defined by $s_k^{\mathrm{ref}}=-(A_k^{\mathrm{ref}})^{-1}\tilde{\bar g}_k$.
	In the analysis below this definition is used only on indices where
	the regularization lower bound is positive.
	
	\begin{lemma}\label{lem:M-vs-s}
		For all $i,k$, the local step satisfies
		\[
		M\norm{s_{i,k}}\le \lambda_{i,k}
		\qquad\text{and hence}\qquad
		L_2\norm{s_{i,k}}\le \lambda_{i,k}.
		\]
	\end{lemma}
	
	\begin{lemma}\label{lem:avg-regime}
		For all $k$,
		\[
		L_2\|\bar s_k\|\le \frac1N\sum_{i=1}^N \lambda_{i,k}\le \tilde{\bar\lambda}_k.
		\]
	\end{lemma}
	
	\subsection{Bridge bounds: tracked averages versus true quantities}\label{ssec:bridge}
	The reference step uses the tracked averages
	$\tilde{\bar g}_k,\tilde{\bar H}_k$ rather than the true gradient
	$g_k:=\nabla f(\bar x_k)$ and true Hessian
	$\nabla^2 f(\bar x_k)$.
	By Lemma~\ref{lem:avg-track-g},
	$\tilde{\bar g}_k = \frac{1}{N}\sum_{i=1}^N\nabla f_i(x_{i,k})$,
	which coincides with $\nabla f(\bar x_k)$ only at exact
	consensus ($x_{i,k}=\bar x_k$ for all~$i$).
	Away from consensus the discrepancy is controlled by the
	bridge lemmas below, which quantify it in terms of the
	post-mixing spatial disagreement $D(X_k)$.
	All constants depending on iterate boundedness
	(e.g., $L_1,L_2$) are evaluated on the compact set fixed by
	Assumption~\ref{ass:bounded_iterates}.
	
	The post-mixing disagreement (RMS) is
	$D(X_k)=\bigl(\frac1N\sum_{i=1}^N \|x_{i,k}-\bar x_k\|^2\bigr)^{1/2}$,
	i.e., the dispersion $D(\cdot)$ defined in Section~\ref{sec:prelim}
	applied to the stacked iterate $X_k$.
	
	\begin{lemma}\label{lem:bridge-grad}
		For every $k$,
		\[
		\|\nabla f(\bar x_k)-\tilde{\bar g}_k\|\le L_1\,D(X_k).
		\]
	\end{lemma}
	
	\begin{lemma}\label{lem:bridge-hess}
		For every $k$,
		\[
		\|\nabla^2 f(\bar x_k)-\tilde{\bar H}_k\|_2\le L_2\,D(X_k).
		\]
	\end{lemma}
	
	\subsection{The eigenvalue-shift stabilizer \texorpdfstring{$\delta_{i,k}$}{delta\_ik}}\label{ssec:stab}
	The stabilizer $\delta_{i,k}$ makes the local coefficient matrix
	positive definite whenever $\tilde g_{i,k}\ne0$. When
	$\tilde g_{i,k}=0$, Algorithm~\ref{alg:disgrem} uses the convention
	$s_{i,k}=0$. Therefore the local step is well posed in all cases.
	Because it acts as an additive bias in $\tilde\lambda_{i,k}$, we need
	to show that its average $\bar\delta_k$ is controlled by quantities
	that vanish as consensus improves.
	
	Define the stabilizer average
	\[
	\bar\delta_k:=\frac1N\sum_{i=1}^N \delta_{i,k}.
	\]
	
	\begin{lemma}\label{lem:delta_control}
		For every $k$ and every $i$,
		\[
		0\le \delta_{i,k}\le \big\|\tilde H_{i,k}-\nabla^2 f(\bar x_k)\big\|_2.
		\]
		Consequently,
		\begin{gather*}
			\bar\delta_k\le \Delta_k^H + L_2D(X_k),\\
			\Delta_k^H:=\tfrac1N\!\sum_{i=1}^N\big\|\tilde H_{i,k}-\tilde{\bar H}_k\big\|_2 .
		\end{gather*}
	\end{lemma}
	
	\subsection{Step dispersion via a resolvent identity}\label{ssec:stepdisp}
	The inexact Newton condition (Proposition~\ref{prop:inexact-rn} below)
	requires three quantities to be small: the step dispersion
	$\|\bar s_k - s_k^{\mathrm{ref}}\|$, the gradient bridge error
	$L_1 D(X_k)$, and the Hessian bridge error $L_2 D(X_k)$.
	The latter two are already controlled by
	Lemmas~\ref{lem:bridge-grad}--\ref{lem:bridge-hess};
	this subsection handles the first.
	The key tool is a resolvent identity that factorizes the
	difference of two linear-system solutions through their
	coefficient matrices.
	
	Define dispersions
	\begin{align*}
		\Delta_k^g&:=\tfrac1N\textstyle\sum_{i=1}^N \|\tilde g_{i,k}-\tilde{\bar g}_k\|,\\
		\Delta_k^\lambda&:=\tfrac1N\textstyle\sum_{i=1}^N |\tilde\lambda_{i,k}-\tilde{\bar\lambda}_k|,\\
		\underline\lambda_k&:=\min_{1\le i\le N}\tilde\lambda_{i,k}.
	\end{align*}
	
	\begin{lemma}\label{lem:step-disp}
		For every $k$ with $\underline\lambda_k>0$,
		\[
		\|\bar s_k-s_k^{\mathrm{ref}}\|
		\le \frac{1}{\underline\lambda_k}\Delta_k^g
		+\frac{\|\tilde{\bar g}_k\|}{\underline\lambda_k\,\tilde{\bar\lambda}_k}\big(\Delta_k^H+\Delta_k^\lambda\big).
		\]
	\end{lemma}
	
	\begin{lemma}\label{lem:lambda_dispersion}
		If $\tilde{\bar g}_k\neq 0$ and the relative gradient dispersion condition
		\begin{equation}\label{eq:rel_grad_disp}
			\max_{1\le i\le N}\|\tilde g_{i,k}-\tilde{\bar g}_k\|\le \alpha_d\|\tilde{\bar g}_k\|,
			\qquad\text{for some }\alpha_d\in(0,1),
		\end{equation}
		then
		\begin{equation}\label{eq:lambda_disp_bound}
			\Delta_k^\lambda
			\le
			\frac{\sqrt{M}}{\sqrt{1-\alpha_d}}\cdot\frac{\Delta_k^g}{\sqrt{\|\tilde{\bar g}_k\|}}
			+2\bar\delta_k.
		\end{equation}
	\end{lemma}
	
	\subsection{Inexact regularized Newton condition for the average step}\label{ssec:inexact}
	Collecting the bridge, stabilizer, and step-dispersion bounds, we
	verify that the average iterate $\bar x_k$ satisfies a standard
	inexact regularized Newton condition once the burn-in phase
	is complete, that is, for $k\ge K_0(\varepsilon)$, where $K_0$ is the
	burn-in index from Proposition~\ref{prop:burnin}.
	This is the key step, as it permits invoking the one-step descent
	lemma from the centralized analysis.
	
	Specifically, we define $g_k:=\nabla f(\bar x_k)$, $\lambda_k:=\tilde{\bar\lambda}_k$, and $r_k:=(\nabla^2 f(\bar x_k)+\lambda_k I)\bar s_k+g_k$.
	
	\begin{proposition}\label{prop:inexact-rn}
		If the following hold at iteration $k$ for some $\eta\in(0,1)$:
		\begin{enumerate}[leftmargin=1.2em,itemsep=2pt]
			\item Dispersion control:
			\[
			\|\bar s_k-s_k^{\mathrm{ref}}\|\le \frac{\eta}{8}\cdot \frac{\lambda_k}{M_{H,\max}+\lambda_k}\,\|\bar s_k\|.
			\]
			\item Gradient bridge accuracy:
			\[
			L_1D(X_k)\le \frac{\eta}{8}\lambda_k\|\bar s_k\|.
			\]
			\item Hessian bridge accuracy:
			\[
			L_2D(X_k)\le \frac{\eta}{8}\lambda_k.
			\]
		\end{enumerate}
		then
		\[
		\|r_k\|\le \eta\,\lambda_k\|\bar s_k\|,
		\qquad
		L_2\|\bar s_k\|\le \lambda_k.
		\]
	\end{proposition}
	
	\subsection{Descent and global rates}\label{ssec:descent}
	The preceding subsections have established the four building blocks
	of the proof pipeline: the average
	iterate satisfies an inexact regularized Newton condition
	(Proposition~\ref{prop:inexact-rn}) once the burn-in phase ends.
	The remaining descent step follows the centralized analysis of
	\citet{mishchenko2023rn}: a one-step descent lemma yields a
	$3/2$-power recursion on the optimality gap, from which the global
	complexity bound follows.
	
	We introduce the optimality gap $\Phi_k:=f(\bar x_k)-f_\star$.
	
	\begin{lemma}\label{lem:descent}
		Assume
		\[
		\|(\nabla^2 f(\bar x_k)+\lambda_k I)\bar s_k+g_k\|\le \eta\lambda_k\|\bar s_k\|,
		\qquad
		L_2\|\bar s_k\|\le \lambda_k,
		\]
		for some $\eta\in(0,5/6)$.
		Then
		\[
		f(\bar x_{k+1})\le f(\bar x_k) - \Big(1-\eta-\frac{1}{6}\Big)\lambda_k\|\bar s_k\|^2.
		\]
	\end{lemma}
	
	\begin{lemma}\label{lem:grad_bound}
		Under the conditions of Lemma~\ref{lem:descent},
		\[
		\|g_{k+1}\|\le \Big(1+\eta+\frac{1}{2}\Big)\lambda_k\|\bar s_k\|.
		\]
	\end{lemma}
	
	We partition the iteration indices into \emph{steady} iterations
	(where the gradient norm does not drop sharply) and
	\emph{super-descent} iterations (where it drops by at least a factor of 4):
	\begin{gather*}
		\mathcal I:=\{k\ge0:\ \|g_{k+1}\|\ge \tfrac14\|g_k\|\},\\
		\mathcal S:=\{k\ge0:\ \|g_{k+1}\|< \tfrac14\|g_k\|\}.
	\end{gather*}
	
	The following lemma establishes a sharp decrease in the optimality gap
	during steady-descent iterations.
	The auxiliary requirement $\lambda_k\le C_\lambda\sqrt{\|g_k\|}$ is
	rigorously verified in Appendix~\ref{app:proofs}
	(Lemma~\ref{lem:lambda_scaling}).
	
	\begin{lemma}\label{lem:steady}
		If Proposition~\ref{prop:inexact-rn} holds for all $k$ with some $\eta\le 1/12$
		and there exists $C_\lambda>0$ such that
		\[
		\lambda_k\le C_\lambda\sqrt{\|g_k\|}\qquad\text{for all $k$ with $g_k\neq 0$}.
		\]
		Then there exists $\nu>0$ such that for all $k\in\mathcal I$,
		\[
		\Phi_k-\Phi_{k+1}\ge \nu\,\Phi_k^{3/2}.
		\]
	\end{lemma}
	
	\begin{lemma}\label{lem:solve32}
		If $\Phi_{k+1}\le \Phi_k-\nu\Phi_k^{3/2}$ for some $\nu>0$ and all $k$, then
		\[
		\Phi_k\le \frac{4}{\nu^2(k+2)^2}.
		\]
	\end{lemma}
	
	Throughout we use the logarithmic mixing schedule
	\begin{equation}\label{eq:log_schedule_main}
		\tau_k \;=\; t_k \;=\;
		\Bigl\lceil \frac{p\,\log(k+2)+c_{\mathrm{mix}}}{-\log\rho} \Bigr\rceil,
		\qquad k\ge 0,
	\end{equation}
	where $c_{\mathrm{mix}}\ge0$ is a fixed constant. Then
	$\rho^{\tau_k}\le e^{-c_{\mathrm{mix}}}(k+2)^{-p}$; details and
	motivation are given in Appendix~\ref{app:dispersion}.
	
	\begin{remark}\label{rem:rho_knowledge}
		The schedule~\eqref{eq:log_schedule_main} requires the spectral
		gap $1{-}\rho$, or equivalently, $\rho=\|W{-}\J\|_2$.
		Computing $\rho$ exactly is a centralized operation; in a fully
		decentralized implementation one may use any certified upper
		bound $\hat\rho\ge\rho$ available for the chosen weight matrix.
		Overestimating $\rho$ increases each $\tau_k$ by a constant
		factor but does not affect the $\cO(\varepsilon^{-1})$ iteration
		complexity.
		If only a rough network description is available, one may use a
		conservative upper bound; the asymptotic statement is unchanged.
	\end{remark}
	
	\begin{theorem}\label{thm:main}
		If Assumptions~\ref{ass:problem} and~\ref{ass:bounded_iterates}
		hold and $M\ge L_2$, then the following is true.
		Fix $p\ge3$ and run \textsc{DisGrem} with the logarithmic mixing schedule~\eqref{eq:log_schedule_main}.
		Then for every target accuracy $0<\varepsilon\le1$, there exist a finite burn-in index $K_0(\varepsilon)$
		(depending on the spectral gap, Lipschitz constants, initial dispersions, and $\varepsilon$;
		see Proposition~\ref{prop:burnin})
		and a constant $C_K>0$ (independent of $\varepsilon$) such that the total number of
		outer iterations to achieve $\|\nabla f(\bar x_k)\|\le \varepsilon$ satisfies
		\[
		K(\varepsilon)\;\le\; K_0(\varepsilon) + C_K\,\varepsilon^{-1}.
		\]
		In particular, the post-burn-in phase requires at most
		$\cO(\varepsilon^{-1})$ iterations. Since $p\ge3$ implies
		$K_0(\varepsilon)=\cO(\varepsilon^{-1})$ by
		Remark~\ref{rem:burnin_order}, the total iteration complexity
		is $K(\varepsilon)=\cO(\varepsilon^{-1})$, matching the
		centralized regularized Newton complexity of \citet{mishchenko2023rn},
		and under the schedule~\eqref{eq:log_schedule_main}, for a fixed
		connected network, the total number of neighbor communication
		rounds is $\cO(\varepsilon^{-1}\log(1/\varepsilon))$.
		More explicitly, its dependence on the mixing rate is
		$\cO((1-\rho)^{-1}\varepsilon^{-1}\log(1/\varepsilon))$ as
		$\rho\to1$.
	\end{theorem}
	
	\begin{remark}\label{rem:burnin_order}
		For $0<\varepsilon\le1$, the direct verification of
		Proposition~\ref{prop:inexact-rn}
		uses the explicit step-dispersion estimate in
		Proposition~\ref{prop:burnin}. Because
		$\|\tilde{\bar g}_k\|\gtrsim\varepsilon$,
		$\|s_k^{\mathrm{ref}}\|\gtrsim\varepsilon$, and
		$\lambda_k\gtrsim\sqrt{\varepsilon}$ on the burn-in tail, a
		conservative sufficient set of conditions for Item~1 is
		\[
		\Delta_k^g=\cO(\varepsilon^2),\qquad
		\Delta_k^H=\cO(\varepsilon^{3/2}),\qquad
		D(X_k)=\cO(\varepsilon^{3/2}),
		\]
		up to constants. Since
		$\Delta_k^g,\Delta_k^H=\cO(k^{-(p-1)})$ by
		Proposition~\ref{prop:dispersion_decay}, this gives the
		conservative burn-in estimate
		\[
		K_0(\varepsilon)=\cO\!\bigl(\varepsilon^{-2/(p-1)}\bigr).
		\]
		If one keeps only $p>2$, the same burn-in estimate gives the
		more general bound
		\[
		K(\varepsilon)
		=\cO\!\left(\varepsilon^{-1}+\varepsilon^{-2/(p-1)}\right).
		\]
		In particular, for $p\ge3$, the burn-in estimate is no larger
		than the $\cO(\varepsilon^{-1})$ post-burn-in term, and hence
		$K(\varepsilon)=\cO(\varepsilon^{-1})$.
		
		The dependence on the spectral gap enters through the mixing schedule:
		$\tau_k=\lceil (p\log(k+2)+c_{\mathrm{mix}})/(-\log\rho)\rceil$
		implies that each
		outer iteration uses $\cO\bigl((1-\rho)^{-1}\log k\bigr)$
		communication rounds; the constant $c_{\mathrm{mix}}$ affects only
		the constant. Here $-1/\log\rho\approx(1-\rho)^{-1}$ for $\rho$ near~1.
		The total communication-round budget is therefore
		\[
		\cO\bigl((1-\rho)^{-1}\,\varepsilon^{-1}\log(1/\varepsilon)\bigr),
		\]
		exhibiting the same linear dependence on the inverse spectral gap
		as repeated gossip mixing.
		On very sparse graphs ($\rho\to 1$), the per-iteration communication
		overhead grows, but the iteration count remains network-independent.
	\end{remark}

	\begin{remark}[$\varepsilon$-independence of the algorithm]\label{rem:eps-independence}
		The algorithm is not restarted or retuned for a prescribed~$\varepsilon$;
		the logarithmic mixing schedule~\eqref{eq:log_schedule_main} is fixed
		independently of the target accuracy.  The index $K_0(\varepsilon)$ enters
		only in the complexity proof as the first index after which the scheduled
		consensus errors fall below accuracy-dependent thresholds.
	\end{remark}

	\begin{remark}[Dimension dependence of the constants]\label{rem:dim-constants}
		The constants $C_K$ and $K_0(\varepsilon)$ in
		Theorem~\ref{thm:main} depend on $d$, $N$, $L_1$, $L_2$, $D$,
		$M$, and the spectral gap $1-\rho$.  In particular, the
		Frobenius--spectral norm conversion in the dispersion bounds
		introduces a factor $\sqrt{d}$; hence the iteration bound, while
		independent of~$\varepsilon$ up to $\cO(\varepsilon^{-1})$, is not
		dimension-free.
	\end{remark}

	\begin{remark}\label{rem:two_stages}
		The proof of Theorem~\ref{thm:main} combines two independent
		analyses, distinct from the per-iteration reduction chain in
		Sections~\ref{ssec:avg}--\ref{ssec:descent} (which is used in Stage~2).
		In Stage~1 (dispersion decay;
		Proposition~\ref{prop:dispersion_decay},
		Appendix~\ref{app:dispersion}), the logarithmic schedule
		makes tracker dispersions decay polynomially on the bounded
		trajectory of Assumption~\ref{ass:bounded_iterates}.
		In Stage~2 (descent), once the consensus errors meet the
		burn-in requirements
		(after $K_0(\varepsilon)$ iterations), the per-iteration chain
		verifies an inexact regularized Newton recursion
		(Proposition~\ref{prop:inexact-rn}), and the $3/2$-recursion
		(Lemma~\ref{lem:solve32}) gives $\cO(\varepsilon^{-1})$ further
		iterations.
		The two stages are self-contained and do not depend on each other's conclusions.
	\end{remark}
	
	\subsection{Strong convexity and local superlinear convergence}\label{ssec:superlinear}
	The preceding results require only convexity.
	When the objective is locally strongly convex near the optimum,
	the stabilizer term $\bar\delta_k$ becomes negligible relative to
	$\|g_k\|$, recovering the exact Newton regime and, with it, local
	$Q$-superlinear convergence.
	
	\begin{assumption}\label{ass:strong}
		The average objective $f$ is $\mu$-strongly convex on the bounded set
		$\{x:\|x-x_\star\|\le D\}$ (where $D$ is given by
		Assumption~\ref{ass:bounded_iterates}):
		$\nabla^2 f(x)\succeq \mu I$ for all $\|x-x_\star\|\le D$.
	\end{assumption}
	
	\begin{lemma}\label{lem:delta_o_g}
		If Assumptions~\ref{ass:problem}, \ref{ass:bounded_iterates},
		and~\ref{ass:strong} hold, then the following is true.
		Fix any exponent $\gamma>0$.
		Let
		\[
		G_g:=\sup_{k\ge 0}\ \max_{1\le i\le N}\ \|\tilde g_{i,k}\|\in[0,\infty),
		\]
		which is finite by Lemma~\ref{lem:tilde_track_bounded}.
		Suppose that for all $k$ large enough (and $k\ge 1$) the mixing depths satisfy
		\begin{equation}\label{eq:delta_schedule}
			\max\{\rho^{\tau_k},\rho^{t_{k-1}}\}\le \tfrac{1}{C_\delta}\|g_k\|^{1+\gamma},
		\end{equation}
		where $C_\delta:=\hat C_H+L_2(2D+\sqrt{G_g/M})$ and $\hat C_H:=B_H+2\sqrt{d}\,L_2D$,
		where $B_H:=\sup_{k\ge0}D(\tilde{\mathcal H}_k)<\infty$
		is the uniform Hessian-tracker dispersion bound from Lemma~\ref{lem:tilde_hess_bounded}.
		Then, for all sufficiently large $k$,
		\[
		\bar\delta_k = o(\|g_k\|)\qquad\text{and more precisely}\qquad \bar\delta_k\le \|g_k\|^{1+\gamma}.
		\]
	\end{lemma}
	
	\begin{theorem}\label{thm:superlinear}
		Under the hypotheses of Theorem~\ref{thm:main}, suppose in addition
		that Assumption~\ref{ass:strong} holds,
		Proposition~\ref{prop:inexact-rn} holds for all sufficiently large $k$ with some $\eta\in(0,1/12]$,
		and there exists a post-burn-in tail on which $\|g_k\|\to0$ and
		a constant $\gamma>0$ such that, for all sufficiently large $k$,
		\begin{equation}\label{eq:local_accuracy}
			\max\{\rho^{\tau_k},\rho^{t_{k-1}}\}\le
			\tfrac{1}{C_\delta}\|g_k\|^{1+\gamma},
			\qquad
			\Delta_k^g\le C_g\|g_k\|^{1+\gamma}
		\end{equation}
		for some constant $C_g>0$,
		then there exist constants $C_{\mathrm{sc}}>0$ and $k_{\mathrm{sc}}$ such that, for all $k\ge k_{\mathrm{sc}}$,
		\[
		\|\nabla f(\bar x_{k+1})\|\le C_{\mathrm{sc}}\,\|\nabla f(\bar x_k)\|^{3/2}.
		\]
	\end{theorem}
	
	\begin{remark}
		The logarithmic schedule in Theorem~\ref{thm:main} is sufficient for
		the global $\cO(\varepsilon^{-1})$ rate. The local superlinear result
		uses the relative accuracy condition~\eqref{eq:local_accuracy}, which
		requires the consensus and tracking errors to decay relative to the
		current gradient norm along the local tail.
	\end{remark}
	

	
	\section{Numerical experiments}\label{sec:experiments}
	
	This section evaluates the practical performance of the
	\textsc{DisGrem} family using the logarithmic communication rule
	motivated by the analysis.
	All main convergence experiments report statistics over
	20 independent Monte~Carlo (MC) trials with
	randomized starting points and random Erd\H{o}s--R\'{e}nyi graphs.
	Four \textsc{DisGrem} variants are compared against six first- and
	second-order baselines on nine objectives spanning well-conditioned
	and ill-conditioned convex problems, real-data logistic regression,
	and four nonconvex objectives.
	
	\subsection{Experimental setup}\label{ssec:setup}
	
	We use $N{=}10$ agents communicating over an Erd\H{o}s--R\'{e}nyi (ER) random graph
	with edge probability $p_{\mathrm{er}}{=}0.5$ (regenerated per trial).
	The Metropolis--Hastings doubly stochastic mixing matrix $W$ yields
	spectral gap $1{-}\rho\approx 0.08$
	($\rho\approx 0.92$; typical range $\rho\in[0.88,0.96]$ across 20 trials),
	representing a moderately sparse topology.
	The analysis uses the same logarithmic depth for all tracked quantities.
	In the experiments, the vector mixing depths in
	the \textsc{DisGrem} family are chosen as
	\[
	\tau_k=t_k=
	\min\left\{10,\,
	\Bigl\lceil \frac{3\log(k+2)+2}{-\log\rho}\Bigr\rceil
	\right\},
	\]
	which is the logarithmic rule~\eqref{eq:log_schedule_main} with
	$p=3$ and $c_{\mathrm{mix}}=2$, implemented with a maximum depth of
	10 over the tested tolerance range.
	This choice matches the smallest exponent covered by the
	$\cO(\varepsilon^{-1})$ total-complexity statement in
	Theorem~\ref{thm:main}; the maximum depth is sufficient for all
	tolerances reported below.
	For the full-matrix variants, Hessian pre-mixing is limited to three
	matrix rounds. For the communication-efficient variants, Hessian-matrix
	mixing is limited to two matrix rounds while vector mixing uses the
	full $\tau_k$ and $t_k$ depths. These matrix-round limits reduce
	payload while preserving the vector-mixing schedule used by the
	analysis.
	Baselines retain their standard consensus settings, and all comparisons
	report cumulative communication cost so that different per-iteration
	communication patterns are accounted for explicitly.
	
	We compare ten methods in three groups.
	The proposed \textsc{DisGrem} family comprises four variants:
	\textsc{DisGrem} (Algorithm~\ref{alg:disgrem}),
	\textsc{CeDisGrem} (communication-efficient Hessian tracking using
	Top-$k$ sparsification with a 10\% nominal budget, with low-rank
	alternatives studied in Appendix~\ref{app:supp_exp}),
	\textsc{AdaDisGrem} (adaptive $M$ via secant-based Lipschitz estimation),
	and \textsc{CeAdaDisGrem} (adaptive + compressed).
	In the Ce variants, the lazy Hessian-update period and compression
	budget are adjusted with the current communication depth so that
	communication is reduced early while more Hessian information is
	retained as the requested consensus depth increases.
	The regularization scaling is $M = M_{\mathrm{fac}}\cdot H_{\max}^{0}$ where
	$H_{\max}^{0} = \max_i \|\nabla^2 f_i(x_0)\|_2$;
	per-function $M_{\mathrm{fac}}$ values are listed in Table~\ref{tab:config}.
	The adaptive variant \textsc{AdaDisGrem} replaces this fixed scaling
	by online secant-based updates (Section~\ref{ssec:ada}).
	
	Two first-order baselines are included:
	EXTRA~\citep{shi2015extra} and
	DIGing~\citep{nedic2017diging}.
	Four second-order baselines complete the comparison:
	DQM~\citep{eisen2017dqn},
	ESOM~\citep{mokhtari2016esom},
	SONATA~\citep{sun2022sonata}, and
	Network-GIANT~\citep{maritan2023networkgiant}.
	First-order baselines use a Lipschitz-scaled stepsize
	$\alpha=\alpha_{\mathrm{base}}/H_{\max}^{0}$;
	for each function, $\alpha_{\mathrm{base}}$ is selected from
	$\{0.01, 0.1, 0.5, 1.0\}$ as the value giving the fastest
	convergence without divergence over 5 preliminary runs
	(Table~\ref{tab:config}).
	Second-order baselines (DQM, ESOM, SONATA, Network-GIANT) use their
	published parameters; SONATA and Network-GIANT solve local
	subproblems to machine precision via a direct solver.
	The baselines use their standard consensus settings.
	Because the \textsc{DisGrem} family has a three-stage gossip
	structure and carries Hessian payloads, its per-iteration byte cost
	can exceed that of single-stage baselines; cross-method fairness is
	therefore assessed through the cumulative communication cost (MB)
	reported in Figure~\ref{fig:profiles_comm}.
	We use \emph{light, problem-class-level tuning}:
	the per-function $M_{\mathrm{fac}}$ for the proposed family and the
	per-function $\alpha_{\mathrm{base}}$ for first-order baselines are
	both selected from small discrete grids via a handful of preliminary
	runs, while second-order baselines retain their published settings.
	Sensitivity sweeps in Appendix~\ref{app:supp_exp}
	(Figure~\ref{fig:sweep_second}) show that broad changes in the
	second-order baselines' hyperparameters do not remove their stagnation
	on hard instances such as \textsc{LogSumExp}.
	Code to reproduce all experiments is available at
	\url{https://github.com/huwei0121/DisGRem}.
	
	We test on nine objective functions
	(Table~\ref{tab:functions});
	all synthetic functions use $d{=}30$, while the two logistic regression
	problems inherit $d{=}22$ from the \textsc{svmguide3} dataset.
	See Table~\ref{tab:config} for per-function algorithmic parameters.
	Five objectives are convex:
	(i)~\textsc{Ridge} ($\ell_2$-regularized least squares, $\lambda{=}10^{-3}$);
	(ii)~\textsc{QuadBad} (heterogeneous ill-conditioned quadratic, $\kappa{=}10^3$);
	(iii)~\textsc{LogSumExp}
	(smooth approximation of the max function~\citep{nesterov2005smooth},
	$\sigma{=}0.5$);
	(iv)~\textsc{Huber}
	(pseudo-Huber loss~\citep{charbonnier1997huber}, $\delta{=}1$);
	(v)~\textsc{LogReg-real} ($\ell_2$-regularized logistic regression on the
	\textsc{svmguide3} LibSVM dataset~\citep{chang2011libsvm},
	$m{=}1243$, $d{=}22$;
	a standard benchmark in decentralized
	optimization~\citep{shi2015extra,sun2022sonata}).
	To assess behavior beyond the convex setting, we also include four
	nonconvex objectives:
	(vi)~\textsc{LinLog} (piecewise quadratic-logarithmic loss with flat curvature regions);
	(vii)~\textsc{Rosenbrock}~\citep{rosenbrock1960automatic};
	(viii)~\textsc{Styblinski--Tang}~\citep{styblinski1990experiments} (multimodal);
	(ix)~\textsc{LogReg-NCVR} (logistic regression with
	bounded nonconvex penalty $\alpha\sum_k x_k^2/(1{+}x_k^2)$~\citep{geman1995nonlinear}
	on \textsc{svmguide3}).
	On these nonconvex objectives, the eigenvalue-shift stabilizer
	$\delta_{i,k}$ makes the local coefficient matrix positive definite whenever
	$\tilde g_{i,k}\ne0$. When $\tilde g_{i,k}=0$, Algorithm~\ref{alg:disgrem}
	uses the convention $s_{i,k}=0$. Thus the algorithm remains well posed and operates as a
	damped Newton-like method; however, the convergence
	guarantees of Theorems~\ref{thm:main}--\ref{thm:superlinear} do
	not apply.
	The shift $\delta_{i,k}$ plays two structurally similar but
	theoretically distinct roles:
	in the convex analysis it compensates for transient indefiniteness of
	tracked Hessian matrices (which are PSD in the exact case);
	on nonconvex objectives it additionally absorbs true negative
	curvature, acting as a Levenberg--Marquardt damping term.
	For convex objectives, $f_{\mathrm{ref}}:=f^\star$ (the global minimum,
	whose existence is guaranteed by coercivity).
	For nonconvex objectives, $f_{\mathrm{ref}}$ denotes the best value
	found by multi-start L-BFGS-B (50 restarts, tolerance $10^{-15}$);
	it is not a certified global optimum.
	
	\begin{table}[htbp]
		\centering
		\caption{Test function definitions.
			Each $f_i\colon\R^d\to\R$ is the local objective of agent~$i$.
			For synthetic problems (Ridge--LinLog), data $(A_i, b_i)$
			are independently generated per agent with i.i.d.\ Gaussian entries;
			for logistic regression, the \textsc{svmguide3}
			dataset~\citep{chang2011libsvm}
			($m{=}1243$, $d{=}22$) is randomly partitioned across agents.
			Rosenbrock and Styblinski--Tang are homogeneous
			(all agents share the same $f_i$).}
		\label{tab:functions}
		\setlength{\tabcolsep}{3pt}
		\footnotesize
		\begin{tabular}{@{}llp{4.5cm}@{}}
			\toprule
			Name & Local objective $f_i(x)$ & Data / parameters \\
			\midrule
			\multicolumn{3}{@{}l}{\textit{Convex}} \\[2pt]
			Ridge
			& $\frac{1}{2}\|A_i x{-}y_i\|^2{+}\frac{\lambda}{2}\|x\|^2$
			& $A_i{\in}\R^{150\times d}$,
			$y_i{=}A_i x_{\mathrm{true}}{+}0.05\varepsilon_i$,
			$\lambda{=}10^{-3}$ \\[4pt]
			QuadBad
			& $\frac{1}{2}x^\top Q_i x + b_i^\top x$
			& $Q_i$ diagonal, eigenvalues log-spaced
			in $[1,\chi_i]$,
			$\chi_i{\approx}\kappa$;
			$b_i{\sim}\mathcal{N}(0,I)$ \\[4pt]
			LogSumExp
			& $\sigma\log\!\bigl(\sum_{j=1}^{p}
			e^{(A_i^\top x-b_i)_j/\sigma}\bigr)$
			& $A_i{\in}\R^{d\times p}$,
			$p{=}\max(d{+}2,12)$,
			$\sigma{=}0.5$ \\[4pt]
			Huber
			& $\sum_{j=1}^{p}\delta^2
			\bigl(\sqrt{1{+}(r_j/\delta)^2}{-}1\bigr)$,
			$r{=}A_ix{-}b_i$
			& $A_i{\in}\R^{5\times d}$,
			$\delta{=}1$ \\[4pt]
			LogReg-real
			& $\frac{\iota}{2}\|x\|^2
			{+}\frac{1}{m_i}\sum_{j=1}^{m_i}
			\log(1{+}e^{-b_j a_j^\top x})$
			& \textsc{svmguide3} data split;
			$\iota{=}10^{-2}$ \\[4pt]
			\midrule
			\multicolumn{3}{@{}l}{\textit{Nonconvex}} \\[2pt]
			LinLog
			& $\sum_{j=1}^{d}\ell\bigl((A_ix{-}b_i)_j\bigr)$;
			see ${}^\dagger$ below
			& $A_i{\in}\R^{d\times d}$ \\[4pt]
			Rosenbrock
			& $\sum_{j=1}^{d/2}\bigl[100(x_{2j}
			{-}x_{2j-1}^2)^2{+}(x_{2j-1}{-}1)^2\bigr]$
			& shared ($f_i{=}f$\;\;$\forall\,i$) \\[4pt]
			Styblinski
			& $\sum_{j=1}^{d}
			(x_j^4{-}16x_j^2{+}5x_j)$
			& shared ($f_i{=}f$\;\;$\forall\,i$) \\[4pt]
			LogReg-NCVR
			& $\frac{1}{m_i}\sum_{j=1}^{m_i}
			\log(1{+}e^{-b_ja_j^\top x})
			{+}\alpha\sum_{k=1}^{d}\frac{x_k^2}{1{+}x_k^2}$
			& \textsc{svmguide3} data split;
			$\alpha{=}0.05$ \\
			\bottomrule
		\end{tabular}
		\\[4pt]
		{\footnotesize ${}^\dagger$LinLog:
			$\ell(r)=r^2/2$ for $|r|\le 1$;\;
			$\ell(r)=\ln|r|+\tfrac{1}{2}$ for $|r|>1$.}
	\end{table}
	
	\begin{table}[htbp]
		\centering
		\caption{Per-function experimental configuration.
			$M_{\mathrm{fac}}$: regularization scaling for \textsc{DisGrem}
			($M = M_{\mathrm{fac}}\cdot H_{\max}^{0}$);
			$\alpha_{\mathrm{base}}$: stepsize coefficient for first-order
			baselines (actual step $\alpha_{\mathrm{base}}/H_{\max}^{0}$);
			$K_{\max}$: iteration budget (shared by all algorithms);
			Decay: whether a $1/\sqrt{k}$ stepsize decay is applied to
			first-order baselines (relevant only for nonconvex problems).}
		\label{tab:config}
		\setlength{\tabcolsep}{5pt}
		\small
		\begin{tabular}{lccccl}
			\toprule
			Function & $M_{\mathrm{fac}}$ & $\alpha_{\mathrm{base}}$ & $K_{\max}$ & Decay & Notes \\
			\midrule
			Ridge       & 0.1  & 0.20 &  200 & No  & $\lambda{=}10^{-3}$ \\
			QuadBad     & 0.1  & 0.10 & 1500 & No  & $\kappa{=}10^3$ \\
			LogSumExp   & 5.0  & 0.30 &  400 & No  & \\
			Huber       & 1.5  & 0.30 &  800 & No  & Pseudo-Huber, $\delta{=}1$ \\
			LogReg-real & 3.0  & 1.00 &  600 & No  & svmguide3, $d{=}22$ \\
			LinLog      & 1.0  & 0.20 & 1500 & No  & Nonconvex \\
			Rosenbrock  & 3.0  & 0.10 &  300 & Yes & Nonconvex \\
			Styblinski  & 15.0 & 0.05 &  100 & Yes & Multimodal \\
			LogReg-NCVR & 3.0  & 1.00 & 1000 & Yes & svmguide3, $d{=}22$ \\
			\bottomrule
		\end{tabular}
	\end{table}
	
	We define the consensus residual
	\[
	\mathrm{cons}_k
	\;:=\;
	D(X_k)
	\;=\;
	\sqrt{\tfrac{1}{N}\textstyle\sum_{i=1}^{N}\|x_{i,k}-\bar x_k\|^{2}},
	\]
	and the combo stopping criterion
	\[
	\mathrm{combo}_k
	\;:=\;
	\|\nabla f(\bar x_k)\| + \mathrm{cons}_k,
	\]
	where $\nabla f(\bar x_k)=\frac{1}{N}\sum_{i=1}^N\nabla f_i(\bar x_k)$
	is the true gradient at the average iterate (computed offline,
	not from the tracker), ensuring a fair comparison across algorithms
	with different internal tracking mechanisms.
	Each algorithm runs to $K_{\max}$ iterations or until
	$\mathrm{combo}_k < 10^{-12}$.
	We report four metrics:
	(a)~$\mathrm{combo}=\min_k\,\mathrm{combo}_k$;
	(b)~$\mathrm{relF}=\min_k |f(\bar x_k)-f_{\mathrm{ref}}|/|f(\bar x_0)-f_{\mathrm{ref}}|$;
	(c)~wall-clock time (seconds; for reference only---all implementations
	use Python/NumPy with comparable vectorization);
	(d)~cumulative communication cost (MB), defined by summing the
	stage-wise directed neighbor-message payloads over all outer iterations
	(consistent with the accounting in Section~\ref{ssec:comm}).
	All convergence curves show the 20-run median with interquartile
	shading, after applying the standard running-best envelope to monotone
	accuracy metrics.
	Analytical gradients and Hessians are used for all functions.
	A run is deemed successful at threshold $\varepsilon$ if it reaches
	$\mathrm{relF}\le\varepsilon$ within $K_{\max}$ iterations without
	encountering NaN or overflow; for nonconvex problems, $\mathrm{relF}$ is
	computed relative to $f_{\mathrm{ref}}$.
	Unless otherwise noted, the Hessian is updated every iteration
	($K_{\mathrm{lazy}}{=}1$); Section~\ref{ssec:commexp} studies the
	effect of lazy updates.
	
	The theoretical results correspond to the logarithmic rule without a
	maximum-depth restriction; the reported experiments use the displayed
	maximum depth because the
	tolerances in the figures and tables are reached within that range.
	
	\FloatBarrier
	\subsection{Convergence benchmarks}\label{ssec:convergence}
	
	
	\begin{figure}[htbp]
		\centering
		\includegraphics[width=\textwidth]{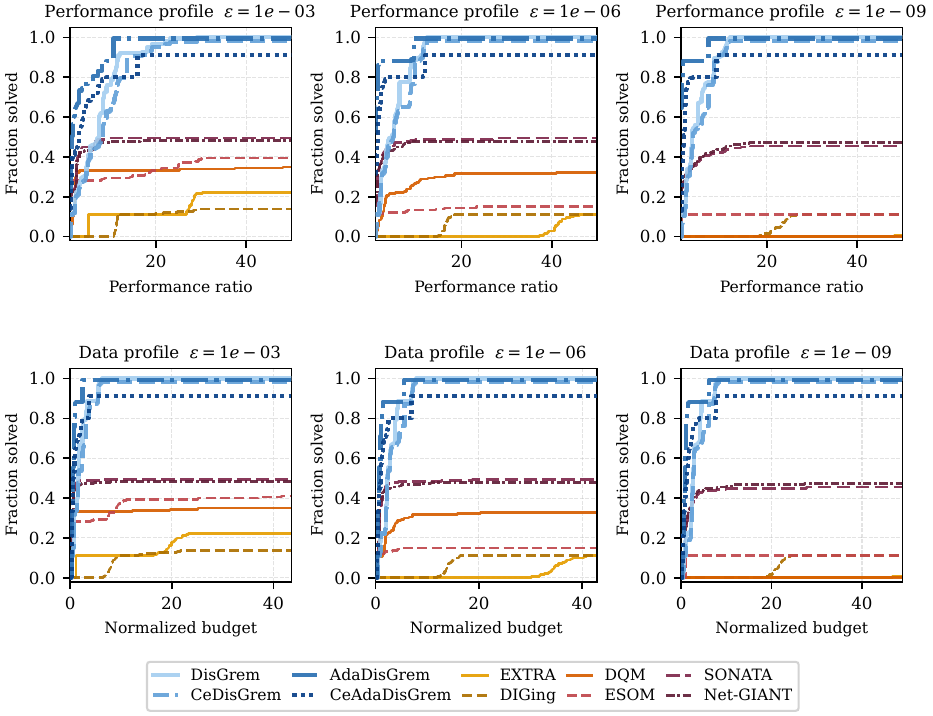}
		\caption{Iteration-budget profiles at three precision levels
			$\varepsilon\in\{10^{-3},10^{-6},10^{-9}\}$.
			Top: performance profiles; bottom: data profiles.
			A solver's curve at ordinate~$\pi$ indicates it solves a
			$\pi$-fraction of the problem--MC instances within the given budget.}
		\label{fig:profiles}
	\end{figure}
	
	\begin{figure}[htbp]
		\centering
		\includegraphics[width=\textwidth]{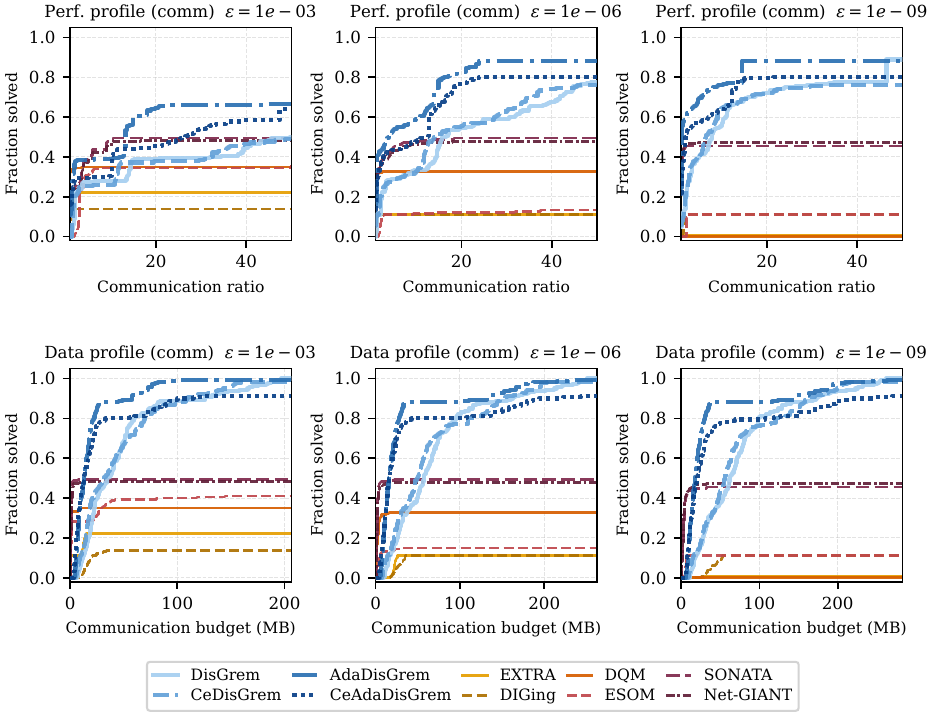}
		\caption{Communication-budget profiles at three precision levels
			$\varepsilon\in\{10^{-3},10^{-6},10^{-9}\}$.
			Top: performance profiles with cumulative communication cost (MB)
			as the ratio axis.
			Bottom: data profiles with communication cost on the $\kappa$-axis.}
		\label{fig:profiles_comm}
	\end{figure}
	The main empirical finding is that the \textsc{DisGrem} family is
	consistently accurate across the full benchmark suite:
	it attains high accuracy on all nine test problems, whereas the
	baselines exhibit stagnation, divergence, or strong problem dependence
	on at least one instance
	(Figures~\ref{fig:profiles}--\ref{fig:profiles_comm}).
	
	Table~\ref{tab:summary} reports key numerical results on four
	representative functions covering both favorable and challenging
	cases; per-problem details follow.
	
	\begin{table}[htbp]
		\centering
		\caption{Summary on four representative functions
			($N{=}10$, 20 MC runs; synthetic functions use $d{=}30$,
			and \textsc{LogReg-real} uses $d{=}22$).
			Column keys: Time (s), $K$ (iterations), min(relF), Comm (MB).}
		\label{tab:summary}
		\setlength{\tabcolsep}{4pt}
		\small
		\resizebox{\textwidth}{!}{%
			\begin{tabular}{l rrrr rrrr rrrr rrrr}
				\toprule
				& \multicolumn{4}{c}{\textsc{Huber} ($K_{\max}{=}800$)}
				& \multicolumn{4}{c}{\textsc{LogSumExp} ($K_{\max}{=}400$)}
				& \multicolumn{4}{c}{\textsc{LinLog} ($K_{\max}{=}1500$)}
				& \multicolumn{4}{c}{\textsc{Styblinski--Tang} ($K_{\max}{=}100$)} \\
				\cmidrule(lr){2-5}\cmidrule(lr){6-9}\cmidrule(lr){10-13}\cmidrule(lr){14-17}
				Algorithm & Time & $K$ & min(relF) & MB
				& Time & $K$ & min(relF) & MB
				& Time & $K$ & min(relF) & MB
				& Time & $K$ & min(relF) & MB \\
				\midrule
				\multicolumn{17}{l}{Proposed (\textsc{DisGrem} family)} \\
				\textsc{DisGrem}      & 0.16 & 85  & $3\mathrm{e}^{-13}$ & 39
				& 0.57 & 259 & $1\mathrm{e}^{-13}$ & 121
				& 0.08 & 41  & $2\mathrm{e}^{-13}$ & 19
				& 0.14 & 94  & $7\mathrm{e}^{-16}$ & 44 \\
				\textsc{CeDisGrem}    & 1.13 & 800 & $2\mathrm{e}^{-12}$ & 260
				& 0.43 & 259 & $1\mathrm{e}^{-13}$ & 98
				& 0.23 & 140 & $1\mathrm{e}^{-12}$ & 57
				& 0.10 & 93  & $1\mathrm{e}^{-13}$ & 37 \\
				\textsc{AdaDisGrem}   & 1.33 & 800 & $4\mathrm{e}^{-12}$ & 426
				& 0.07 & 33  & $9\mathrm{e}^{-13}$ & 15
				& 0.08 & 42  & $6\mathrm{e}^{-13}$ & 19
				& 0.02 & 17  & $5\mathrm{e}^{-16}$ & 8 \\
				\textsc{CeAdaDisGrem} & 1.21 & 800 & $5\mathrm{e}^{-10}$ & 313
				& 0.07 & 35  & $3\mathrm{e}^{-13}$ & 13
				& 0.26 & 156 & $6\mathrm{e}^{-13}$ & 64
				& 0.02 & 18  & $5\mathrm{e}^{-16}$ & 7 \\
				\midrule
				\multicolumn{17}{l}{First-order baselines} \\
				EXTRA  & 5.9 & 800 & $5\mathrm{e}^{-3}$  & 8
				& 3.5 & 400 & $4\mathrm{e}^{-1}$  & 4
				& 16.9 & 1500 & $7\mathrm{e}^{-9}$  & 15
				& 0.70 & 100 & $7\mathrm{e}^{-1}$  & 1 \\
				DIGing & 6.0 & 800 & $3\mathrm{e}^{-4}$  & 25
				& 3.3 & 400 & $2\mathrm{e}^{-1}$  & 12
				& 10.3 & 989 & $1\mathrm{e}^{-12}$ & 31
				& 0.72 & 100 & $9\mathrm{e}^{-1}$  & 3 \\
				\midrule
				\multicolumn{17}{l}{Second-order baselines} \\
				DQM    & 0.8 & 800 & $2\mathrm{e}^{-1}$  & 12
				& 0.5 & 400 & $3\mathrm{e}^{-3}$  & 6
				& 1.7 & 1500 & $2\mathrm{e}^{0}$   & 23
				& 0.09 & 100 & $2\mathrm{e}^{-1}$  & 2 \\
				ESOM   & 1.4 & 800 & $1\mathrm{e}^{-1}$  & 74
				& 0.7 & 400 & $2\mathrm{e}^{-3}$  & 37
				& 3.5 & 1500 & $1\mathrm{e}^{1}$   & 139
				& 0.19 & 100 & $1\mathrm{e}^{1}$   & 9 \\
				SONATA & 1.0 & 800 & $1\mathrm{e}^{-1}$  & 37
				& 0.5 & 400 & $3\mathrm{e}^{-2}$  & 19
				& 2.3 & 1500 & $1\mathrm{e}^{0}$   & 69
				& 0.03 & 12  & $5\mathrm{e}^{-16}$ & 1 \\
				Net-GIANT & 1.0 & 800 & $4\mathrm{e}^{-2}$ & 37
				& 0.6 & 400 & $1\mathrm{e}^{-2}$  & 19
				& 2.5 & 1500 & $1\mathrm{e}^{0}$   & 69
				& 0.01 & 12  & $5\mathrm{e}^{-16}$ & 1 \\
				\bottomrule
		\end{tabular}}
	\end{table}
	
	Figure~\ref{fig:relF_steps_all} shows $\mathrm{relF}$ versus iteration for
	all nine benchmark functions; additional metrics (combo, wall-clock
	time, communication cost) are collected in
	Appendix~\ref{app:supp_exp},
	Figures~\ref{fig:app_combo_steps}--\ref{fig:app_relF_comm}.
	
	\begin{figure}[htbp]
		\centering
		\includegraphics[width=\textwidth]{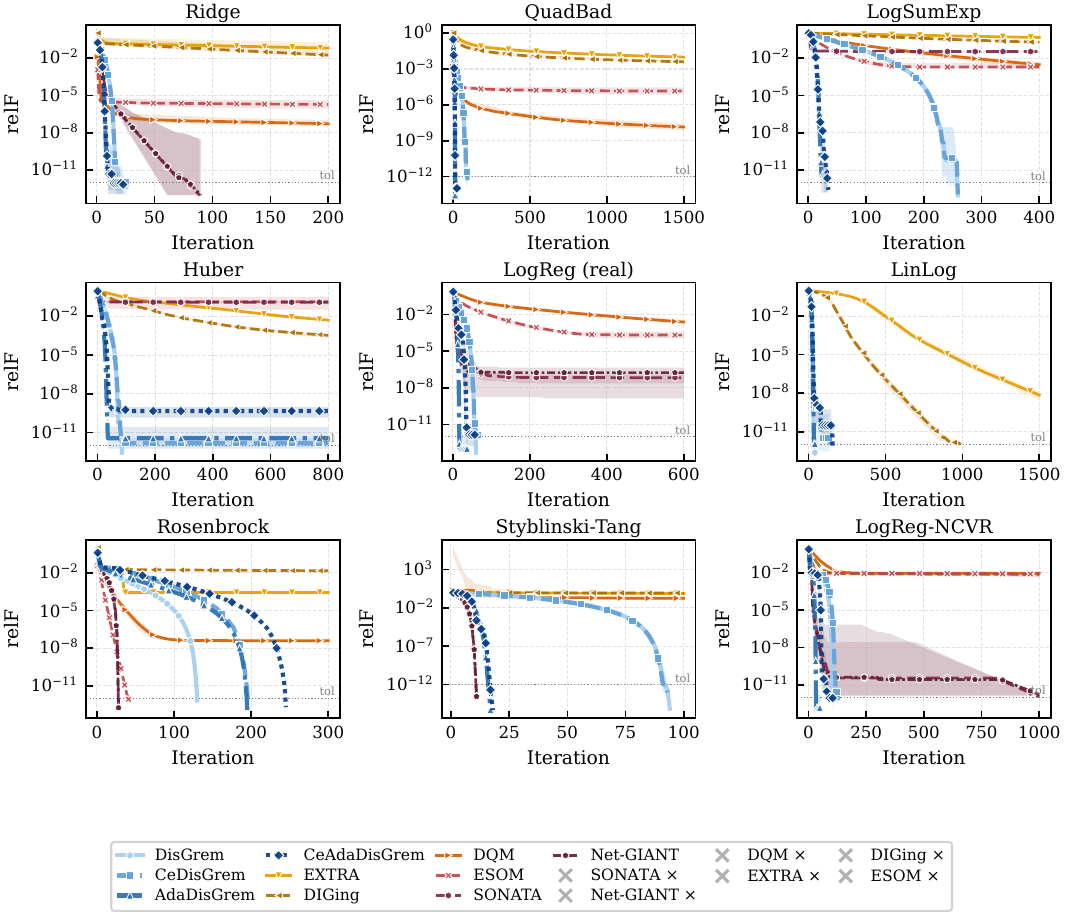}
		\caption{$\mathrm{relF}$ vs.\ iteration on all nine test functions
			($N{=}10$, 20 MC runs; synthetic objectives use $d{=}30$,
			whereas the two logistic-regression objectives use the
			$d{=}22$ \textsc{svmguide3} feature dimension; logarithmic
			communication implementation).
			Solid curves: \textsc{DisGrem} family; dashed: baselines.
			Shaded bands: interquartile range.
			Legend entries marked by $\times$ indicate algorithms whose
			median curve is not drawn because the corresponding runs
			terminated by NaN, overflow, or divergence.}
		\label{fig:relF_steps_all}
	\end{figure}
	
	We now detail the per-problem behavior.
	On the harder convex problems the accuracy gap is significant.
	On LogSumExp, the \textsc{DisGrem} family reaches
	$\mathrm{relF}\lesssim10^{-12}$, whereas every baseline
	stagnates above $10^{-3}$ (first-order) or $10^{-2}$
	(SONATA, Net-GIANT).
	On Huber the proposed methods attain $10^{-12}$--$10^{-10}$, well
	below the $>10^{-1}$ plateau of DQM/ESOM/SONATA.
	On ill-conditioned problems (QuadBad, $\kappa{=}10^3$), the $\mathrm{relF}$
	curve exhibits a clear two-phase pattern: an initial plateau of
	$\sim$50--100 iterations during which the tracker dispersions
	contract but the objective barely decreases, followed by a rapid
	descent phase consistent with the $\cO(1/k^2)$ theoretical rate.
	This matches the predicted burn-in/descent structure of
	Theorem~\ref{thm:main}.
	
	We note two exceptions where individual baselines outperform.
	DIGing on LinLog reaches $\mathrm{relF}\approx10^{-12}$, comparable
	to the \textsc{DisGrem} family:
	LinLog exhibits very flat curvature regions where
	the Hessian-based regularization overestimates curvature,
	whereas DIGing's simpler dynamics are less affected.
	However, DIGing is much slower in wall-clock time
	and fails on the majority of other functions.
	SONATA and Net-GIANT reach $\sim10^{-16}$ on
	Styblinski--Tang in 12 steps, but stagnate or diverge on
	Huber and LinLog ($\mathrm{relF}>1$).
	
	Regarding the adaptive versus fixed-$M$ trade-off,
	\textsc{AdaDisGrem} is fastest on LogSumExp
	(33 versus 259 iterations for fixed-$M$ \textsc{DisGrem}) and
	Styblinski--Tang (17 versus 94 iterations), where the secant estimate
	quickly reduces the effective regularization scale.
	In summary, adaptation reduces the need for manually selecting a fixed $M$
	(Section~\ref{ssec:ada}) and often reduces the iteration count,
	although fixed $M$ may still use less communication on some problems.
	Section~\ref{ssec:commexp} compares total communication
	volume to a target precision, the more informative metric
	given the $\cO(d^2)$ per-iteration Hessian payload of \textsc{DisGrem}.
	
	Theorem~\ref{thm:main} assumes convexity, yet the \textsc{DisGrem}
	family converges on all four nonconvex benchmarks in our tests:
	all four proposed variants reach $\mathrm{relF}\lesssim10^{-12}$ on
	LinLog; all four variants achieve
	$\mathrm{relF}<10^{-12}$ on Rosenbrock within 300 iterations;
	\textsc{AdaDisGrem} attains $\sim10^{-16}$ on Styblinski--Tang in
	17 steps; and the family converges to $\sim10^{-12}$ on
	LogReg-NCVR within 128 steps for the fixed-$M$ pair and within
	114 steps for \textsc{CeAdaDisGrem}.
	These tests indicate that the regularization and eigenvalue shift provide
	useful damping beyond the convex regime analyzed in the theory.
	
	\FloatBarrier
	\subsection{Communication cost}\label{ssec:commexp}
	
	The profiles in Section~\ref{ssec:convergence} use iteration count
	as the budget axis; here we take a communication-volume perspective.
	We focus on four representative functions (Ridge, LogSumExp, Huber,
	and LogReg-real) with 5 MC runs; the synthetic problems use $d=30$,
	while LogReg-real uses the \textsc{svmguide3} dimension $d=22$.
	
	\subsubsection{Benefit of communication-efficient (Ce) variants}
	
	\begin{figure}[htbp]
		\centering
		\includegraphics[width=\textwidth]{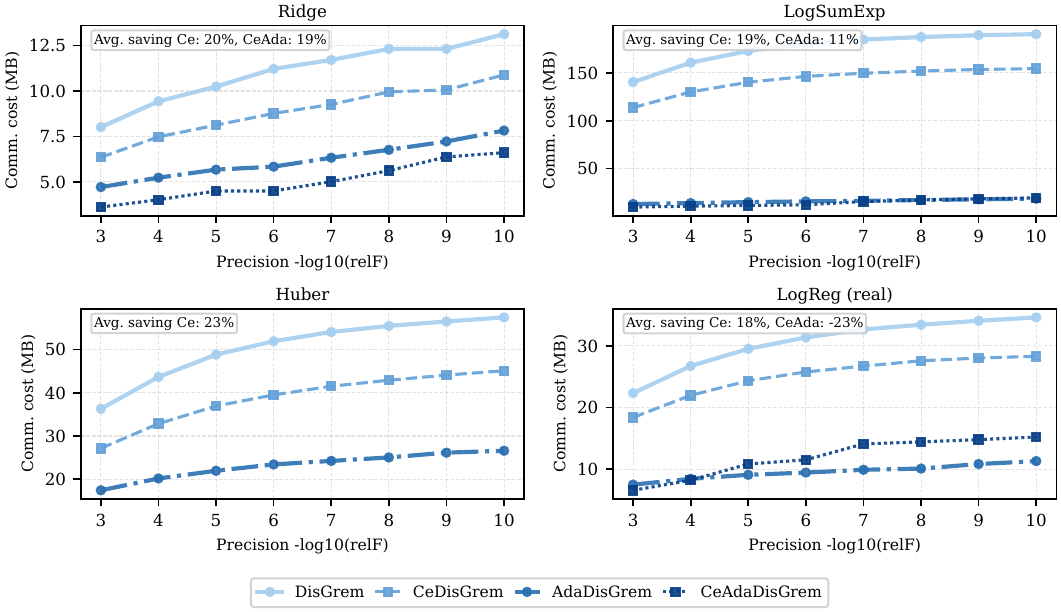}
		\caption{Communication cost (MB) to reach target precision
			$\varepsilon\in\{10^{-3},\ldots,10^{-10}\}$ for full vs.\ Ce variants.
			Missing markers: target not reached.}
		\label{fig:ce_benefit}
	\end{figure}
	
	\begin{table}[htbp]
		\centering
		\caption{Communication savings (\%) of Ce variants over their full
			counterparts at selected precision levels.
			Positive values indicate savings; negative values indicate that
			Ce requires more communication (due to slower convergence
			erasing per-iteration byte savings).}
		\label{tab:ce_savings}
		\setlength{\tabcolsep}{5pt}
		\small
		\begin{tabular}{llrrrr}
			\toprule
			Function & Pair & $\varepsilon{=}10^{-3}$ & $\varepsilon{=}10^{-6}$ & $\varepsilon{=}10^{-8}$ & $\varepsilon{=}10^{-10}$ \\
			\midrule
			\multirow{2}{*}{Ridge}
			& \textsc{DisGrem} $\to$ Ce       & $+21\%$ & $+22\%$  & $+19\%$ & $+17\%$ \\
			& \textsc{AdaDisGrem} $\to$ Ce    & $+23\%$ & $+23\%$  & $+17\%$ & $+16\%$ \\
			\midrule
			\multirow{2}{*}{LogSumExp}
			& \textsc{DisGrem} $\to$ Ce       & $+19\%$ & $+19\%$  & $+19\%$ & $+19\%$ \\
			& \textsc{AdaDisGrem} $\to$ Ce    & $+24\%$ & $+22\%$  & $-1\%$  & $-4\%$  \\
			\midrule
			\multirow{2}{*}{Huber}
			& \textsc{DisGrem} $\to$ Ce       & $+25\%$ & $+24\%$  & $+23\%$ & $+22\%$ \\
			& \textsc{AdaDisGrem} $\to$ Ce    & ---     & ---      & ---     & ---     \\
			\midrule
			\multirow{2}{*}{LogReg-real}
			& \textsc{DisGrem} $\to$ Ce       & $+18\%$ & $+18\%$  & $+18\%$ & $+18\%$ \\
			& \textsc{AdaDisGrem} $\to$ Ce    & $+12\%$ & $-22\%$  & $-43\%$ & $-35\%$ \\
			\bottomrule
		\end{tabular}
	\end{table}
	
	Figure~\ref{fig:ce_benefit} and Table~\ref{tab:ce_savings} compare
	the total communication cost needed to reach precision levels $\varepsilon\in\{10^{-3},10^{-6},10^{-8},10^{-10}\}$.
	For the fixed-$M$ pair (\textsc{DisGrem} vs.\ \textsc{CeDisGrem}), the Ce variant
	saves about 18--25\% communication across the four representative
	functions and the tested precision range.
	
	For the adaptive pair (\textsc{AdaDisGrem} vs.\ \textsc{CeAdaDisGrem}),
	the savings are more problem-dependent.
	At low precision, \textsc{CeAdaDisGrem} provides savings on Ridge,
	LogSumExp, and LogReg-real, but at high precision it can require more
	communication on LogReg-real; on Huber it does not reach the selected
	precision levels in this study.
	The fixed-$M$ pair therefore gives the clearest communication benefit,
	while compression can interfere with the adaptive scaling rule at high
	precision.
	
	In summary, the Ce mechanism reliably reduces
	per-iteration payload (by compressing the $\cO(d^2)$
	Hessian exchange), but total communication savings depend on
	the convergence-speed trade-off: with fixed $M$, savings of
	18--25\% are typical in the fixed-$M$ setting; with adaptive $M$,
	compression noise can still slow convergence enough to increase
	total bytes at high precision on some problems
	(Table~\ref{tab:ce_savings}).
	
	\subsubsection{\texorpdfstring{$K_{\mathrm{lazy}}$}{K-lazy} and compression ablation}
	
	Appendix~\ref{app:supp_exp} (Figures~\ref{fig:klazy}--\ref{fig:compress})
	reports a full sweep of $K_{\mathrm{lazy}}\in\{1,5,10,20,40,80\}$
	and compression methods (Top-$k$, Low-Rank).
	Moderate values $K_{\mathrm{lazy}}\in\{5,10\}$ reduce
	communication by 40--60\% on well-conditioned problems with
	negligible precision loss.
	Low-rank compression with $r{=}d/5$ provides a good balance
	across all tested functions; aggressive compression ($r{=}1$ or
	Top-$k$ at 5\%) degrades convergence on Huber and LogSumExp.
	
	\FloatBarrier
	\subsection{Adaptive mechanism}\label{ssec:ada}
	
	We study the adaptive $M$ mechanism of \textsc{AdaDisGrem} on four
	representative functions (Ridge, LogSumExp, LogReg-real, LogReg-NCVR)
	with 5 Monte Carlo runs each.
	
	\begin{figure}[htbp]
		\centering
		\includegraphics[width=\textwidth]{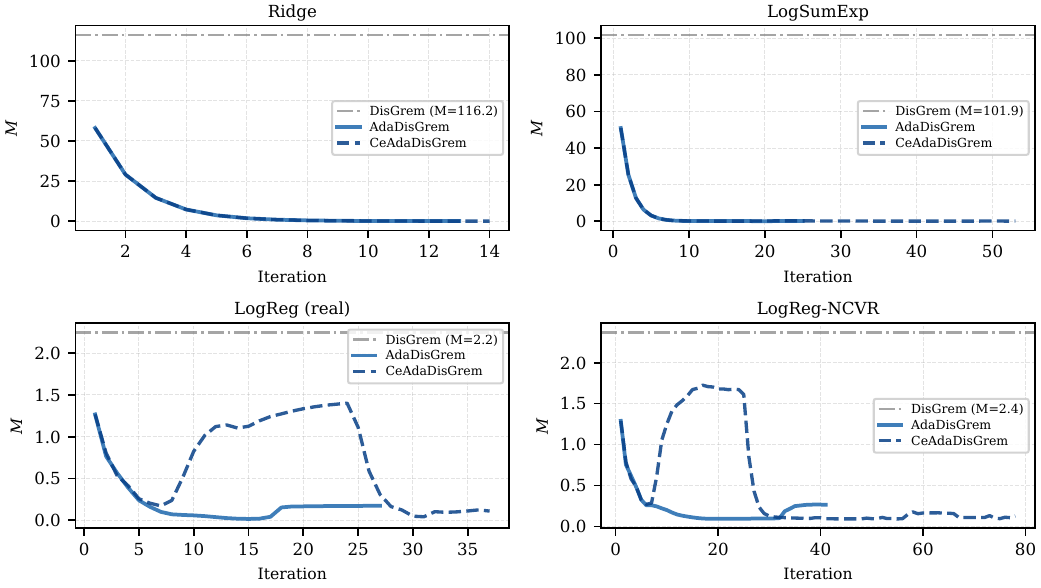}
		\caption{Regularization parameter $M$ trajectory for \textsc{AdaDisGrem}
			(solid) vs.\ the fixed $M$ of \textsc{DisGrem} (dashed).
			Curves averaged over 5 MC runs.}
		\label{fig:ada_trajectory}
	\end{figure}
	
	\begin{figure}[htbp]
		\centering
		\includegraphics[width=\textwidth]{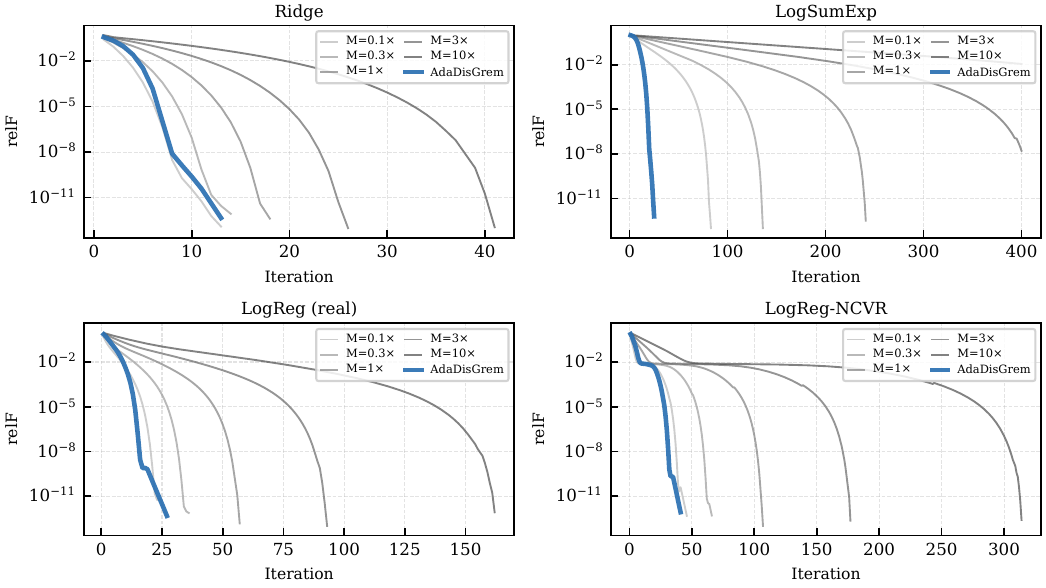}
		\caption{\textsc{AdaDisGrem} (bold blue) vs.\ \textsc{DisGrem}
			at five fixed-$M$ values
			($0.1\times$--$10\times$ the baseline $M^*$; gray curves
			light-to-dark).
			The adaptive variant is competitive with the best fixed
			choice on most functions and avoids failures at aggressive
			fixed-$M$ choices.}
		\label{fig:ada_vs_fixed}
	\end{figure}
	
	Figure~\ref{fig:ada_trajectory} shows the $M$ trajectory.
	The adaptive estimate responds to the local curvature: it starts near
	or above the fixed baseline value, then decays as the iterate approaches
	stationarity, often below the conservative fixed values used for the
	fixed-$M$ method.
	Figure~\ref{fig:ada_vs_fixed} compares \textsc{AdaDisGrem} against
	five manually chosen fixed-$M$ settings spanning a $100\times$ range:
	\textsc{AdaDisGrem} is competitive with the best fixed choice on
	most functions while avoiding failures at aggressive fixed-$M$ choices.
	Furthermore, all initializations in $[0.1M^*,10M^*]$
	converge to a common $M$ trajectory within 50--100 iterations
	(Appendix~\ref{app:supp_exp}, Figure~\ref{fig:ada_init_m}),
	showing that the adaptive dynamics, rather than the initial scale alone,
	drive its empirical performance.
	
	\FloatBarrier
	\subsection{Robustness}\label{ssec:robust}
	
	We investigate robustness along two complementary axes:
	(i)~sensitivity to the starting point, and
	(ii)~sensitivity to the key algorithmic parameter ($M_{\mathrm{fac}}$ for
	the \textsc{DisGrem} family, stepsize $\alpha$ for first-order methods,
	penalty or regularization for second-order baselines).
	All runs use $N{=}10$ and the same communication settings as the
	convergence benchmarks. All synthetic objectives use $d{=}30$,
	whereas the two logistic-regression objectives use the $d{=}22$
	\textsc{svmguide3} feature dimension.
	
	\subsubsection{Starting-point robustness}
	
	Each algorithm is run from 100 independent random initial points
	sampled uniformly on a ball of radius $r$ centered at the reference
	initialization.
	We test two regimes: \textit{near} ($r{=}1$) and \textit{far}
	($r{=}3$).
	A run is counted as successful if it reaches $\mathrm{relF}<10^{-6}$
	within the iteration budget $K_{\max}$.
	(This threshold is deliberately less stringent than the combo~$<10^{-12}$
	criterion used in the convergence benchmarks, meaning that iteration counts
	between the two sections are not directly comparable.)
	Table~\ref{tab:robust} reports
	success rates (\%) across all nine functions and ten algorithms.
	
	\begin{table}[htbp]
		\centering
		\caption{Starting-point robustness: success rate (\%)
			at two radii ($r{=}1$ near, $r{=}3$ far; 100 MC runs).
			Success: $\mathrm{relF}<10^{-6}$.
			Bold: best value in each row; ``---'': 0\%.}
		\label{tab:robust}
		\setlength{\tabcolsep}{3pt}
		\footnotesize
		\begin{tabular}{l rrrr rr rrrr}
			\toprule
			& \multicolumn{4}{c}{\textsc{DisGrem} family}
			& \multicolumn{2}{c}{1st-order}
			& \multicolumn{4}{c}{2nd-order baselines} \\
			\cmidrule(lr){2-5}\cmidrule(lr){6-7}\cmidrule(lr){8-11}
			Function & DG & CeDG & AdaDG & CeAdaDG
			& EXTRA & DIGing & DQM & ESOM & SON & N-GI \\
			\midrule
			\multicolumn{11}{l}{Near initialization ($r{=}1$)} \\
			\midrule
			Ridge       & \textbf{100} & \textbf{100} & \textbf{100} & \textbf{100} & --- & --- & \textbf{100} & 59 & \textbf{100} & \textbf{100} \\
			QuadBad     & \textbf{100} & \textbf{100} & \textbf{100} & \textbf{100} & --- & --- & 98 & --- & 1 & 1 \\
			LogSumExp   & \textbf{100} & \textbf{100} & 98 & \textbf{100} & --- & --- & --- & --- & --- & --- \\
			Huber       & \textbf{100} & 68 & 96 & 17 & --- & --- & --- & --- & --- & --- \\
			LinLog      & 99 & \textbf{100} & 99 & \textbf{100} & \textbf{100} & \textbf{100} & --- & --- & --- & --- \\
			LogReg-real & \textbf{100} & \textbf{100} & \textbf{100} & \textbf{100} & --- & --- & --- & --- & 53 & 49 \\
			Rosenbrock  & \textbf{100} & \textbf{100} & \textbf{100} & \textbf{100} & --- & --- & 97 & \textbf{100} & \textbf{100} & \textbf{100} \\
			Styblinski  & \textbf{100} & \textbf{100} & \textbf{100} & \textbf{100} & --- & --- & --- & --- & \textbf{100} & \textbf{100} \\
			LogReg-NCVR & \textbf{100} & \textbf{100} & \textbf{100} & 99 & --- & --- & --- & --- & 98 & 96 \\
			\textit{Avg} & \textbf{\textit{100}} & \textit{96} & \textit{99} & \textit{91}
			& \textit{11} & \textit{11} & \textit{33} & \textit{18}
			& \textit{50} & \textit{50} \\
			\midrule
			\multicolumn{11}{l}{Far initialization ($r{=}3$)} \\
			\midrule
			Ridge       & \textbf{100} & \textbf{100} & \textbf{100} & \textbf{100} & --- & --- & \textbf{100} & 59 & \textbf{100} & \textbf{100} \\
			QuadBad     & \textbf{100} & \textbf{100} & \textbf{100} & \textbf{100} & --- & --- & 98 & 1 & 1 & 1 \\
			LogSumExp   & \textbf{100} & \textbf{100} & 98 & \textbf{100} & --- & --- & --- & --- & --- & --- \\
			Huber       & \textbf{100} & 80 & 96 & 24 & --- & --- & --- & --- & --- & --- \\
			LinLog      & 99 & \textbf{100} & 99 & \textbf{100} & 89 & \textbf{100} & --- & --- & --- & --- \\
			LogReg-real & \textbf{100} & \textbf{100} & \textbf{100} & \textbf{100} & --- & --- & --- & --- & 51 & 48 \\
			Rosenbrock  & \textbf{100} & \textbf{100} & \textbf{100} & \textbf{100} & --- & --- & 82 & 96 & \textbf{100} & \textbf{100} \\
			Styblinski  & 57 & 61 & \textbf{100} & \textbf{100} & --- & --- & --- & --- & \textbf{100} & \textbf{100} \\
			LogReg-NCVR & \textbf{100} & \textbf{100} & \textbf{100} & \textbf{100} & --- & 1 & --- & --- & 99 & 97 \\
			\textit{Avg} & \textit{95} & \textit{93} & \textbf{\textit{99}} & \textit{92}
			& \textit{10} & \textit{11} & \textit{31} & \textit{17}
			& \textit{50} & \textit{50} \\
			\bottomrule
		\end{tabular}
	\end{table}
	
	\textsc{AdaDisGrem} is the most robust variant
	(about 99\,\% average success, near and far), followed closely by
	\textsc{DisGrem}.
	Both maintain nearly identical rates across the two initialization
	radii, showing stable regions of convergence.
	The adaptive mechanism is especially valuable on
	Styblinski--Tang, where \textsc{AdaDisGrem} and
	\textsc{CeAdaDisGrem} maintain 100\% success even from far starts.
	Hessian compression has mild effects on most functions but reduces robustness on Huber,
	where \textsc{CeAdaDisGrem} succeeds on only 18--24\% of trials.
	The drop is consistent with compression noise accumulating in the
	Hessian tracker and, for \textsc{CeAdaDisGrem}, perturbing the
	$\hat M$ estimation.
	The baselines are function-specific:
	first-order methods succeed only on LinLog; DQM only on
	Ridge/QuadBad/Rosenbrock; SONATA/Net-GIANT only on
	Ridge/Rosenbrock/Styblinski--Tang.
	Huber remains the hardest function for the compressed adaptive variant,
	due to its near-linear tails and flat curvature regions.
	
	\subsubsection{Parameter sensitivity}
	
	We sweep the key parameter ($M_{\mathrm{fac}}$ for the
	\textsc{DisGrem} family; stepsize~$\alpha$ for first-order methods;
	penalty/regularization for second-order baselines) over 10 values
	on four representative functions.
	Full curves are in Appendix~\ref{app:supp_exp},
	Figures~\ref{fig:sweep_disgrem}--\ref{fig:sweep_second}.
	
	The \textsc{DisGrem} family converges across a wide range of
	$M_{\mathrm{fac}}$ on well-conditioned problems; only below a
	problem-dependent threshold does it diverge (e.g., $M_{\mathrm{fac}}{<}1$
	on LogSumExp).
	First-order baselines are highly sensitive to~$\alpha$: too large
	triggers divergence, too small causes stagnation, and even the
	best $\alpha$ fails to break the $\mathrm{relF}>10^{-1}$ barrier on
	LogSumExp.
	Second-order baselines exhibit moderate sensitivity; DQM and ESOM
	diverge for aggressive penalty values, while SONATA and Net-GIANT are
	more stable but stagnate on LogSumExp.
	
	\FloatBarrier
	
	\subsection{Scalability with problem dimension}\label{ssec:scalability}
	
	To assess how the \textsc{DisGrem} family scales with problem
	dimension, we repeat the benchmark on three representative functions
	(Ridge, LogSumExp, Rosenbrock) at $d=30$, $100$, and $200$, keeping
	$N=10$ agents and the same logarithmic communication
	implementation as in Section~\ref{ssec:setup}, with
	5~Monte Carlo runs
	(reduced from the 20 runs in Section~\ref{ssec:convergence}
	due to the higher per-run cost at $d{=}200$;
	the deterministic objective functions ensure low inter-run variance,
	so the median curves are highly stable).
	We compare \textsc{DisGrem}, \textsc{AdaDisGrem}, EXTRA, and SONATA.
	
	Figure~\ref{fig:scalability} reports $\mathrm{relF}$ versus iteration.
	On Ridge and Rosenbrock, the \textsc{DisGrem} family reaches
	$\mathrm{relF}\le 10^{-12}$ within a similar iteration count across
	all three dimensions, suggesting an empirically near dimension-insensitive
	iteration count, consistent with the network-independent iteration bound
	of Theorem~\ref{thm:main}.
	On LogSumExp, the final accuracy degrades slightly at $d=200$, yet
	\textsc{DisGrem} still outperforms both baselines by 3--5 orders
	of magnitude in $\mathrm{relF}$.
	EXTRA stagnates above $\mathrm{relF}\sim 1$ on all problems
	regardless of dimension, while SONATA converges but requires
	substantially more iterations.
	
	The per-iteration wall-clock time scales as $\cO(d^3)$, because each
	agent solves a dense $d\times d$ linear system
	(e.g., on \textsc{Ridge}, the cumulative
	time for all 5 MC runs increases from 2.3\,s at $d{=}30$ to 118\,s at $d{=}200$).
	Since per-iteration communication volume is $\cO(d^2)$
	(already analyzed in \S\ref{ssec:comm}), the dominant bottleneck
	at high dimensions shifts from communication to local computation.
	
	\begin{figure}[htbp]
		\centering
		\includegraphics[width=\textwidth]{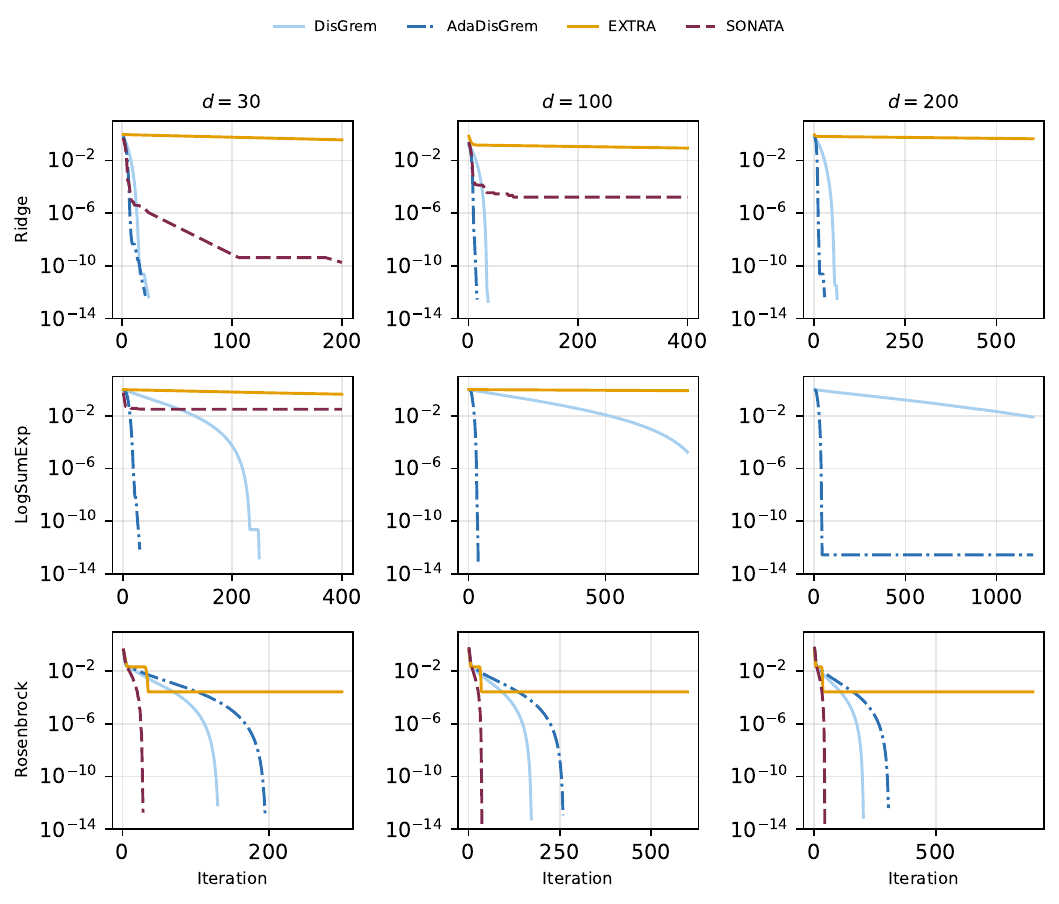}
		\caption{Dimension scalability: $\mathrm{relF}$ vs.\ iteration for
			$d\in\{30,100,200\}$ on three functions (5 MC runs, median).}
		\label{fig:scalability}
	\end{figure}
	
	
	\section{Conclusion and future directions}\label{sec:conclusion}
	
	Under the bounded-trajectory and smoothness assumptions used in the
	analysis, \textsc{DisGrem} retains the
	centralized regularized Newton post-burn-in rate, with the burn-in
	order controlled by the scheduled consensus accuracy,
	in a fully decentralized setting without line search, stepsize tuning,
	or static Hessian-heterogeneity constants.
	Algorithmic stability is achieved by combining a local eigenvalue-shift
	stabilizer with a two-stage mixing protocol.
	Analytically, the key ingredients are a virtual reference-step
	construction---which reduces the decentralized dynamics to an inexact
	centralized update---and an increment-based dispersion analysis.
	By bounding tracker mismatch through Lipschitz differences rather than
	fixed heterogeneity constants, the consensus error becomes transient and
	imposes no accuracy floor.
	Under the logarithmic mixing schedule used in the analysis,
	with $p\ge3$, the total communication cost for a fixed connected
	network is $\cO(\varepsilon^{-1}\log(1/\varepsilon))$ rounds.
	More explicitly, the dependence on the mixing rate is
	$\cO((1-\rho)^{-1}\varepsilon^{-1}\log(1/\varepsilon))$ as
	$\rho\to1$.
	The experiments use the corresponding implementation over the
	precision range reported in the figures and tables.
	
	The experiments are consistent with the theoretical picture and provide
	three further findings. First, \textsc{AdaDisGrem} attains the highest success rate
	(about 99\% average over nine functions and two initialization radii)
	without choosing a separate fixed $M$ for each problem. Second, Hessian compression
	(\textsc{CeDisGrem}) can reduce
	total communication in the fixed-$M$ regime, saving 18--25\% at
	moderate precision in our tests, though compression noise can erode
	these savings at extremely high precision. Third, dimension-scalability
	experiments ($d\in\{30,100,200\}$) demonstrate a nearly dimension-insensitive
	iteration count, consistent with the network-independent theoretical bound.
	
	Future directions include a nonconvex extension, lower bounds for
	decentralized second-order methods with explicit spectral-gap dependence,
	and extensions to directed or time-varying graphs.
	The $\cO(d^3)$ cost of solving the local Newton system also motivates
	inversion-free variants based on iterative linear solvers~\citep{jakovetic2025dinas}
	or inversion-free tracking~\citep{yuan2023indo}, as well as stochastic
	or variance-reduced gradient and Hessian estimators for large-scale settings.
	
	
	\section*{Acknowledgments}
	The work of W.~Hu, Y.-X.~Yuan, and L.~Zhang was supported in part by NSFC and the Chinese Academy of Sciences.
	The work of L.~Zhang was also supported by the China Postdoctoral Science Foundation under Grant Nos.~2023T160670 and 2023M743720.
	The work of P.~Xie was supported in part by the U.S.\ Department of Energy, Office of Science.
	
	\bibliographystyle{plain}

	\appendix
	
	\section{Dispersion recursions, burn-in analysis, and logarithmic mixing}\label{app:dispersion}
	
	\subsection{Basic inequalities}\label{ssec:basic}
	\begin{lemma}\label{lem:taylor3}
		Let $h:\R^d\to\R$ have $L_2$-Lipschitz Hessian. Then for all $x,s$,
		\[
		h(x+s)\le h(x)+\ip{\nabla h(x)}{s}+\frac12 s^\top\nabla^2 h(x)s+\frac{L_2}{6}\norm{s}^3.
		\]
	\end{lemma}
	
	\begin{lemma}\label{lem:consensus}
		For any $z\in\R^N$ and integer $t\ge 1$,
		\[
		\normtwo{(W^t-\J)z}\le \rho^t\normtwo{z}.
		\]
		Consequently, for any stacked $Z\in\R^{Nd}$,
		\[
		\norm{((W^t-\J)\otimes I_d)Z}\le \rho^t\norm{Z}.
		\]
	\end{lemma}
	
	\begin{lemma}\label{lem:gap-grad}
		Let $h$ be convex with $L$-Lipschitz gradient. Then for any minimizer $x_\star$,
		\[
		\norm{\nabla h(x)}^2\le 2L\big(h(x)-h(x_\star)\big).
		\]
	\end{lemma}
	
	This appendix develops the tracker-dispersion analysis in full detail,
	eliminating the non-vanishing heterogeneity constants present in prior
	work.
	The central idea is to bound tracker increments using Lipschitz
	continuity of differences
	$\nabla F(X_{k+1})-\nabla F(X_k)$ and
	$\nabla^2F(X_{k+1})-\nabla^2F(X_k)$.
	
	\subsection{Dispersion operators and a contraction identity}
	
	For $Z=[z_1;\dots;z_N]\in\R^{Nd}$ define the (RMS) dispersion
	\begin{gather*}
		D(Z):=\sqrt{\tfrac1N\textstyle\sum_{i=1}^N\|z_i\!-\!\bar z\|^2}
		=\tfrac{1}{\sqrt N}\|(\Pmat\!\otimes\! I_d)Z\|,\\
		\bar z:=\tfrac1N\textstyle\sum_{i=1}^N z_i.
	\end{gather*}
	Since $\|A\|_F = \|\operatorname{vec}(A)\|_2$, the standard vector dispersion $D(\cdot)$ naturally applies to the stacked vectorized Hessians $\mathcal{H} \in \R^{Nd^2}$.
	
	\begin{lemma}\label{lem:mix_contract}
		For any stacked vectors $Z\in\R^{Nd}$ and any integer $t\ge1$,
		\[
		D\big((W^t\otimes I_d)Z\big)\le \rho^t\,D(Z).
		\]
		Likewise, for any stacked matrices $\mathcal H\in\R^{Nd^2}$,
		\[
		D\big((W^t\otimes I_{d^2})\mathcal H\big)\le \rho^t\,D(\mathcal H).
		\]
	\end{lemma}
	\begin{proof}
		We prove the vector case; the matrix case follows identically with $I_d$ replaced by $I_{d^2}$.
		Let $Z^\perp:=((\Pmat)\otimes I_d)Z$. Since $W$ is doubly stochastic, $W^t\J=\J$ and $\J W^t=\J$, hence
		$\Pmat W^t=(W^t-\J)$. Therefore
		\begin{align*}
			&(\Pmat\!\otimes\! I_d)(W^t\!\otimes\! I_d)Z\\
			&\quad=((W^t\!-\!\J)\!\otimes\! I_d)Z=((W^t\!-\!\J)\!\otimes\! I_d)Z^\perp,
		\end{align*}
		because $(W^t-\J)\J=0$ implies $(W^t-\J)=(W^t-\J)\Pmat$.
		Taking norms and using Lemma~\ref{lem:consensus} gives
		\begin{align*}
			\|(\Pmat\!\otimes\! I_d)(W^t\!\otimes\! I_d)Z\|
			&=\|((W^t\!-\!\J)\!\otimes\! I_d)Z^\perp\|\\
			&\le \|W^t\!-\!\J\|_2\,\|Z^\perp\|
			\le \rho^t\|Z^\perp\|.
		\end{align*}
		Divide by $\sqrt N$ to obtain the claim.
	\end{proof}
	
	\subsection{A uniform step bound}
	
	\begin{lemma}\label{lem:step_bound_app}
		For every $k\ge0$ and $i\in\{1,\dots,N\}$,
		\[
		\|s_{i,k}\|\le \sqrt{\frac{\|\tilde g_{i,k}\|}{M}}.
		\]
		Consequently, using the boundedness from Assumption~\ref{ass:bounded_iterates},
		\[
		\|x_{i,k+1}-x_{i,k}\|\le 2D
		\qquad\text{and}\qquad
		\|X_{k+1}-X_k\|\le 2\sqrt N\,D,
		\]
		where $X_k=[x_{1,k};\dots;x_{N,k}]$.
	\end{lemma}
	\begin{proof}
		Fix $i,k$. The local linear system is
		\[
		(\tilde H_{i,k}+(\lambda_{i,k}+\delta_{i,k})I)s_{i,k}=-\tilde g_{i,k}.
		\]
		By definition of $\delta_{i,k}$, the matrix $\tilde H_{i,k}+\delta_{i,k}I$ is PSD, hence
		\[
		\tilde H_{i,k}+(\lambda_{i,k}+\delta_{i,k})I \succeq \lambda_{i,k}I.
		\]
		Taking norms, we obtain
		\[
		\lambda_{i,k}\|s_{i,k}\|\le \|(\tilde H_{i,k}+(\lambda_{i,k}+\delta_{i,k})I)s_{i,k}\|=\|\tilde g_{i,k}\|.
		\]
		If $\tilde g_{i,k}=0$, then $s_{i,k}=0$ and the bound holds. Otherwise divide by $\lambda_{i,k}>0$:
		\[
		\|s_{i,k}\|\le \frac{\|\tilde g_{i,k}\|}{\lambda_{i,k}}
		=\frac{\|\tilde g_{i,k}\|}{\sqrt{M\|\tilde g_{i,k}\|}}
		=\sqrt{\frac{\|\tilde g_{i,k}\|}{M}}.
		\]
		For the second part, Assumption~\ref{ass:bounded_iterates} gives $\|x_{i,k}\|\le \|x_\star\|+D$ and in particular
		$\|x_{i,k+1}-x_{i,k}\|\le \|x_{i,k+1}-x_\star\|+\|x_{i,k}-x_\star\|\le 2D$.
		Summing squares yields $\|X_{k+1}-X_k\|\le 2\sqrt N\,D$.
	\end{proof}
	
	\subsection{Primal dispersion recursion (explicit)}
	
	Recall the stacked post-mixing update is
	\[
	X_{k+1}=(W^{t_k}\otimes I_d)(\tilde X_k+S_k),
	\qquad
	\tilde X_{k+1}=(W^{\tau_{k+1}}\otimes I_d)X_{k+1}.
	\]
	\begin{lemma}\label{lem:DX_rec}
		For all $k\ge0$,
		\[
		D(\tilde X_{k+1})
		\le \rho^{\tau_{k+1}}\rho^{t_k}\Big(D(\tilde X_k)+\frac{\|S_k\|}{\sqrt N}\Big).
		\]
		Also,
		\[
		D(X_{k+1})
		\le \rho^{t_k}\Big(D(\tilde X_k)+\frac{\|S_k\|}{\sqrt N}\Big).
		\]
	\end{lemma}
	\begin{proof}
		By Lemma~\ref{lem:mix_contract} applied to $Z=\tilde X_k+S_k$,
		\[
		D(X_{k+1})=D((W^{t_k}\otimes I_d)(\tilde X_k+S_k))\le \rho^{t_k}D(\tilde X_k+S_k).
		\]
		Next we bound $D(\tilde X_k+S_k)$ by the triangle inequality in the Hilbert space:
		\begin{align*}
			D(\tilde X_k\!+\!S_k)
			&=\tfrac1{\sqrt N}\|(\Pmat\!\otimes\! I_d)(\tilde X_k\!+\!S_k)\|\\
			&\le \tfrac1{\sqrt N}\|(\Pmat\!\otimes\! I_d)\tilde X_k\|\\
			&\quad+\tfrac1{\sqrt N}\|(\Pmat\!\otimes\! I_d)S_k\|.
		\end{align*}
		The first term equals $D(\tilde X_k)$.
		For the second, $\|((\Pmat)\otimes I_d)S_k\|\le \|S_k\|$, hence it is bounded by $\|S_k\|/\sqrt N$.
		This proves the post-mixing inequality.
		Finally, apply Lemma~\ref{lem:mix_contract} to the pre-mixing step:
		$D(\tilde X_{k+1})\le \rho^{\tau_{k+1}}D(X_{k+1})$ and combine.
	\end{proof}
	
	\subsection{Gradient-tracker dispersion recursion (increment-based, fully expanded)}
	
	Write stacked variables
	\begin{gather*}
		G_k:=[g_{1,k};\dots;g_{N,k}],\;
		\tilde G_k:=[\tilde g_{1,k};\dots;\tilde g_{N,k}],\\
		\Delta\nabla F_k:=\nabla F(X_{k+1})-\nabla F(X_k).
	\end{gather*}
	Algorithm~\ref{alg:disgrem} implies the stacked recursion
	\begin{gather}
		G_{k+1}\!=\!(W^{t_k}\!\otimes\! I_d)(\tilde G_k\!+\!\Delta\nabla F_k),\label{eq:G_update_stack}\\
		\tilde G_{k+1}\!=\!(W^{\tau_{k+1}}\!\otimes\! I_d)G_{k+1}.\notag
	\end{gather}
	
	\begin{lemma}\label{lem:DG_rec}
		For all $k\ge0$,
		\begin{equation}\label{eq:DG_rec}
			D(\tilde G_{k+1})
			\le \rho^{\tau_{k+1}}\rho^{t_k}\Big(D(\tilde G_k)+\frac{L_1}{\sqrt N}\|X_{k+1}-X_k\|\Big).
		\end{equation}
		In particular, using Lemma~\ref{lem:step_bound_app},
		\begin{equation}\label{eq:DG_rec_const}
			D(\tilde G_{k+1})
			\le \rho^{\tau_{k+1}}\rho^{t_k}\Big(D(\tilde G_k)+2L_1D\Big).
		\end{equation}
	\end{lemma}
	\begin{proof}
		Starting from \eqref{eq:G_update_stack}, substitute $G_{k+1}$ into $\tilde G_{k+1}$:
		\begin{align*}
			\tilde G_{k+1}
			&=(W^{\tau_{k+1}}\!\otimes\! I_d)(W^{t_k}\!\otimes\! I_d)(\tilde G_k\!+\!\Delta\nabla F_k)\\
			&=(W^{\tau_{k+1}+t_k}\!\otimes\! I_d)(\tilde G_k\!+\!\Delta\nabla F_k).
		\end{align*}
		Let $m:=\tau_{k+1}+t_k$ and $Z:=\tilde G_k+\Delta\nabla F_k$.
		Apply $(\Pmat\otimes I_d)$ and use $\Pmat W^m=W^m-\J$:
		\[
		(\Pmat\!\otimes\! I_d)\tilde G_{k+1}
		=((W^m-\J)\!\otimes\! I_d)Z.
		\]
		The projection entering the contraction is explicit because
		\[
		W^m-\J=(W^m-\J)(I-\J),
		\]
		and therefore
		\begin{align*}
			D(\tilde G_{k+1})
			&=\tfrac1{\sqrt N}\|((W^m-\J)\!\otimes\! I_d)
			((\Pmat\!\otimes\! I_d)Z)\|\\
			&\le \rho^m\,D(Z).
		\end{align*}
		Next, the triangle inequality for projected RMS dispersions gives
		\[
		D(Z)\le D(\tilde G_k)+D(\Delta\nabla F_k)
		\le D(\tilde G_k)+\tfrac1{\sqrt N}\|\Delta\nabla F_k\|,
		\]
		where the last inequality uses
		$\|(\Pmat\otimes I_d)\Delta\nabla F_k\|\le\|\Delta\nabla F_k\|$.
		It remains to bound $\|\Delta\nabla F_k\|$.
		By block structure,
		\begin{align*}
			\|\Delta\nabla F_k\|^2
			&=\textstyle\sum_{i=1}^N\|\nabla f_i(x_{i,k+1})\!-\!\nabla f_i(x_{i,k})\|^2\\
			&\le \textstyle\sum_{i=1}^N (L_1\|x_{i,k+1}\!-\!x_{i,k}\|)^2
			=L_1^2\|X_{k+1}\!-\!X_k\|^2.
		\end{align*}
		Taking square roots gives $\|\Delta\nabla F_k\|\le L_1\|X_{k+1}-X_k\|$ and yields \eqref{eq:DG_rec}.
		Finally, Lemma~\ref{lem:step_bound_app} gives
		$\|X_{k+1}-X_k\|\le 2\sqrt N D$, which proves
		\eqref{eq:DG_rec_const}.
	\end{proof}
	
	\subsection{Hessian-tracker dispersion recursion (increment-based, fully expanded)}
	
	Define stacked Hessian trackers $\mathcal H_k:=[H_{1,k};\dots;H_{N,k}]\in\R^{Nd^2}$ and
	$\tilde{\mathcal H}_k:=[\tilde H_{1,k};\dots;\tilde H_{N,k}]$.
	Define the stacked Hessian increment
	\[
	\Delta\nabla^2 F_k:=\nabla^2 F(X_{k+1})-\nabla^2 F(X_k),
	\]
	viewed as a block vector in $\R^{Nd^2}$ with Frobenius norm per block.
	The Hessian-tracker recursion is
	\begin{gather}
		\mathcal H_{k+1}\!=\!(W^{t_k}\!\otimes\! I_{d^2})(\tilde{\mathcal H}_k\!+\!\Delta\nabla^2 F_k),\label{eq:H_update_stack}\\
		\tilde{\mathcal H}_{k+1}\!=\!(W^{\tau_{k+1}}\!\otimes\! I_{d^2})\mathcal H_{k+1}.\notag
	\end{gather}
	
	\begin{lemma}\label{lem:DH_rec}
		For all $k\ge0$,
		\begin{equation}\label{eq:DH_rec}
			D(\tilde{\mathcal H}_{k+1})
			\le \rho^{\tau_{k+1}}\rho^{t_k}\Big(D(\tilde{\mathcal H}_k)+\frac{\sqrt d\,L_2}{\sqrt N}\|X_{k+1}-X_k\|\Big).
		\end{equation}
		In particular, using Lemma~\ref{lem:step_bound_app},
		\begin{equation}\label{eq:DH_rec_const}
			D(\tilde{\mathcal H}_{k+1})
			\le \rho^{\tau_{k+1}}\rho^{t_k}\Big(D(\tilde{\mathcal H}_k)+2\sqrt d\,L_2D\Big).
		\end{equation}
	\end{lemma}
	\begin{proof}
		Applying the projected mixing identity used in
		Lemma~\ref{lem:DG_rec}, now with block dimension $d^2$, gives
		\[
		\tilde{\mathcal H}_{k+1} = (W^{\tau_{k+1}}\!\otimes\! I_{d^2})\mathcal H_{k+1} \implies
		D(\tilde{\mathcal H}_{k+1}) \le \rho^{\tau_{k+1}} D(\mathcal H_{k+1}),
		\]
		where we used Lemma~\ref{lem:mix_contract}. From the linear tracker update,
		\[
		\mathcal H_{k+1} = (W^{t_k}\!\otimes\! I_{d^2})(\tilde{\mathcal H}_k\!+\!\Delta\nabla^2 F_k),
		\]
		hence Lemma~\ref{lem:mix_contract} yields
		\begin{align*}
			D(\mathcal H_{k+1})
			&\le \rho^{t_k}D(\tilde{\mathcal H}_k\!+\!\Delta\nabla^2 F_k)\\
			&\le \rho^{t_k}\!\Big(D(\tilde{\mathcal H}_k)\!+\!\tfrac1{\sqrt N}\|\Delta\nabla^2 F_k\|_F\Big).
		\end{align*}
		Combining the last two displays yields \eqref{eq:DH_rec} once
		$\|\Delta\nabla^2F_k\|_F$ is bounded.
		By block structure and the inequality $\|A\|_F\le \sqrt d\,\|A\|_2$,
		\begin{align*}
			\|\Delta\nabla^2 F_k\|_F^2
			&=\textstyle\sum_{i=1}^N\|\nabla^2 f_i(x_{i,k+1})\!-\!\nabla^2 f_i(x_{i,k})\|_F^2\\
			&\le \textstyle\sum_{i=1}^N d\,\|\nabla^2 f_i(x_{i,k+1})\!-\!\nabla^2 f_i(x_{i,k})\|_2^2.
		\end{align*}
		Using $L_2$-Lipschitzness of $\nabla^2 f_i$,
		\[
		\|\nabla^2 f_i(x_{i,k+1})\!-\!\nabla^2 f_i(x_{i,k})\|_2\le L_2\|x_{i,k+1}\!-\!x_{i,k}\|,
		\]
		hence
		\begin{align*}
			\|\Delta\nabla^2 F_k\|_F^2
			&\le \textstyle\sum_{i=1}^N d L_2^2\|x_{i,k+1}\!-\!x_{i,k}\|^2\\
			&= dL_2^2\|X_{k+1}\!-\!X_k\|^2.
		\end{align*}
		Taking square roots gives $\|\Delta\nabla^2 F_k\|_F\le \sqrt d\,L_2\,\|X_{k+1}-X_k\|$,
		which yields \eqref{eq:DH_rec}. Finally use Lemma~\ref{lem:step_bound_app} to obtain \eqref{eq:DH_rec_const}.
	\end{proof}
	
	\subsection{Uniform tracker bounds and post-mixing primal decay}
	
	\begin{lemma}\label{lem:tilde_track_bounded}
		If Assumptions~\ref{ass:problem} and~\ref{ass:bounded_iterates}
		hold and $\tau_k,t_k\ge 1$ for all $k$,
		then there exists a finite constant $B_G>0$ such that
		\[
		\sup_{k\ge 0} D(\tilde G_k)\le B_G,
		\]
		and consequently there exists a finite constant $G_g>0$ such that
		\[
		\sup_{k\ge 0}\max_{1\le i\le N}\|\tilde g_{i,k}\|\le G_g.
		\]
	\end{lemma}
	\begin{proof}
		From \eqref{eq:DG_rec_const} we have
		\[
		D(\tilde G_{k+1})\le \rho^{\tau_{k+1}}\rho^{t_k}\big(D(\tilde G_k)+2L_1D\big)\le \rho^2\big(D(\tilde G_k)+2L_1D\big),
		\]
		since $\tau_{k+1},t_k\ge 1$ and $\rho\in[0,1)$.
		Set $q:=\rho^2\in[0,1)$ and $b:=2L_1D$. Unrolling yields
		\[
		D(\tilde G_k)\le q^k D(\tilde G_0)+\frac{qb}{1-q}\quad\forall k\ge 0,
		\]
		so we may take $B_G:=D(\tilde G_0)+\frac{qb}{1-q}$.
		Next, Lemma~\ref{lem:avg-track-g} and the iterate boundedness from Assumption~\ref{ass:bounded_iterates} imply that
		$\|\tilde{\bar g}_k\|=\|\bar g_k\|=\big\|\frac1N\sum_{i=1}^N \nabla f_i(x_{i,k})\big\|$ is uniformly bounded on the level set;
		let $G_{\mathrm{avg}}:=\sup_{k\ge 0}\|\tilde{\bar g}_k\|<\infty$.
		For any $i$,
		\[
		\|\tilde g_{i,k}\|\le \|\tilde{\bar g}_k\|+\|\tilde g_{i,k}-\tilde{\bar g}_k\|
		\le G_{\mathrm{avg}}+\sqrt N\,D(\tilde G_k)\le G_{\mathrm{avg}}+\sqrt N\,B_G.
		\]
		Thus we may take $G_g:=G_{\mathrm{avg}}+\sqrt N\,B_G$.
	\end{proof}
	
	For later reference, we record the uniform gradient bound from Lemma~\ref{lem:tilde_track_bounded}:
	\begin{equation}\label{eq:step_unif_app}
		G_g:=\sup_{k\ge 0}\ \max_{1\le i\le N}\ \|\tilde g_{i,k}\| < \infty,
	\end{equation}
	where finiteness follows from Lemma~\ref{lem:tilde_track_bounded}.
	
	\begin{lemma}\label{lem:tilde_hess_bounded}
		If Assumptions~\ref{ass:problem} and~\ref{ass:bounded_iterates}
		hold and $\tau_k,t_k\ge 1$ for all $k$,
		then there exists a finite constant $B_H>0$ such that
		\[
		\sup_{k\ge 0} D(\tilde{\mathcal H}_k)\le B_H.
		\]
		Consequently, there exists a finite constant $G_H>0$ such that
		\[
		\sup_{k\ge 0}\max_{1\le i\le N}\big\|\tilde H_{i,k}\big\|_2 \le G_H,
		\qquad
		\sup_{k\ge 0}\bar\delta_k \le G_H.
		\]
	\end{lemma}
	
	\begin{proof}
		From Lemma~\ref{lem:DH_rec} and $\tau_{k+1},t_k\ge 1$, we have
		\[
		D(\tilde{\mathcal H}_{k+1})
		\le \rho^{\tau_{k+1}}\rho^{t_k}\Big(D(\tilde{\mathcal H}_k)+2\sqrt d\,L_2D\Big)
		\le \rho^2\Big(D(\tilde{\mathcal H}_k)+2\sqrt d\,L_2D\Big).
		\]
		Let $q:=\rho^2\in[0,1)$ and $b_H:=2\sqrt d\,L_2D$. Unrolling yields
		\[
		D(\tilde{\mathcal H}_k)\le q^k D(\tilde{\mathcal H}_0)+\frac{qb_H}{1-q}\quad \forall k\ge 0,
		\]
		so we may take $B_H:=D(\tilde{\mathcal H}_0)+qb_H/(1-q)$.
		
		Next, by Lemma~\ref{lem:avg-track-h},
		$\tilde{\bar H}_k=\bar H_k=\frac1N\sum_{i=1}^N\nabla^2 f_i(x_{i,k})$,
		hence $\|\tilde{\bar H}_k\|_2\le M_{H,\max}$.
		For any $i$,
		\[
		\|\tilde H_{i,k}\|_2
		\le \|\tilde{\bar H}_k\|_2
		+ \|\tilde H_{i,k}\!-\!\tilde{\bar H}_k\|_2
		\le M_{H,\max}\!+\!\sqrt{N}\,B_H.
		\]
		Define $G_H:=M_{H,\max} + \sqrt{N}\,B_H$.
		
		Finally, $\delta_{i,k}=\max\{0,-\lambda_{\min}(\tilde H_{i,k})\}
		\le \|\tilde H_{i,k}\|_2\le G_H$,
		and averaging yields $\bar\delta_k\le G_H$.
	\end{proof}
	
	\begin{lemma}\label{lem:DX_post}
		For all $k\ge 0$,
		\[
		D(X_{k+1})\le \rho^{t_k}\Big(D(\tilde X_k)+\frac{\|S_k\|}{\sqrt N}\Big).
		\]
	\end{lemma}
	\begin{proof}
		Recall $X_{k+1}=(W^{t_k}\otimes I_d)(\tilde X_k+S_k)$ and that $D(Z)=\|Z^\perp\|/\sqrt N$ with $Z^\perp:=((I-\J)\otimes I_d)Z$.
		Since $(W^{t_k}-\J)$ annihilates consensus components, we have
		\[
		(X_{k+1})^\perp = ((W^{t_k}-\J)\otimes I_d)(\tilde X_k+S_k),
		\]
		and by Lemma~\ref{lem:consensus} (applied in stacked form) and the triangle inequality,
		\[
		\| (X_{k+1})^\perp\|
		\le \rho^{t_k}\|(\tilde X_k+S_k)^\perp\|
		\le \rho^{t_k}\big(\|\tilde X_k^\perp\|+\|S_k^\perp\|\big)
		\le \rho^{t_k}\big(\|\tilde X_k^\perp\|+\|S_k\|\big).
		\]
		Divide by $\sqrt N$ to obtain the claim.
	\end{proof}
	
	\begin{lemma}\label{lem:DX_post_decay}
		Under Assumptions~\ref{ass:problem} and~\ref{ass:bounded_iterates}
		and schedule~\eqref{eq:log_schedule}, there exists a constant $C_X^{+}>0$ such that for all $k\ge 0$,
		\[
		D(X_k)\le \frac{C_X^{+}}{(k+1)^{p}}.
		\]
	\end{lemma}
	\begin{proof}
		From Lemma~\ref{lem:DX_post} and
		$\rho^{t_k}\le e^{-c_{\mathrm{mix}}}(k+2)^{-p}\le (k+2)^{-p}$
		under \eqref{eq:log_schedule}, it remains to bound
		$D(\tilde X_k)$ and $\|S_k\|/\sqrt N$ uniformly.
		By Assumption~\ref{ass:bounded_iterates}, each $\tilde x_{i,k}$ is a convex combination of $\{x_{j,k}\}$, hence $\|\tilde x_{i,k}-x_\star\|\le D$ and
		\[
		D(\tilde X_k)\le \max_{1\le i\le N}\|\tilde x_{i,k}-\bar x_k\|
		\le \max_i \|\tilde x_{i,k}-x_\star\|+\|\bar x_k-x_\star\|\le 2D.
		\]
		Moreover, Lemma~\ref{lem:step_bound_app} yields $\|S_k\|/\sqrt N\le \sqrt{G_g/M}$, where $G_g$ is defined in \eqref{eq:step_unif_app}.
		Therefore,
		\[
		D(X_{k+1})\le (k+2)^{-p}\Big(2D+\sqrt{\frac{G_g}{M}}\Big),
		\]
		and the claim follows with $C_X^{+}:=2D+\sqrt{G_g/M}$ after shifting indices.
	\end{proof}
	
	\subsection{Polynomial decay under a logarithmic schedule}
	
	Fix $p>2$ and set the logarithmic schedule
	\begin{equation}\label{eq:log_schedule}
		\tau_k=t_k=
		\left\lceil \frac{p\log(k+2)+c_{\mathrm{mix}}}{-\log\rho}\right\rceil,\qquad k\ge 0,
	\end{equation}
	so that $\rho^{\tau_k}=\rho^{t_k}\le e^{-c_{\mathrm{mix}}}(k+2)^{-p}$.
	
	\begin{lemma}\label{lem:compare}
		Let $\{u_k\}_{k\ge0}$ satisfy $u_{k+1}\le a_k(u_k+b)$ with $b\ge0$, where
		$a_k\le q<1$ for all $k$ and $a_k\le (k+2)^{-p}$ for all $k$ with $p>2$.
		Then there exists $C>0$ such that $u_k\le C/(k+2)^{p-1}$ for all $k\ge0$.
	\end{lemma}
	\begin{proof}
		Since $a_k\le q<1$, we have $u_{k+1}\le q(u_k+b)$ and hence by induction
		\[
		u_k\le q^k u_0+\sum_{j=0}^{k-1} q^{k-j}qb \le u_0+\frac{qb}{1-q}=:u_{\max},\qquad \forall k\ge 0.
		\]
		Since $a_k\le (k+2)^{-p}$, for every $k\ge 0$ we have
		\[
		u_{k+1}\le a_k(u_k+b)\le \frac{u_{\max}+b}{(k+2)^{p}}\le \frac{u_{\max}+b}{(k+2)^{p-1}}.
		\]
		Since $(k+3)/(k+2)\le 2$ for all $k\ge 0$, we obtain
		$u_{k+1}\le 2^{p-1}(u_{\max}+b)/(k+3)^{p-1}$.
		For $k=0$, the base case $u_0\le u_{\max}\le C/2^{p-1}$ holds by definition.
		The claim follows with $C:=2^{p-1}(u_{\max}+b)$.
	\end{proof}
	
	\begin{proposition}\label{prop:dispersion_decay}
		Under Assumptions~\ref{ass:problem} and~\ref{ass:bounded_iterates}
		and schedule~\eqref{eq:log_schedule},
		there exist constants $C_X,C_G,C_H>0$ such that for all $k\ge0$,
		\[
		D(\tilde X_k)\le \frac{C_X}{(k+2)^{p-1}},
		\qquad
		D(\tilde G_k)\le \frac{C_G}{(k+2)^{p-1}},
		\qquad
		D(\tilde{\mathcal H}_k)\le \frac{C_H}{(k+2)^{p-1}}.
		\]
	\end{proposition}
	\begin{proof}
		Apply Lemma~\ref{lem:compare} to the recursion in Lemma~\ref{lem:DX_rec} and the bounds \eqref{eq:DG_rec_const} and \eqref{eq:DH_rec_const},
		using the bounded forcing constants $b_X:=\sup_k\|S_k\|/\sqrt N<\infty$ (Lemma~\ref{lem:step_bound_app}),
		$b_G:=2L_1D$, and $b_H:=2\sqrt d\,L_2D$.
	\end{proof}
	
	\subsection{Burn-in conditions for Proposition~\ref{prop:inexact-rn}}
	
	We now establish that for every target accuracy $0<\varepsilon\le1$, there
	exists a finite index $K_0(\varepsilon)$ such that whenever
	$\|g_k\|\ge\varepsilon$ and $k\ge K_0(\varepsilon)$,
	the three hypotheses of Proposition~\ref{prop:inexact-rn} are
	satisfied with $\eta=1/12$.
	\begin{proposition}\label{prop:burnin}
		Fix $p>2$ and the schedule~\eqref{eq:log_schedule}.
		For every $0<\varepsilon\le1$ there exists an integer $K_0(\varepsilon)\ge0$ such that
		for every $k\ge K_0(\varepsilon)$ with $\|g_k\|=\|\nabla f(\bar x_k)\|\ge \varepsilon$,
		all three items of Proposition~\ref{prop:inexact-rn} hold with $\eta:=1/12$.
		
		\smallskip
		\noindent\textbf{Moreover (auxiliary conditions).}
		With the same choice of $K_0(\varepsilon)$, we may further guarantee that for all $k\ge K_0(\varepsilon)$,
		\begin{equation}\label{eq:burnin_aux_enforce}
			\Delta_k^g \le \varepsilon^2,\qquad
			\Delta_k^H \le \varepsilon^{3/2},\qquad
			D(X_k)\le \varepsilon^{3/2},\qquad
			\bar\delta_k \le (1+L_2)\varepsilon^{3/2},
		\end{equation}
		where $\Delta_k^g:=\frac1N\sum_{i=1}^N\|\tilde g_{i,k}-\tilde{\bar g}_k\|$,
		$\Delta_k^H:=\frac1N\sum_{i=1}^N\|\tilde H_{i,k}-\tilde{\bar H}_k\|_2$,
		and $\bar\delta_k:=\frac1N\sum_{i=1}^N\delta_{i,k}$.
	\end{proposition}
	
	\begin{proof}
		Fix $0<\varepsilon\le1$ and set $\eta:=1/12$.
		We will produce an explicit index $K_0(\varepsilon)$ such that for every $k\ge K_0(\varepsilon)$ with $\|g_k\|\ge \varepsilon$,
		all three conditions of Proposition~\ref{prop:inexact-rn} hold.
		
		\smallskip
		\noindent\textbf{Step 1: post-mixing disagreement and bridge terms.}
		Recall $D(X_k)=\|X_k^\perp\|/\sqrt N$.
		By Lemma~\ref{lem:DX_post_decay}, there exists $C_X^{+}>0$ such that for all $k\ge 0$,
		\begin{equation}\label{eq:DX_post_decay_app}
			D(X_k)\le \frac{C_X^{+}}{(k+1)^p}.
		\end{equation}
		Therefore, for all $k\ge 0$,
		\begin{equation}\label{eq:bridge_targets_app}
			L_1D(X_k)\le \frac{L_1C_X^{+}}{(k+1)^p},
			\qquad
			L_2D(X_k)\le \frac{L_2C_X^{+}}{(k+1)^p}.
		\end{equation}
		
		\smallskip
		\noindent\textbf{Step 2: tracker dispersions.}
		From Proposition~\ref{prop:dispersion_decay}, there exist constants $C_G,C_H>0$ such that for all $k\ge 0$,
		\begin{equation}\label{eq:GH_decay_app}
			D(\tilde G_k)\le \frac{C_G}{(k+2)^{p-1}},
			\qquad
			D(\tilde{\mathcal H}_k)\le \frac{C_H}{(k+2)^{p-1}}.
		\end{equation}
		By Cauchy--Schwarz and $\|\cdot\|_2\le \|\cdot\|_F$,
		\[
		\Delta_k^g
		=\frac1N\sum_{i=1}^N\|\tilde g_{i,k}-\tilde{\bar g}_k\|
		\le \sqrt{\frac1N\sum_{i=1}^N\|\tilde g_{i,k}-\tilde{\bar g}_k\|^2}
		= D(\tilde G_k),
		\]
		and
		\begin{align*}
			\Delta_k^H
			&=\tfrac1N\textstyle\sum_{i=1}^N\big\|\tilde H_{i,k}\!-\!\tilde{\bar H}_k\big\|_2\\
			&\le \tfrac1N\textstyle\sum_{i=1}^N\big\|\tilde H_{i,k}\!-\!\tilde{\bar H}_k\big\|_F\\
			&\le D(\tilde{\mathcal H}_k).
		\end{align*}
		Hence, using Proposition~\ref{prop:dispersion_decay},
		\begin{equation}\label{eq:Delta_decay_app}
			\Delta_k^g\le \frac{C_G}{(k+2)^{p-1}},
			\qquad
			\Delta_k^H\le \frac{C_H}{(k+2)^{p-1}}.
		\end{equation}
		
		\smallskip
		\noindent\textbf{Step 3: a usable lower bound on the reference step.}
		Assume $\|g_k\|\ge \varepsilon$ and suppose $k$ is large enough such that $L_1D(X_k)\le \varepsilon/2$.
		Then Lemma~\ref{lem:bridge-grad} implies
		$\|\tilde{\bar g}_k\|\ge \varepsilon/2$.
		Moreover, since $A_k^{\mathrm{ref}}\preceq (M_{H,\max}+\tilde{\bar\lambda}_k)I$,
		\[
		\|s_k^{\mathrm{ref}}\|
		\ge \frac{\|\tilde{\bar g}_k\|}{M_{H,\max}+\tilde{\bar\lambda}_k}
		\ge \frac{\varepsilon/2}{M_{H,\max}+\sqrt{MG_g}+G_H}
		=:\underline s(\varepsilon),
		\]
		where we used $\tilde{\bar\lambda}_k=\frac{1}{N}\sum_{i=1}^N(\lambda_{i,k}+\delta_{i,k})\le \sqrt{MG_g}+G_H$, with $G_g$ from~\eqref{eq:step_unif_app} and $G_H$ from Lemma~\ref{lem:tilde_hess_bounded}.
		
		\smallskip
		\noindent\textbf{Step 4: ensure dispersion control (Item~1).}
		We first verify the relative dispersion condition \eqref{eq:rel_grad_disp} with $\alpha_d=1/2$.
		Since $D(\tilde G_k)=\sqrt{\frac1N\sum_{i=1}^N\|\tilde g_{i,k}-\tilde{\bar g}_k\|^2}$, we have
		\[
		\max_{1\le i\le N}\|\tilde g_{i,k}-\tilde{\bar g}_k\|
		\le \sqrt{N}\,D(\tilde G_k).
		\]
		From Step 3, for all $k$ large enough with $\|g_k\|\ge \varepsilon$ we have $\|\tilde{\bar g}_k\|\ge \varepsilon/2$.
		Thus, choosing $k$ large enough so that $\sqrt N\,D(\tilde G_k)\le \frac12\|\tilde{\bar g}_k\|$,
		condition \eqref{eq:rel_grad_disp} holds with $\alpha_d=1/2$.
		Using Lemma~\ref{lem:step-disp} and Lemma~\ref{lem:lambda_dispersion} with $\alpha_d:=1/2$,
		for all sufficiently large $k$ with $\|g_k\|\ge \varepsilon$ we have
		\[
		\|\bar s_k-s_k^{\mathrm{ref}}\|
		\le \frac{6}{\sqrt{M\varepsilon}}\Delta_k^g+\frac{2}{M}\Delta_k^H+\frac{4}{M}\bar\delta_k.
		\]
		To make this estimate quantitative, set
		\[
		B_\lambda:=M_{H,\max}+\sqrt{MG_g}+G_H,\qquad
		\theta:=\frac{\eta\sqrt M}{64B_\lambda^2}.
		\]
		Once the relative dispersion condition holds, Step~5 gives
		$\lambda_k\ge\frac12\sqrt{M\varepsilon}$; Step~3 gives
		$\|s_k^{\mathrm{ref}}\|\ge\varepsilon/(2B_\lambda)$, and
		$\lambda_k\le \sqrt{MG_g}+G_H$ gives
		$M_{H,\max}+\lambda_k\le B_\lambda$. Hence
		\[
		\frac{\eta}{16}\cdot
		\frac{\lambda_k}{M_{H,\max}+\lambda_k}\,
		\|s_k^{\mathrm{ref}}\|
		\ge \theta\,\varepsilon^{3/2}.
		\]
		Using Lemma~\ref{lem:delta_control},
		$\bar\delta_k\le \Delta_k^H+L_2D(X_k)$, the estimate above is
		therefore implied by the explicit sufficient conditions
		\[
		\Delta_k^g\le a_g\varepsilon^2,\qquad
		\Delta_k^H\le a_H\varepsilon^{3/2},\qquad
		D(X_k)\le a_X\varepsilon^{3/2},
		\]
		where
		\[
		a_g:=\frac{\theta\sqrt M}{18},\qquad
		a_H:=\frac{\theta M}{18},\qquad
		a_X:=\frac{\theta M}{12\max\{L_2,1\}}.
		\]
		Indeed, these three inequalities bound
		\[
		\frac{6}{\sqrt{M\varepsilon}}\Delta_k^g
		+\frac{2}{M}\Delta_k^H
		+\frac{4}{M}\bar\delta_k
		\le
		\frac{6}{\sqrt{M\varepsilon}}\Delta_k^g
		+\frac{6}{M}\Delta_k^H
		+\frac{4L_2}{M}D(X_k)
		\le \theta\varepsilon^{3/2}.
		\]
		Define $K_1(\varepsilon)$ as any index such that, for every
		$k\ge K_1(\varepsilon)$,
		\[
		\sqrt N\,D(\tilde G_k)\le \frac{\varepsilon}{4},\qquad
		\Delta_k^g\le a_g\varepsilon^2,\qquad
		\Delta_k^H\le a_H\varepsilon^{3/2},\qquad
		D(X_k)\le a_X\varepsilon^{3/2}.
		\]
		Such an index exists by \eqref{eq:DX_post_decay_app} and
		\eqref{eq:Delta_decay_app}, and its order is still
		$\cO(\varepsilon^{-2/(p-1)})$ because $0<\varepsilon\le1$.
		Consequently, for all $k\ge K_1(\varepsilon)$ with
		$\|g_k\|\ge\varepsilon$,
		\[
		\|\bar s_k-s_k^{\mathrm{ref}}\|\le \frac{\eta}{16}\cdot \frac{\lambda_k}{M_{H,\max}+\lambda_k}\,\|s_k^{\mathrm{ref}}\|.
		\]
		This implies $\|s_k^{\mathrm{ref}}\|\le 2\|\bar s_k\|$ by the triangle inequality, and hence Item~1 of Proposition~\ref{prop:inexact-rn} holds.
		
		\smallskip
		\noindent\textbf{Step 5: ensure Items~2--3.}
		After the relative dispersion condition has been enforced with
		$\alpha_d=1/2$, Step~3 gives
		$\|\tilde g_{i,k}\|\ge \frac12\|\tilde{\bar g}_k\|\ge \varepsilon/4$
		for every~$i$.
		Hence
		\[
		\lambda_k=\frac1N\sum_{i=1}^N
		\bigl(\sqrt{M\|\tilde g_{i,k}\|}+\delta_{i,k}\bigr)
		\ge \frac12\sqrt{M\varepsilon}.
		\]
		Since $\|s_k^{\mathrm{ref}}\|\le 2\|\bar s_k\|$, it is enough to ensure
		\[
		L_1D(X_k)\le \frac{\eta}{16}\lambda_k\|s_k^{\mathrm{ref}}\|,
		\qquad
		L_2D(X_k)\le \frac{\eta}{8}\lambda_k.
		\]
		By \eqref{eq:bridge_targets_app} and $\|s_k^{\mathrm{ref}}\|\ge \underline s(\varepsilon)$, there exists $K_2(\varepsilon)$ such that both inequalities hold for all $k\ge K_2(\varepsilon)$.
		
		\smallskip
		\noindent\textbf{Step 6: enforce the auxiliary burn-in bounds.}
		By \eqref{eq:Delta_decay_app}, $\Delta_k^g\le \frac{C_G}{(k+2)^{p-1}}$.
		Define
		\[
		K_\Delta(\varepsilon):=\min\Big\{k\ge 0:\ 
		\frac{C_G}{(k+2)^{p-1}}\le \varepsilon^2,\ 
		\frac{C_H}{(k+2)^{p-1}}\le \varepsilon^{3/2}\Big\}.
		\]
		Then for all $k\ge K_\Delta(\varepsilon)$, we have
		$\Delta_k^g\le \varepsilon^2$ and
		$\Delta_k^H\le \varepsilon^{3/2}$.
		
		\smallskip
		\noindent\textbf{Step 7: enforce the post-mixing and stabilizer bounds.}
		Define
		\[
		K_\delta(\varepsilon):=\min\Big\{k\ge 1:\ 
		\frac{C_X^{+}}{(k+1)^p}\le \varepsilon^{3/2}\Big\}.
		\]
		Then for all $k\ge K_\delta(\varepsilon)$, we have
		$D(X_k)\le \varepsilon^{3/2}$. Lemma~\ref{lem:delta_control}
		and Step~6 further give
		$\bar\delta_k\le \Delta_k^H+L_2D(X_k)
		\le (1+L_2)\varepsilon^{3/2}$.
		
		\smallskip
		Finally set
		\[
		K_0(\varepsilon):=\max\{K_1(\varepsilon),K_2(\varepsilon),K_\Delta(\varepsilon),K_\delta(\varepsilon)\}.
		\]
		Then for every $k\ge K_0(\varepsilon)$ with $\|g_k\|\ge \varepsilon$, all three items of Proposition~\ref{prop:inexact-rn} hold with $\eta=1/12$,
		and \eqref{eq:burnin_aux_enforce} also holds.
	\end{proof}
	
	\begin{theorem}\label{thm:comm}
		If the hypotheses of Theorem~\ref{thm:main} hold and
		$0<\varepsilon\le1$, then
		\[
		\sum_{k=0}^{K(\varepsilon)-1}(\tau_k+2t_k)
		=\cO\!\Big((1-\rho)^{-1}
		(K_0(\varepsilon)+\varepsilon^{-1})
		\log(K_0(\varepsilon)+\varepsilon^{-1})\Big).
		\]
		In particular, since Theorem~\ref{thm:main} assumes $p\ge3$,
		\[
		\sum_{k=0}^{K(\varepsilon)-1}(\tau_k+2t_k)
		=\cO((1-\rho)^{-1}\varepsilon^{-1}\log(1/\varepsilon)).
		\]
		For a fixed connected network, this reduces to
		$\cO(\varepsilon^{-1}\log(1/\varepsilon))$.
	\end{theorem}
	
	\begin{remark}\label{rem:burnin_order_app}
		For $0<\varepsilon\le1$, the burn-in index
		$K_0(\varepsilon)$ in Theorem~\ref{thm:main} is determined by
		the direct verification of the three conditions in
		Proposition~\ref{prop:inexact-rn}.
		In Step~4 of Proposition~\ref{prop:burnin}, the estimate
		\[
		\|\bar s_k-s_k^{\mathrm{ref}}\|
		\le
		\frac{6}{\sqrt{M\varepsilon}}\Delta_k^g
		+\frac{2}{M}\Delta_k^H
		+\frac{4}{M}\bar\delta_k
		\]
		is compared with an Item~1 target of order
		$\lambda_k\|s_k^{\mathrm{ref}}\|=\Omega(\varepsilon^{3/2})$.
		A conservative sufficient condition is therefore
		\[
		\Delta_k^g=\cO(\varepsilon^2),\qquad
		\Delta_k^H=\cO(\varepsilon^{3/2}),\qquad
		D(X_k)=\cO(\varepsilon^{3/2}),
		\]
		up to constants, where
		$\bar\delta_k\le \Delta_k^H+L_2D(X_k)$.
		Under the logarithmic schedule~\eqref{eq:log_schedule} with
		$\tau_k,t_k=\lceil(p\log(k+2)+c_{\mathrm{mix}})/(-\log\rho)\rceil$,
		Proposition~\ref{prop:dispersion_decay} yields
		$D(\tilde G_k),D(\tilde{\mathcal H}_k)=\cO(k^{-(p-1)})$.
		It therefore suffices, conservatively, to choose $k$ so that
		$k^{-(p-1)}\lesssim\varepsilon^2$, giving
		\[
		K_0(\varepsilon)=\cO\!\big(\varepsilon^{-2/(p-1)}\big).
		\]
		If one keeps only $p>2$, this gives the more general total
		iteration bound
		\[
		K(\varepsilon)
		=\cO\!\left(\varepsilon^{-1}+\varepsilon^{-2/(p-1)}\right).
		\]
		In particular, for $p\ge3$, the burn-in estimate is no larger than
		the $\cO(\varepsilon^{-1})$ post-burn-in term. Hence the total
		iteration complexity in Theorem~\ref{thm:main} and the communication
		complexity in Theorem~\ref{thm:comm} become
		\[
		K(\varepsilon)=\cO(\varepsilon^{-1}),
		\qquad
		\sum_{k=0}^{K(\varepsilon)-1}(\tau_k+2t_k)
		=\cO\bigl((1-\rho)^{-1}\varepsilon^{-1}\log(1/\varepsilon)\bigr),
		\]
		with the latter reducing to
		$\cO(\varepsilon^{-1}\log(1/\varepsilon))$ for a fixed connected
		network.
	\end{remark}
	
	
	\section{Proofs}\label{app:proofs}
	
	\begin{proof}[Proof of Lemma~\ref{lem:taylor3}]
		Fix $x,s$ and define $\phi(t):=h(x+ts)$ for $t\in[0,1]$.
		Then $\phi'(t)=\ip{\nabla h(x+ts)}{s}$ and $\phi''(t)=s^\top\nabla^2 h(x+ts)s$.
		By $L_2$-Lipschitzness of $\nabla^2 h$,
		\begin{align*}
			|\phi''(t)\!-\!\phi''(0)|
			&=|s^\top(\nabla^2 h(x\!+\!ts)\!-\!\nabla^2 h(x))s|\\
			&\le \|\nabla^2 h(x\!+\!ts)\!-\!\nabla^2 h(x)\|_2\|s\|^2\\
			&\le L_2 t\|s\|^3.
		\end{align*}
		Hence $\phi''(t)\le \phi''(0)+L_2 t\|s\|^3$.
		Integrating twice,
		\begin{align*}
			\phi(1)&=\phi(0)+\phi'(0)+\int_0^1\!(1\!-\!t)\phi''(t)\,dt\\
			&\le \phi(0)+\phi'(0)+\!\int_0^1\!(1\!-\!t)\big(\phi''(0)+L_2 t\|s\|^3\big)dt.
		\end{align*}
		Compute $\int_0^1(1-t)\,dt=\tfrac12$ and $\int_0^1(1-t)t\,dt=\tfrac16$ to obtain
		\[
		\phi(1)\le \phi(0)+\phi'(0)+\tfrac12\phi''(0)+\tfrac{L_2}{6}\|s\|^3.
		\]
		Substitute $\phi(0)=h(x)$, $\phi'(0)=\ip{\nabla h(x)}{s}$, and $\phi''(0)=s^\top\nabla^2 h(x)s$.
	\end{proof}
	
	\begin{proof}[Proof of Lemma~\ref{lem:consensus}]
		Because $W$ is symmetric and doubly stochastic, it is diagonalizable with eigenvalues
		$1=\lambda_1>\lambda_2\ge\cdots\ge\lambda_N\ge -1$ and eigenvector $\one$ for $\lambda_1$.
		The projector $\J=\frac1N\one\one^\top$ is the orthogonal projector onto $\mathrm{span}\{\one\}$ and satisfies $W\J=\J W=\J$.
		We show $W^t-\J=(W-\J)^t$ by induction: for $t=1$ trivial. If true for $t$, then
		\begin{align*}
			(W\!-\!\J)^{t+1}&=(W\!-\!\J)^t(W\!-\!\J)\\
			&=(W^t\!-\!\J)(W\!-\!\J)\\
			&=W^{t+1}\!-\!W^t\J\!-\!\J W\!+\!\J^2\\
			&=W^{t+1}\!-\!\J,
		\end{align*}
		using $W^t\J=\J$, $\J W=\J$, and $\J^2=\J$.
		Thus $\|W^t-\J\|_2=\|(W-\J)^t\|_2\le \|W-\J\|_2^t=\rho^t$.
		For stacked vectors, $\|(A\otimes I_d)\|_2=\|A\|_2$.
	\end{proof}
	
	\begin{proof}[Proof of Lemma~\ref{lem:gap-grad}]
		By $L$-smoothness,
		\[
		h(y)\le h(x)+\ip{\nabla h(x)}{y-x}+\frac{L}{2}\|y-x\|^2.
		\]
		Choose $y=x-\frac1L\nabla h(x)$ to obtain
		\[
		h\Big(x-\frac1L\nabla h(x)\Big)\le h(x)-\frac{1}{2L}\|\nabla h(x)\|^2.
		\]
		Since $x_\star$ minimizes $h$, $h(x_\star)\le h(x-\frac1L\nabla h(x))$. Rearranging yields the claim.
	\end{proof}
	
	\begin{proof}[Proof of Lemma~\ref{lem:avg-pres}]
		Because $W^{\tau_k}$ is doubly stochastic,
		\[
		\tilde{\bar x}_k=\frac1N\sum_{i=1}^N\sum_{j=1}^N [W^{\tau_k}]_{ij}x_{j,k}
		=\frac1N\sum_{j=1}^N\Big(\sum_{i=1}^N [W^{\tau_k}]_{ij}\Big)x_{j,k}
		=\frac1N\sum_{j=1}^N x_{j,k}=\bar x_k.
		\]
		Similarly, $W^{t_k}$ is doubly stochastic and $y_{i,k+1}=\tilde x_{i,k}+s_{i,k}$, hence
		\begin{align*}
			\bar x_{k+1}
			&=\tfrac1N\textstyle\sum_{i=1}^N x_{i,k+1}
			=\tfrac1N\textstyle\sum_{i=1}^N\sum_{j=1}^N [W^{t_k}]_{ij}y_{j,k+1}\\
			&=\tfrac1N\textstyle\sum_{j=1}^N y_{j,k+1}
			=\tilde{\bar x}_k+\bar s_k=\bar x_k+\bar s_k.
		\end{align*}
	\end{proof}
	
	\begin{proof}[Proof of Lemma~\ref{lem:avg-track-g}]
		Average the tracker update in Algorithm~\ref{alg:disgrem}:
		\[
		\bar g_{k+1}
		=\frac1N\sum_{i=1}^N g_{i,k+1}
		=\frac1N\sum_{i=1}^N\sum_{j=1}^N [W^{t_k}]_{ij}\Big(\tilde g_{j,k}+\nabla f_j(x_{j,k+1})-\nabla f_j(x_{j,k})\Big).
		\]
		Swap sums and use column-stochasticity $\sum_{i=1}^N [W^{t_k}]_{ij}=1$:
		\begin{align*}
			\bar g_{k+1}
			&=\tfrac1N\textstyle\sum_{j=1}^N
			\Big(\tilde g_{j,k}
			+\nabla f_j(x_{j,k+1})-\nabla f_j(x_{j,k})\Big)\\
			&=\tilde{\bar g}_k
			+\tfrac1N\textstyle\sum_{j=1}^N
			\big(\nabla f_j(x_{j,k+1})-\nabla f_j(x_{j,k})\big).
		\end{align*}
		Pre-mixing preserves averages, hence $\tilde{\bar g}_k=\bar g_k$.
		Starting from
		\[
		\bar g_0=\frac1N\sum_{i=1}^N \nabla f_i(x_{i,0}),
		\]
		induction gives
		$\bar g_k=N^{-1}\sum_{i=1}^N \nabla f_i(x_{i,k})$ for all $k$,
		and thus $\tilde{\bar g}_k=\bar g_k$.
	\end{proof}
	
	\begin{proof}[Proof of Lemma~\ref{lem:avg-track-h}]
		Repeating the analysis in the proof of Lemma~\ref{lem:avg-track-g} with gradients replaced by Hessians yields the claim.
		Averaging the Hessian update gives
		\[
		\bar H_{k+1}=\tilde{\bar H}_k+\frac1N\sum_{i=1}^N\big(\nabla^2 f_i(x_{i,k+1})-\nabla^2 f_i(x_{i,k})\big),
		\]
		and $\tilde{\bar H}_k=\bar H_k$. Induction yields $\bar H_k=\frac1N\sum_{i=1}^N \nabla^2 f_i(x_{i,k})$.
	\end{proof}
	
	\begin{proof}[Proof of Lemma~\ref{lem:M-vs-s}]
		If $\tilde g_{i,k}=0$, Algorithm~\ref{alg:disgrem} sets
		$s_{i,k}=0$ and $\lambda_{i,k}=0$, so the claim is immediate.
		Assume henceforth that $\tilde g_{i,k}\ne0$, so
		$\lambda_{i,k}>0$.
		As in Appendix~\ref{app:dispersion}, the local system satisfies
		$(\tilde H_{i,k}+(\lambda_{i,k}+\delta_{i,k})I)\succeq \lambda_{i,k}I$.
		Thus $\lambda_{i,k}\|s_{i,k}\|\le \|\tilde g_{i,k}\|$.
		Since $\lambda_{i,k}^2=M\|\tilde g_{i,k}\|$, we have $\lambda_{i,k}\|s_{i,k}\|\le \lambda_{i,k}^2/M$,
		hence $M\|s_{i,k}\|\le \lambda_{i,k}$. Using $M\ge L_2$ gives $L_2\|s_{i,k}\|\le\lambda_{i,k}$.
	\end{proof}
	
	\begin{proof}[Proof of Lemma~\ref{lem:avg-regime}]
		From Lemma~\ref{lem:M-vs-s}, $M\|s_{i,k}\|\le \lambda_{i,k}$ for all $i$, hence
		\[
		\frac1N\sum_{i=1}^N\lambda_{i,k}\ge \frac{M}{N}\sum_{i=1}^N\|s_{i,k}\|\ge M\Big\|\frac1N\sum_{i=1}^N s_{i,k}\Big\|=M\|\bar s_k\|,
		\]
		using $\|\frac1N\sum_{i=1}^N s_i\|\le \frac1N\sum_{i=1}^N\|s_i\|$.
		Since $M\ge L_2$, $L_2\|\bar s_k\|\le \frac1N\sum_{i=1}^N\lambda_{i,k}$.
		Finally $\tilde{\bar\lambda}_k=\frac1N\sum_{i=1}^N(\lambda_{i,k}+\delta_{i,k})\ge \frac1N\sum_{i=1}^N\lambda_{i,k}$.
	\end{proof}
	
	\begin{proof}[Proof of Lemma~\ref{lem:bridge-grad}]
		By Lemma~\ref{lem:avg-track-g}, $\tilde{\bar g}_k=\frac1N\sum_{i=1}^N\nabla f_i(x_{i,k})$.
		Therefore
		\begin{align*}
			\|\nabla f(\bar x_k)\!-\!\tilde{\bar g}_k\|
			&=\Big\|\tfrac1N\textstyle\sum_{i=1}^N\!\big(\nabla f_i(\bar x_k)\!-\!\nabla f_i(x_{i,k})\big)\Big\|\\
			&\le \tfrac1N\textstyle\sum_{i=1}^N \|\nabla f_i(\bar x_k)\!-\!\nabla f_i(x_{i,k})\|.
		\end{align*}
		The $L_1$-Lipschitzness of $\nabla f_i$ implies
		$\|\nabla f_i(\bar x_k)-\nabla f_i(x_{i,k})\|\le L_1\|x_{i,k}-\bar x_k\|$.
		Averaging this inequality and applying Cauchy--Schwarz gives
		\[
		\frac1N\sum_{i=1}^N\|x_{i,k}-\bar x_k\|
		\le \sqrt{\frac1N\sum_{i=1}^N\|x_{i,k}-\bar x_k\|^2}=D(X_k).
		\]
		Combining this bound with the preceding display yields the claim.
	\end{proof}
	
	\begin{proof}[Proof of Lemma~\ref{lem:bridge-hess}]
		By Lemma~\ref{lem:avg-track-h}, $\tilde{\bar H}_k=\frac1N\sum_{i=1}^N \nabla^2 f_i(x_{i,k})$.
		Thus
		\[
		\nabla^2 f(\bar x_k)-\tilde{\bar H}_k
		=\frac1N\sum_{i=1}^N\big(\nabla^2 f_i(\bar x_k)-\nabla^2 f_i(x_{i,k})\big),
		\]
		and so
		\begin{align*}
			&\|\nabla^2 f(\bar x_k)\!-\!\tilde{\bar H}_k\|_2\\
			&\;\le \tfrac1N\textstyle\sum_{i=1}^N\|\nabla^2 f_i(\bar x_k)\!-\!\nabla^2 f_i(x_{i,k})\|_2\\
			&\;\le \tfrac{L_2}{N}\textstyle\sum_{i=1}^N\|x_{i,k}\!-\!\bar x_k\|.
		\end{align*}
		Cauchy--Schwarz applied to the last average gives
		$\|\nabla^2 f(\bar x_k)-\tilde{\bar H}_k\|_2 \le L_2 D(X_k)$.
	\end{proof}
	
	\begin{proof}[Proof of Lemma~\ref{lem:delta_control}]
		Let $A:=\tilde H_{i,k}$ and $B:=\nabla^2 f(\bar x_k)\succeq0$.
		By Weyl's inequality,
		\[
		\lambda_{\min}(A)\ge \lambda_{\min}(B)-\|A-B\|_2\ge -\|A-B\|_2.
		\]
		Therefore $-\lambda_{\min}(A)\le \|A-B\|_2$, and since $\delta_{i,k}=\max\{0,-\lambda_{\min}(A)\}$,
		we obtain $0\le \delta_{i,k}\le \|A-B\|_2$, proving the first claim.
		For the second,
		\[
		\|A-B\|_2\le \|A-\tilde{\bar H}_k\|_2+\|\tilde{\bar H}_k-B\|_2.
		\]
		Averaging over $i$ gives
		$\bar\delta_k\le \Delta_k^H+\|\tilde{\bar H}_k-\nabla^2 f(\bar x_k)\|_2$,
		and Lemma~\ref{lem:bridge-hess} gives the stated bound.
	\end{proof}
	
	\begin{proof}[Proof of Lemma~\ref{lem:step-disp}]
		Fix $k$ and abbreviate $A_i:=A_{i,k}$, $A:=A_k^{\mathrm{ref}}$, $\tilde g_i:=\tilde g_{i,k}$, $\tilde{\bar g}:=\tilde{\bar g}_k$.
		Then $s_i=-A_i^{-1}\tilde g_i$ and $s^{\mathrm{ref}}=-A^{-1}\tilde{\bar g}$.
		Decompose:
		\[
		s_i-s^{\mathrm{ref}}
		=-A_i^{-1}(\tilde g_i-\tilde{\bar g}) + (A^{-1}-A_i^{-1})\tilde{\bar g}.
		\]
		Since $A_i\succeq \tilde\lambda_{i,k}I\succeq \underline\lambda_k I$, we have $\|A_i^{-1}\|_2\le 1/\underline\lambda_k$.
		Thus
		\[
		\|s_i-s^{\mathrm{ref}}\|
		\le \frac1{\underline\lambda_k}\|\tilde g_i-\tilde{\bar g}\|+\|A^{-1}-A_i^{-1}\|_2\,\|\tilde{\bar g}\|.
		\]
		Use the resolvent identity $A^{-1}-A_i^{-1}=A_i^{-1}(A_i-A)A^{-1}$ to get
		\[
		\|A^{-1}-A_i^{-1}\|_2\le \|A_i^{-1}\|_2\,\|A_i-A\|_2\,\|A^{-1}\|_2
		\le \frac1{\underline\lambda_k}\cdot \frac1{\tilde{\bar\lambda}_k}\,\|A_i-A\|_2,
		\]
		since $A\succeq \tilde{\bar\lambda}_k I$ implies $\|A^{-1}\|_2\le 1/\tilde{\bar\lambda}_k$.
		Furthermore,
		\[
		A_i-A=\big(\tilde H_{i,k}-\tilde{\bar H}_k\big)+(\tilde\lambda_{i,k}-\tilde{\bar\lambda}_k)I,
		\]
		so $\|A_i-A\|_2\le \|\tilde H_{i,k}-\tilde{\bar H}_k\|_2+|\tilde\lambda_{i,k}-\tilde{\bar\lambda}_k|$.
		Combining the resolvent bound with the estimate on $\|A_i-A\|_2$,
		and then averaging over $i$ using
		$\|\bar s-s^{\mathrm{ref}}\|\le \frac1N\sum_{i=1}^N\|s_i-s^{\mathrm{ref}}\|$,
		gives the claim.
	\end{proof}
	
	\begin{proof}[Proof of Lemma~\ref{lem:lambda_dispersion}]
		Write $\tilde\lambda_{i,k}-\tilde{\bar\lambda}_k=(\lambda_{i,k}-\bar\lambda_k)+(\delta_{i,k}-\bar\delta_k)$,
		where $\bar\lambda_k:=\frac1N\sum_{i=1}^N\lambda_{i,k}$.
		Then
		\[
		\Delta_k^\lambda
		=\frac1N\sum_{i=1}^N|\tilde\lambda_{i,k}-\tilde{\bar\lambda}_k|
		\le \underbrace{\frac1N\sum_{i=1}^N|\lambda_{i,k}-\bar\lambda_k|}_{=:T_1}
		+\underbrace{\frac1N\sum_{i=1}^N|\delta_{i,k}-\bar\delta_k|}_{=:T_2}.
		\]
		\textbf{Step 1: bound $T_2$.}
		Since $\delta_{i,k}\ge0$, a standard counting argument yields
		\[
		\sum_{i=1}^N|\delta_{i,k}-\bar\delta_k|
		=2\sum_{i:\delta_{i,k}>\bar\delta_k}(\delta_{i,k}-\bar\delta_k)
		\le 2\sum_{i:\delta_{i,k}>\bar\delta_k}\delta_{i,k}
		\le 2\sum_{i=1}^N\delta_{i,k}=2N\bar\delta_k,
		\]
		so $T_2\le 2\bar\delta_k$.
		
		\textbf{Step 2: bound $T_1$.}
		For any scalars $\{u_i\}$ with $\bar u=\frac1N\sum_{i=1}^N u_i$,
		$|u_i-\bar u|=\big|\frac1N\sum_{j=1}^N(u_i-u_j)\big|\le \frac1N\sum_{j=1}^N|u_i-u_j|$.
		Averaging over $i$ gives
		\[
		T_1\le \frac{1}{N^2}\sum_{i=1}^N\sum_{j=1}^N|\lambda_{i,k}-\lambda_{j,k}|.
		\]
		Let $a_i:=\|\tilde g_{i,k}\|$. Then
		\[
		|\lambda_{i,k}-\lambda_{j,k}|
		=\sqrt{M}\,|\sqrt{a_i}-\sqrt{a_j}|
		=\sqrt{M}\,\frac{|a_i-a_j|}{\sqrt{a_i}+\sqrt{a_j}}.
		\]
		Moreover $|a_i-a_j|\le \|\tilde g_{i,k}-\tilde g_{j,k}\|$.
		By triangle inequality,
		$\|\tilde g_{i,k}-\tilde g_{j,k}\|\le \|\tilde g_{i,k}-\tilde{\bar g}_k\|+\|\tilde g_{j,k}-\tilde{\bar g}_k\|$,
		hence
		\[
		\frac1{N^2}\sum_{i=1}^N\sum_{j=1}^N\|\tilde g_{i,k}-\tilde g_{j,k}\|
		\le \frac1{N^2}\sum_{i=1}^N\sum_{j=1}^N\big(\|\tilde g_{i,k}-\tilde{\bar g}_k\|+\|\tilde g_{j,k}-\tilde{\bar g}_k\|\big)
		=2\Delta_k^g.
		\]
		Under \eqref{eq:rel_grad_disp}, $a_i\ge \|\tilde{\bar g}_k\|-\|\tilde g_{i,k}-\tilde{\bar g}_k\|\ge (1-\alpha_d)\|\tilde{\bar g}_k\|$ for all $i$,
		so $\sqrt{a_i}+\sqrt{a_j}\ge 2\sqrt{(1-\alpha_d)\|\tilde{\bar g}_k\|}$.
		Therefore
		\[
		T_1\le \sqrt{M}\cdot \frac{2\Delta_k^g}{2\sqrt{(1-\alpha_d)\|\tilde{\bar g}_k\|}}
		=\frac{\sqrt{M}}{\sqrt{1-\alpha_d}}\cdot \frac{\Delta_k^g}{\sqrt{\|\tilde{\bar g}_k\|}}.
		\]
		The displayed estimate for $T_1$ and the Step~1 bound
		$T_2\le 2\bar\delta_k$ give the stated bound on
		$\Delta_k^\lambda$.
	\end{proof}
	
	\begin{proof}[Proof of Proposition~\ref{prop:inexact-rn}]
		Write
		\[
		r_k=(\nabla^2 f(\bar x_k)+\lambda_k I)\bar s_k+g_k.
		\]
		Add and subtract the reference step:
		\[
		r_k=(\nabla^2 f(\bar x_k)+\lambda_k I)(\bar s_k-s_k^{\mathrm{ref}})
		+(\nabla^2 f(\bar x_k)+\lambda_k I)s_k^{\mathrm{ref}}+g_k.
		\]
		Using $(\tilde{\bar H}_k+\lambda_k I)s_k^{\mathrm{ref}}=-\tilde{\bar g}_k$, we have
		\[
		(\nabla^2 f(\bar x_k)+\lambda_k I)s_k^{\mathrm{ref}}+g_k
		=(\nabla^2 f(\bar x_k)-\tilde{\bar H}_k)s_k^{\mathrm{ref}}+(g_k-\tilde{\bar g}_k).
		\]
		Therefore
		\[
		r_k=(\nabla^2 f(\bar x_k)+\lambda_k I)(\bar s_k-s_k^{\mathrm{ref}})
		+(\nabla^2 f(\bar x_k)-\tilde{\bar H}_k)s_k^{\mathrm{ref}}+(g_k-\tilde{\bar g}_k).
		\]
		Take norms:
		\begin{align*}
			\|r_k\|
			&\le \|\nabla^2 f(\bar x_k)\!+\!\lambda_k I\|_2\|\bar s_k\!-\!s_k^{\mathrm{ref}}\|\\
			&\quad+\|\nabla^2 f(\bar x_k)\!-\!\tilde{\bar H}_k\|_2\|s_k^{\mathrm{ref}}\|
			+\|g_k\!-\!\tilde{\bar g}_k\|.
		\end{align*}
		By the global Hessian bound~\eqref{eq:MHmax_def},
		$\|\nabla^2 f(\bar x_k)\|_2\le M_{H,\max}$, hence
		$\|\nabla^2 f(\bar x_k)+\lambda_k I\|_2\le M_{H,\max}+\lambda_k$.
		Item~1 of Proposition~\ref{prop:inexact-rn} implies
		\begin{align*}
			&\|\nabla^2 f(\bar x_k)\!+\!\lambda_k I\|_2\,\|\bar s_k\!-\!s_k^{\mathrm{ref}}\|\\
			&\;\le (M_{H,\max}\!+\!\lambda_k)\!\cdot\!\tfrac{\eta}{8}\!\cdot\!\tfrac{\lambda_k}{M_{H,\max}+\lambda_k}\|\bar s_k\|
			=\tfrac{\eta}{8}\lambda_k\|\bar s_k\|.
		\end{align*}
		Next, item (3) and Lemma~\ref{lem:bridge-hess} imply
		$\|\nabla^2 f(\bar x_k)-\tilde{\bar H}_k\|_2\le \frac{\eta}{8}\lambda_k$.
		Moreover,
		\[
		\|s_k^{\mathrm{ref}}\|\le \|\bar s_k\|+\|\bar s_k-s_k^{\mathrm{ref}}\|
		\le \Big(1+\frac{\eta}{8}\Big)\|\bar s_k\|\le \frac98\|\bar s_k\|,
		\]
		since $\eta\le 1/12$ implies $1+\eta/8\le 9/8$.
		Therefore the second term is bounded by
		\[
		\|\nabla^2 f(\bar x_k)-\tilde{\bar H}_k\|_2\,\|s_k^{\mathrm{ref}}\|
		\le \frac{\eta}{8}\lambda_k\cdot \frac98\|\bar s_k\|=\frac{9\eta}{64}\lambda_k\|\bar s_k\|.
		\]
		Finally, item (2) and Lemma~\ref{lem:bridge-grad} imply $\|g_k-\tilde{\bar g}_k\|\le \frac{\eta}{8}\lambda_k\|\bar s_k\|$.
		Summing yields
		\[
		\|r_k\|\le \left(\frac{\eta}{8}+\frac{9\eta}{64}+\frac{\eta}{8}\right)\lambda_k\|\bar s_k\|
		=\frac{25\eta}{64}\lambda_k\|\bar s_k\|\le \eta\lambda_k\|\bar s_k\|.
		\]
		The inequality $L_2\|\bar s_k\|\le \lambda_k$ is Lemma~\ref{lem:avg-regime}.
	\end{proof}
	
	\begin{proof}[Proof of Lemma~\ref{lem:descent}]
		Apply Lemma~\ref{lem:taylor3} to $f$ at $(\bar x_k,\bar s_k)$:
		\[
		f(\bar x_k+\bar s_k)\le f(\bar x_k)+\ip{g_k}{\bar s_k}+\frac12\bar s_k^\top\nabla^2 f(\bar x_k)\bar s_k+\frac{L_2}{6}\|\bar s_k\|^3.
		\]
		Let $d_k:=(\nabla^2 f(\bar x_k)+\lambda_k I)\bar s_k+g_k$.
		Then $\ip{g_k}{\bar s_k}=-\bar s_k^\top\nabla^2 f(\bar x_k)\bar s_k-\lambda_k\|\bar s_k\|^2+\ip{d_k}{\bar s_k}$.
		Substitute:
		\[
		f(\bar x_{k+1})\le f(\bar x_k)-\frac12\bar s_k^\top\nabla^2 f(\bar x_k)\bar s_k-\lambda_k\|\bar s_k\|^2+\ip{d_k}{\bar s_k}+\frac{L_2}{6}\|\bar s_k\|^3.
		\]
		Since $\nabla^2 f(\bar x_k)\succeq 0$, drop the nonpositive term. Also
		$\ip{d_k}{\bar s_k}\le \|d_k\|\,\|\bar s_k\|\le \eta\lambda_k\|\bar s_k\|^2$
		and $L_2\|\bar s_k\|\le \lambda_k$ implies $\frac{L_2}{6}\|\bar s_k\|^3\le \frac16\lambda_k\|\bar s_k\|^2$.
		Combining the residual, cubic-remainder, and nonpositive-Hessian
		bounds yields the claim.
	\end{proof}
	
	\begin{proof}[Proof of Lemma~\ref{lem:grad_bound}]
		By the integral form of Taylor's theorem,
		\[
		\nabla f(\bar x_k+\bar s_k)=g_k+\nabla^2 f(\bar x_k)\bar s_k+e_k,
		\qquad \|e_k\|\le \frac{L_2}{2}\|\bar s_k\|^2.
		\]
		The definition $d_k=(\nabla^2 f(\bar x_k)+\lambda_k I)\bar s_k+g_k$ gives
		$g_k+\nabla^2 f(\bar x_k)\bar s_k=-\lambda_k\bar s_k+d_k$.
		Thus
		\[
		\|g_{k+1}\|\le \lambda_k\|\bar s_k\|+\|d_k\|+\|e_k\|
		\le \lambda_k\|\bar s_k\|+\eta\lambda_k\|\bar s_k\|+\frac{L_2}{2}\|\bar s_k\|^2.
		\]
		Finally $L_2\|\bar s_k\|\le \lambda_k$ implies $\frac{L_2}{2}\|\bar s_k\|^2\le \frac12\lambda_k\|\bar s_k\|$.
	\end{proof}
	
	\begin{lemma}\label{lem:lambda_scaling}
		If Assumptions~\ref{ass:problem} and~\ref{ass:bounded_iterates}
		hold and the schedule~\eqref{eq:log_schedule} is used, then the following is true.
		Fix any $0<\varepsilon\le1$ and let $K_0(\varepsilon)$ be as in Proposition~\ref{prop:burnin}.
		Then there exists a constant $C_\lambda>0$ such that for all $k\ge K_0(\varepsilon)$ with $\|g_k\|\ge \varepsilon$,
		\[
		\lambda_k\le C_\lambda\,\sqrt{\|g_k\|}.
		\]
	\end{lemma}
	\begin{proof}
		Recall $\lambda_k=\tilde{\bar\lambda}_k=\frac1N\sum_{i=1}^N(\lambda_{i,k}+\delta_{i,k})$ with
		$\lambda_{i,k}=\sqrt{M\|\tilde g_{i,k}\|}$ and $\delta_{i,k}\ge 0$.
		By Jensen's inequality (concavity of $x\mapsto\sqrt{x}$),
		\[
		\frac1N\sum_{i=1}^N \sqrt{\|\tilde g_{i,k}\|}
		\le \sqrt{\frac1N\sum_{i=1}^N \|\tilde g_{i,k}\|}.
		\]
		Moreover, by the triangle inequality and the definition of $\Delta_k^g$,
		\[
		\frac1N\sum_{i=1}^N \|\tilde g_{i,k}\|
		\le \|\tilde{\bar g}_k\|+\Delta_k^g.
		\]
		By Lemma~\ref{lem:bridge-grad}, $\|\tilde{\bar g}_k-g_k\|\le L_1D(X_k)$.
		Using Item~2 of Proposition~\ref{prop:inexact-rn} (which holds for all $k\ge K_0(\varepsilon)$ with $\|g_k\|\ge \varepsilon$),
		together with
		$\|(\nabla^2 f(\bar x_k)+\lambda_k I)\bar s_k\|
		=\|g_k-r_k\|\ge \lambda_k\|\bar s_k\|$
		and $\|r_k\|\le \eta\lambda_k\|\bar s_k\|$,
		we obtain $(1-\eta)\lambda_k\|\bar s_k\|\le \|g_k\|$ and hence
		\[
		L_1D(X_k)\le \frac{\eta}{8}\lambda_k\|\bar s_k\|\le \frac{\eta}{8(1-\eta)}\|g_k\|.
		\]
		Therefore,
		\[
		\|\tilde{\bar g}_k\|\le \|g_k\|+L_1D(X_k)\le \Big(1+\frac{\eta}{8(1-\eta)}\Big)\|g_k\|.
		\]
		In addition, Proposition~\ref{prop:burnin} ensures
		$\Delta_k^g\le \varepsilon^2\le \varepsilon\le \|g_k\|$
		for all $k\ge K_0(\varepsilon)$ with $\|g_k\|\ge \varepsilon$.
		Thus,
		\[
		\frac1N\sum_{i=1}^N \|\tilde g_{i,k}\|
		\le \Big(2+\frac{\eta}{8(1-\eta)}\Big)\|g_k\|.
		\]
		Hence,
		\[
		\lambda_k
		= \frac1N\sum_{i=1}^N\big(\sqrt{M\|\tilde g_{i,k}\|}+\delta_{i,k}\big)
		\le \sqrt{M}\sqrt{\frac1N\sum_{i=1}^N \|\tilde g_{i,k}\|}+\bar\delta_k
		\le C_0\sqrt{\|g_k\|}+\bar\delta_k,
		\]
		with $C_0:=\sqrt{M\big(2+\frac{\eta}{8(1-\eta)}\big)}$.
		Finally, Proposition~\ref{prop:burnin} also ensures
		$\bar\delta_k\le (1+L_2)\varepsilon^{3/2}$ on the same index set.
		Since $0<\varepsilon\le1$ and $\|g_k\|\ge\varepsilon$, we have
		$\varepsilon^{3/2}\le \sqrt{\varepsilon}\le \sqrt{\|g_k\|}$.
		Absorbing this term into the constant yields
		$\lambda_k\le (C_0+1+L_2)\sqrt{\|g_k\|}$, completing the proof.
	\end{proof}
	
	\begin{proof}[Proof of Lemma~\ref{lem:steady}]
		Assume $\eta\le 1/12$, so $c_d:=1-\eta-\frac16\ge \frac34$ in Lemma~\ref{lem:descent}.
		Also from Lemma~\ref{lem:grad_bound}, $C_g:=1+\eta+\frac12\le 2$.
		For $k\in\mathcal I$, $\|g_{k+1}\|\ge \frac14\|g_k\|$ and $\|g_{k+1}\|\le C_g\lambda_k\|\bar s_k\|$ imply
		$\|\bar s_k\|\ge \|g_k\|/(4C_g\lambda_k)$.
		Then Lemma~\ref{lem:descent} yields
		\[
		\Phi_k-\Phi_{k+1}\ge c_d\lambda_k\|\bar s_k\|^2
		\ge c_d\lambda_k\left(\frac{\|g_k\|}{4C_g\lambda_k}\right)^2
		=\frac{c_d}{16C_g^2}\frac{\|g_k\|^2}{\lambda_k}.
		\]
		By the additional assumption in Lemma~\ref{lem:steady}, we have $\lambda_k\le C_\lambda\sqrt{\|g_k\|}$.
		Therefore,
		\[
		\frac{\|g_k\|^2}{\lambda_k}\ge \frac{1}{C_\lambda}\|g_k\|^{3/2}.
		\]
		Finally, convexity and the bounded level set imply $\Phi_k\le \ip{g_k}{\bar x_k-x_\star}\le D\|g_k\|$, hence
		$\|g_k\|^{3/2}\ge D^{-3/2}\Phi_k^{3/2}$.
		Collect constants into $\nu$.
	\end{proof}
	
	\begin{proof}[Proof of Lemma~\ref{lem:solve32}]
		If $\Phi_k=0$ for some $k$, then $\Phi_{k'}=0$ for all $k'\ge k$ and the claim is trivial.
		Assume $\Phi_k>0$ for all $k$.
		
		From $\Phi_{k+1}\le \Phi_k-\nu \Phi_k^{3/2}=\Phi_k(1-\nu\sqrt{\Phi_k})$, we consider two cases.
		
		Case 1: $\nu\sqrt{\Phi_k}\ge 1$.
		Then $\Phi_{k+1}\le 0$. Since $\Phi_{k+1}\ge 0$, we must have $\Phi_{k+1}=0$,
		contradicting $\Phi_{k+1}>0$.
		Hence this case cannot occur under the assumption $\Phi_k>0$.
		
		Therefore, necessarily $\nu\sqrt{\Phi_k}\in(0,1)$ for all $k$.
		
		Let $\psi_k:=\Phi_k^{-1/2}$. Then
		\begin{align*}
			\psi_{k+1}&=\Phi_{k+1}^{-1/2}\ge \Phi_k^{-1/2}(1\!-\!\nu\sqrt{\Phi_k})^{-1/2}\\
			&=\psi_k(1\!-\!u_k)^{-1/2},\quad u_k:=\nu\sqrt{\Phi_k}\in(0,1).
		\end{align*}
		Using convexity of $(1-u)^{-1/2}$ on $[0,1)$ and $(1-u)^{-1/2}\ge 1+\frac12u$, we obtain
		\[
		\psi_{k+1}\ge \psi_k\Big(1+\frac12 u_k\Big)=\psi_k+\frac{\nu}{2}.
		\]
		The positive-sequence case also implies
		$\nu\sqrt{\Phi_0}<1$, hence $\psi_0=\Phi_0^{-1/2}>\nu$.
		Therefore,
		\[
		\psi_k\ge \psi_0+\frac{\nu}{2}k
		> \nu+\frac{\nu}{2}k
		= \frac{\nu}{2}(k+2),
		\]
		and hence
		\[
		\Phi_k=\psi_k^{-2}\le \frac{4}{\nu^2(k+2)^2}.
		\]
	\end{proof}
	
	\begin{proof}[Proof of Theorem~\ref{thm:main}]
		We establish the $\cO(\varepsilon^{-1})$ post-burn-in rate
		by deriving the $3/2$-recursion only along the pre-hitting tail.
		Concretely, fix $\eta\le 1/12$ and let $K_0=K_0(\varepsilon)$ be
		the burn-in index from Proposition~\ref{prop:burnin}, so that
		Proposition~\ref{prop:inexact-rn} holds for every shifted index
		before the first hitting time.
		Define $\Phi_k:=f(\bar x_k)-f_\star$ and $g_k:=\nabla f(\bar x_k)$.
		Set
		\[
		\hat x_j:=\bar x_{K_0+j},\qquad
		\hat g_j:=g_{K_0+j},\qquad
		\hat\Phi_j:=\Phi_{K_0+j},\qquad j\ge0.
		\]
		Let
		\[
		J_\varepsilon:=\inf\{j\ge0:\ \|\hat g_j\|\le\varepsilon\},
		\]
		with $J_\varepsilon=\infty$ if the set is empty.  All
		post-burn-in estimates below are applied only for shifted indices
		$0\le j<J_\varepsilon$.  For notational economy, the shifted
		sequence is relabeled as $(\bar x_k,g_k,\Phi_k)$.
		
		\smallskip
		\noindent\textbf{Step 0: descent along the shifted sequence.}
		By Proposition~\ref{prop:burnin}, the hypotheses of
		Lemma~\ref{lem:descent} hold at every shifted index before the
		first hitting time.
		Hence, with $c_d:=1-\eta-\frac16>0$,
		\begin{equation}\label{eq:Phi_monotone}
			\hat\Phi_{j+1}\le \hat\Phi_j
			- c_d\,\lambda_{K_0+j}\|\bar s_{K_0+j}\|^2
			\le \hat\Phi_j,
			\qquad 0\le j<J_\varepsilon.
		\end{equation}
		Since we seek the first shifted index $j$ with
		$\|\hat g_j\|\le\varepsilon$, this is exactly the range needed for
		the complexity bound.
		
		\smallskip
		\noindent\textbf{Step 1: a uniform gradient growth bound.}
		Let $r_k:=(\nabla^2 f(\bar x_k)+\lambda_k I)\bar s_k+g_k$ as defined in the paper.
		For every shifted index before the first hitting time,
		Proposition~\ref{prop:inexact-rn} gives
		$\|r_k\|\le \eta\lambda_k\|\bar s_k\|$.
		Moreover,
		\[
		\lambda_k\|\bar s_k\|
		\le \|(\nabla^2 f(\bar x_k)+\lambda_k I)\bar s_k\|
		=\|g_k-r_k\|
		\le \|g_k\|+\|r_k\|
		\le \|g_k\|+\eta\lambda_k\|\bar s_k\|.
		\]
		Rearranging yields
		\begin{equation}\label{eq:lambda_s_upper}
			\lambda_k\|\bar s_k\|\le \frac{1}{1-\eta}\|g_k\|,\qquad 0\le k<J_\varepsilon.
		\end{equation}
		By Lemma~\ref{lem:grad_bound}, again only at shifted indices
		before the first hitting time,
		\[
		\|g_{k+1}\|\le C_g\,\lambda_k\|\bar s_k\|,
		\qquad
		C_g:=1+\eta+\frac12.
		\]
		Combining with \eqref{eq:lambda_s_upper} gives the uniform growth bound
		\begin{equation}\label{eq:grad_growth}
			\|g_{k+1}\|\le \kappa_{\mathrm{inc}}\|g_k\|,
			\qquad
			\kappa_{\mathrm{inc}}:=\frac{C_g}{1-\eta}.
		\end{equation}
		Since $\eta\le 1/12$, we have $C_g\le 2$ and thus $\kappa_{\mathrm{inc}}\le 24/11<4$.
		
		\smallskip
		\noindent\textbf{Step 2: decay along the steady-step subsequence.}
		Recall the index sets
		\[
		\mathcal I:=\{k\ge 0:\ \|g_{k+1}\|\ge \tfrac14\|g_k\|\},
		\qquad
		\mathcal S:=\{k\ge 0:\ \|g_{k+1}\|< \tfrac14\|g_k\|\}.
		\]
		Fix $k<J_\varepsilon$.
		Let $n_k:=|\mathcal I\cap\{0,1,\dots,k-1\}|$ be the number of steady indices among the first $k$ iterations.
		List these indices increasingly as $0\le i_0<i_1<\cdots<i_{n_k-1}\le k-1$ (if $n_k=0$ skip this step).
		Define the subsequence $u_j:=\Phi_{i_j}$.
		By Lemma~\ref{lem:steady}, for each steady index $i_j\in\mathcal I$,
		\[
		\Phi_{i_j}-\Phi_{i_j+1}\ge \nu\,\Phi_{i_j}^{3/2}=\nu\,u_j^{3/2}.
		\]
		Since \(i_{j+1}\ge i_j+1\) and \(\Phi_k\) is nonincreasing
		along the indices before termination, we have
		\[
		u_{j+1}
		=\Phi_{i_{j+1}}
		\le \Phi_{i_j+1}
		\le \Phi_{i_j}-\nu\,\Phi_{i_j}^{3/2}
		= u_j-\nu\,u_j^{3/2}.
		\]
		Applying Lemma~\ref{lem:solve32} to this pre-hitting subsequence yields
		\begin{equation}\label{eq:u_bound}
			u_j\le \frac{4}{\nu^2(j+2)^2}
			\quad\text{whenever } i_j< J_\varepsilon.
		\end{equation}
		\smallskip
		\noindent\textbf{Step 3: two-case bound yielding $\cO(\varepsilon^{-1})$ iteration complexity.}
		Let $m_k:=|\mathcal S\cap\{0,1,\dots,k-1\}|=k-n_k$.
		We consider two cases.
		
		Case 1: $n_k\ge k/2$.
		For $k\ge2$, this implies $n_k\ge1$, so the last steady index
		$i_{n_k-1}$ is well defined. By monotonicity,
		$\Phi_k\le \Phi_{i_{n_k-1}}=u_{n_k-1}$, and \eqref{eq:u_bound} gives
		\[
		\Phi_k\le \frac{4}{\nu^2(n_k+1)^2}\le \frac{16}{\nu^2(k+2)^2}.
		\]
		The finitely many cases $k<2$ are absorbed into the final constant.
		
		Case 2: $n_k< k/2$ (hence $m_k>k/2$).
		For every sharp index $\ell\in\mathcal S$, we have $\|g_{\ell+1}\|\le \frac14\|g_\ell\|$ by definition.
		For every other index, we have the growth bound \eqref{eq:grad_growth}.
		Therefore, after $k$ iterations,
		\[
		\|g_k\|
		\le \kappa_{\mathrm{inc}}^{\,n_k}\left(\frac14\right)^{m_k}\|g_0\|
		\le \kappa_{\mathrm{inc}}^{\,k/2}\left(\frac14\right)^{k/2}\|g_0\|
		=\left(\frac{\kappa_{\mathrm{inc}}}{4}\right)^{k/2}\|g_0\|.
		\]
		Since $\kappa_{\mathrm{inc}}/4<1$, the bounded level set and convexity imply
		\[
		\Phi_k=f(\bar x_k)-f_\star
		\le \ip{g_k}{\bar x_k-x_\star}
		\le \|\bar x_k-x_\star\|\,\|g_k\|
		\le D\,\|g_k\|.
		\]
		Therefore,
		\[
		\Phi_k\le D\,G_\star\left(\frac{\kappa_{\mathrm{inc}}}{4}\right)^{k/2},
		\qquad
		G_\star:=\sup_{\ell\ge0}\|g_\ell\|<\infty.
		\]
		The finiteness of $G_\star$ follows from
		Assumption~\ref{ass:bounded_iterates} and the Lipschitz continuity
		of $\nabla f$ on the bounded trajectory.
		Because $\sup_{k\ge 0}(k+2)^2\left(\frac{\kappa_{\mathrm{inc}}}{4}\right)^{k/2}<\infty$, there exists a finite constant $C_{\exp}>0$ such that
		\[
		\Phi_k\le \frac{C_{\exp}}{(k+2)^2}\qquad 0\le k<J_\varepsilon.
		\]
		
		Combining both cases, there exists $C_\Phi>0$ such that
		\[
		\Phi_k\le \frac{C_\Phi}{(k+2)^2}\qquad 0\le k<J_\varepsilon,
		\]
		which proves the pre-hitting $\Phi_k=\cO(1/k^2)$ estimate.
		
		\smallskip
		\noindent\textbf{Step 4: gradient rate and $\varepsilon$-complexity.}
		By Lemma~\ref{lem:gap-grad} with $h=f$ and $L=L_1$,
		\[
		\|g_k\|^2\le 2L_1\Phi_k\le \frac{2L_1 C_\Phi}{(k+2)^2},
		\]
		hence $\|g_k\|\le \sqrt{2L_1 C_\Phi}/(k+2)=\cO(1/k)$.
		For every $j<J_\varepsilon$,
		$\|g_j\|\le \sqrt{2L_1 C_\Phi}/(j+2)$.
		Consequently, if
		\[
		j\ \ge\ \frac{\sqrt{2L_1 C_\Phi}}{\varepsilon}-2
		=\cO(\varepsilon^{-1}).
		\]
		then $j$ cannot remain before the first hitting time. Thus the
		first shifted hitting time satisfies this same $\cO(\varepsilon^{-1})$
		bound.
		
		\smallskip
		\noindent\textbf{Step 5: total complexity with burn-in.}
		Let $K(\varepsilon)$ be the first original iteration where
		$\|g_k\|\le \varepsilon$.
		Translating the shifted-index bound back to the original index,
		it takes at most $\cO(\varepsilon^{-1})$ additional steps after
		$K_0(\varepsilon)$. Thus, the total iteration complexity is bounded by:
		\[
		K(\varepsilon)\le K_0(\varepsilon)+C_K\,\varepsilon^{-1},
		\]
		for some constant $C_K>0$. This completes the proof of the global rate.
	\end{proof}
	
	\begin{proof}[Proof of Lemma~\ref{lem:delta_o_g}]
		From Lemma~\ref{lem:delta_control}, $\bar\delta_k\le \Delta_k^H+L_2D(X_k)$.
		
		\smallskip
		\noindent\textbf{Bound $\Delta_k^H$.}
		By Jensen's inequality and $\|\cdot\|_2\le \|\cdot\|_F$,
		\[
		\Delta_k^H\le D(\tilde{\mathcal H}_k).
		\]
		We apply the Hessian-tracker dispersion recursion (Lemma~\ref{lem:DH_rec}),
		noting that $\mathcal H_k$ is obtained from $\tilde{\mathcal H}_{k-1}$
		via the post-mixing update at iteration $k{-}1$:
		\begin{align*}
			D(\tilde{\mathcal H}_k)
			&\le \rho^{\tau_k}D(\mathcal H_k)\\
			&\le \rho^{\tau_k}\rho^{t_{k-1}}\bigl(D(\tilde{\mathcal H}_{k-1})
			\!+\!2\sqrt{d}L_2D\bigr)\\
			&\le \rho^{\tau_k}(B_H\!+\!2\sqrt{d}L_2D),
		\end{align*}
		where $B_H:=\sup_{j\ge 0}D(\tilde{\mathcal H}_j)<\infty$ is the uniform bound
		established in Lemma~\ref{lem:tilde_hess_bounded}.
		In particular, since $B_H$ is finite, we obtain
		\[
		\Delta_k^H\le \rho^{\tau_k}\big(B_H+2\sqrt{d}\,L_2D\big)
		=:\hat C_H\,\rho^{\tau_k}.
		\]
		
		\smallskip
		\noindent\textbf{Bound $D(X_k)$.}
		By Lemma~\ref{lem:DX_post},
		\[
		D(X_k)\le \rho^{t_{k-1}}\Big(D(\tilde X_{k-1})+\frac{\|S_{k-1}\|}{\sqrt N}\Big).
		\]
		Moreover, $D(\tilde X_{k-1})\le 2D$ and Lemma~\ref{lem:step_bound_app} gives $\|S_{k-1}\|/\sqrt N\le \sqrt{G_g/M}$.
		Thus
		\[
		D(X_k)\le \rho^{t_{k-1}}\Big(2D+\sqrt{\frac{G_g}{M}}\Big).
		\]
		
		\smallskip
		Therefore,
		\[
		\bar\delta_k\le \hat C_H\,\rho^{\tau_k}
		+L_2\rho^{t_{k-1}}\Big(2D+\sqrt{\frac{G_g}{M}}\Big)
		\le C_\delta\max\{\rho^{\tau_k},\rho^{t_{k-1}}\}.
		\]
		Using \eqref{eq:delta_schedule}, we obtain $\bar\delta_k\le \|g_k\|^{1+\gamma}$ for all sufficiently large $k$.
	\end{proof}
	
	\begin{proof}[Proof of Theorem~\ref{thm:superlinear}]
		Let $x_k:=\bar x_k$, $g_k:=\nabla f(x_k)$, and $H_k:=\nabla^2 f(x_k)$.
		Under Assumption~\ref{ass:strong}, for all $k$ we have $H_k\succeq \mu I$ on the level set.
		
		By definition, $r_k:=(H_k+\lambda_k I)\bar s_k+g_k$ and Proposition~\ref{prop:inexact-rn} gives for all large $k$
		\[
		(H_k+\lambda_k I)\bar s_k=-g_k+r_k,\qquad \|r_k\|\le \eta\lambda_k\|\bar s_k\|,\qquad L_2\|\bar s_k\|\le \lambda_k,
		\]
		with $\eta\in(0,1/12]$.
		Under \eqref{eq:local_accuracy} and the bound
		$D(X_k)\le \rho^{t_{k-1}}\bigl(2D+\sqrt{G_g/M}\bigr)$
		(see Lemma~\ref{lem:DX_post} and Lemma~\ref{lem:step_bound_app}),
		we have $D(X_k)=\cO(\|g_k\|^{1+\gamma})$.
		Moreover, Lemma~\ref{lem:delta_o_g} gives
		$\bar\delta_k=\cO(\|g_k\|^{1+\gamma})$.
		By the same Jensen argument as in Lemma~\ref{lem:lambda_scaling},
		\[
		\lambda_k
		\le \sqrt{M}\sqrt{\|\tilde{\bar g}_k\|+\Delta_k^g}
		+\bar\delta_k .
		\]
		The bridge bound gives
		$\|\tilde{\bar g}_k\|\le \|g_k\|+L_1D(X_k)$.
		The second condition in \eqref{eq:local_accuracy} gives
		$\Delta_k^g=\cO(\|g_k\|^{1+\gamma})$.
		Since the theorem assumes $\|g_k\|\to0$ along the local
		post-burn-in tail, the preceding estimates imply
		\[
		\lambda_k
		\le C\sqrt{\|g_k\|+\cO(\|g_k\|^{1+\gamma})}
		+\cO(\|g_k\|^{1+\gamma})
		=\cO(\sqrt{\|g_k\|}).
		\]
		Therefore $\lambda_k\to0$, and hence
		$\eta\lambda_k\le \mu/2$ for all sufficiently large $k$.
		
		Since $H_k+\lambda_k I\succeq \mu I$,
		\[
		\mu\|\bar s_k\|\le \|(H_k+\lambda_k I)\bar s_k\|=\|g_k-r_k\|\le \|g_k\|+\eta\lambda_k\|\bar s_k\|.
		\]
		Therefore $\|\bar s_k\|\le \frac{2}{\mu}\|g_k\|$ for all sufficiently large $k$.
		
		Using $g_{k+1}=g_k+H_k\bar s_k+e_k$ with $\|e_k\|\le \frac{L_2}{2}\|\bar s_k\|^2$ and $g_k+H_k\bar s_k=-\lambda_k\bar s_k+r_k$,
		\[
		\|g_{k+1}\|\le (1+\eta)\lambda_k\|\bar s_k\|+\frac{L_2}{2}\|\bar s_k\|^2
		\le \cO(\|g_k\|^{3/2})+\cO(\|g_k\|^2)=\cO(\|g_k\|^{3/2}),
		\]
		which proves the claim.
	\end{proof}

	\begin{proof}[Proof of Theorem~\ref{thm:comm}]
		Under \eqref{eq:log_schedule},
		$\tau_k=t_k=\lceil(p\log(k+2)+c_{\mathrm{mix}})/(-\log\rho)\rceil
		\le c_0+c_1(1-\rho)^{-1}\log(k+2)$,
		using $-1/\log\rho=\Theta((1-\rho)^{-1})$ as $\rho\to1$.
		Hence
		\[
		\sum_{k=0}^{K(\varepsilon)-1}(\tau_k+2t_k)
		=\cO\!\left((1-\rho)^{-1}K(\varepsilon)\log(K(\varepsilon)+2)\right).
		\]
		Combining with Theorem~\ref{thm:main} yields the claim.
	\end{proof}
	
	\section{Supplementary experiments}\label{app:supp_exp}
	
	This appendix collects additional experimental figures that
	support the main results in Section~\ref{sec:experiments}.
	All settings are identical to those described in
	Section~\ref{ssec:setup} unless stated otherwise.
	
	\subsection{Additional per-problem convergence curves}
	
	The relF-vs.-iteration panel for all nine functions is already shown
	in Figure~\ref{fig:relF_steps_all} of the main text.
	Below we collect the remaining metrics.
	
	\begin{figure}[htbp]
		\centering
		\includegraphics[width=\textwidth]{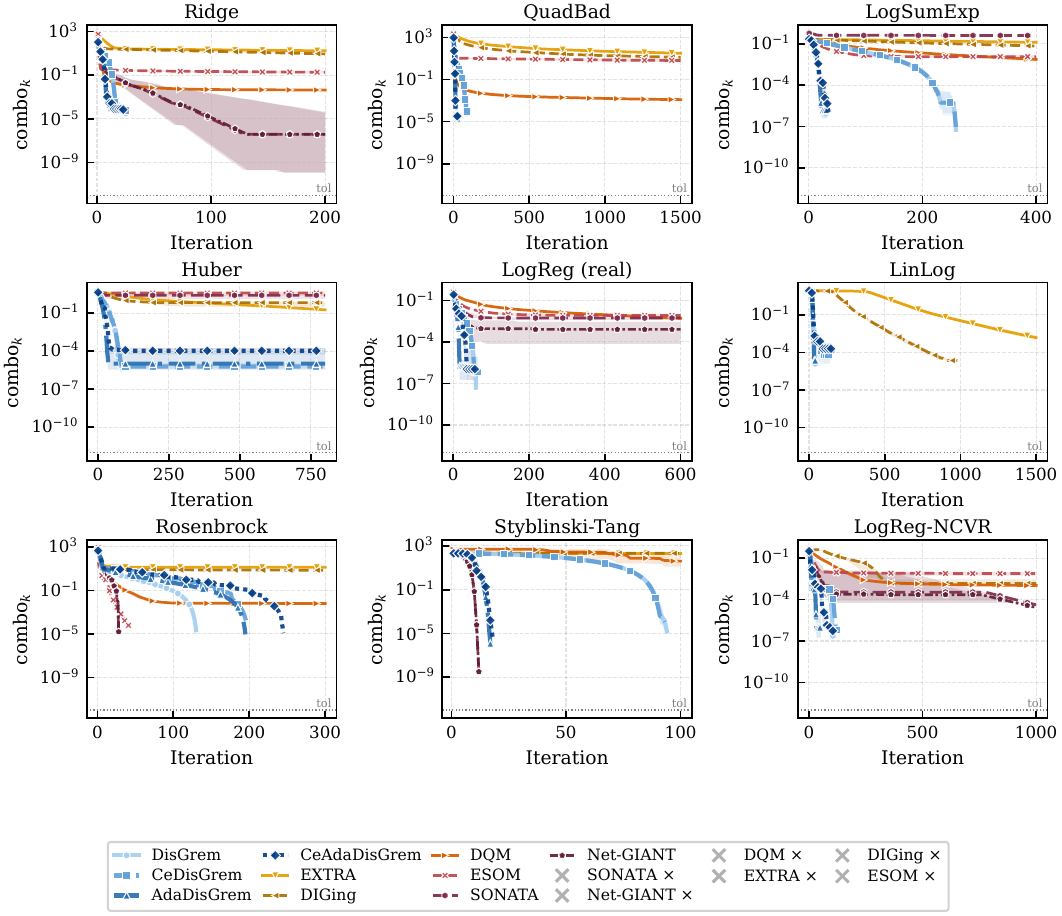}
		\caption{Composite optimality (combo $=\|\nabla f(\bar x_k)\|+\mathrm{cons}_k$)
			vs.\ iteration for all nine functions.
			Same setting as Figure~\ref{fig:relF_steps_all}; $\times$
			legend entries indicate runs terminated by NaN, overflow, or divergence.}
		\label{fig:app_combo_steps}
	\end{figure}
	
	\begin{figure}[htbp]
		\centering
		\includegraphics[width=\textwidth]{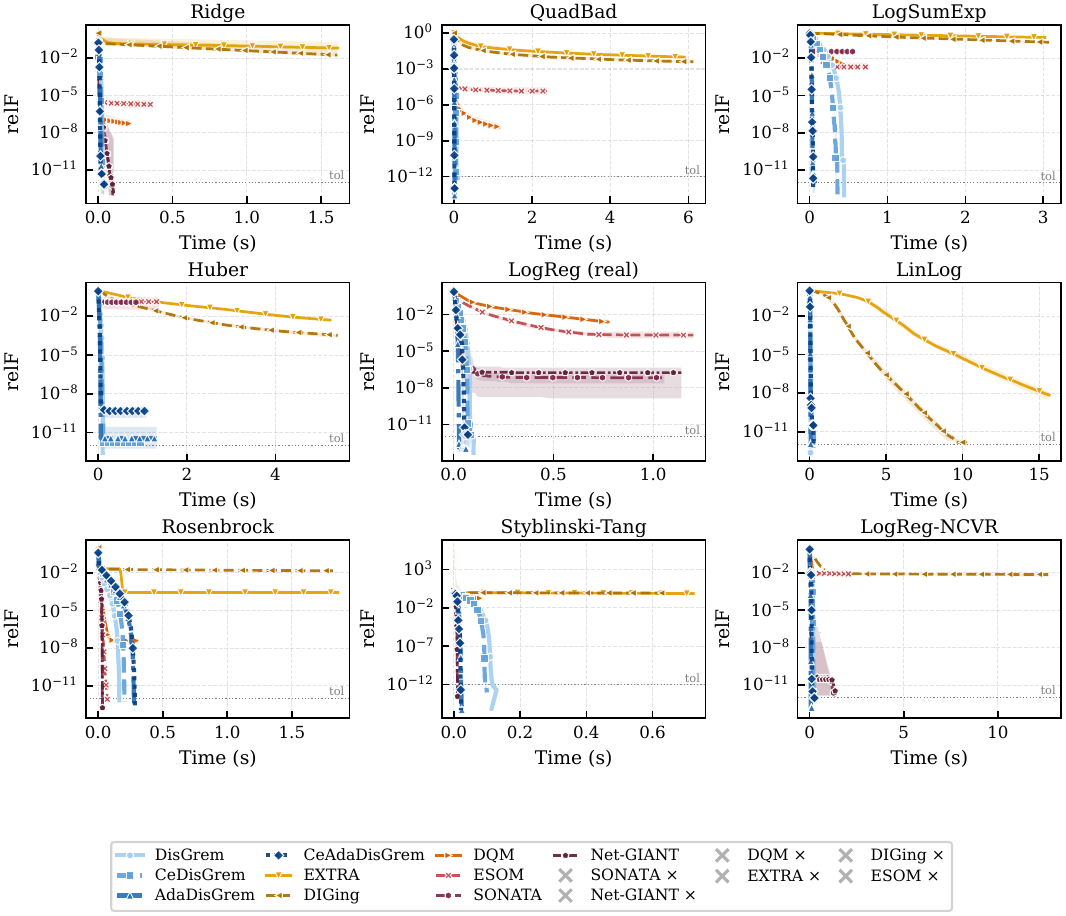}
		\caption{$\mathrm{relF}$ vs.\ wall-clock time (seconds) for all nine
			functions.}
		\label{fig:app_relF_time}
	\end{figure}
	
	\begin{figure}[htbp]
		\centering
		\includegraphics[width=\textwidth]{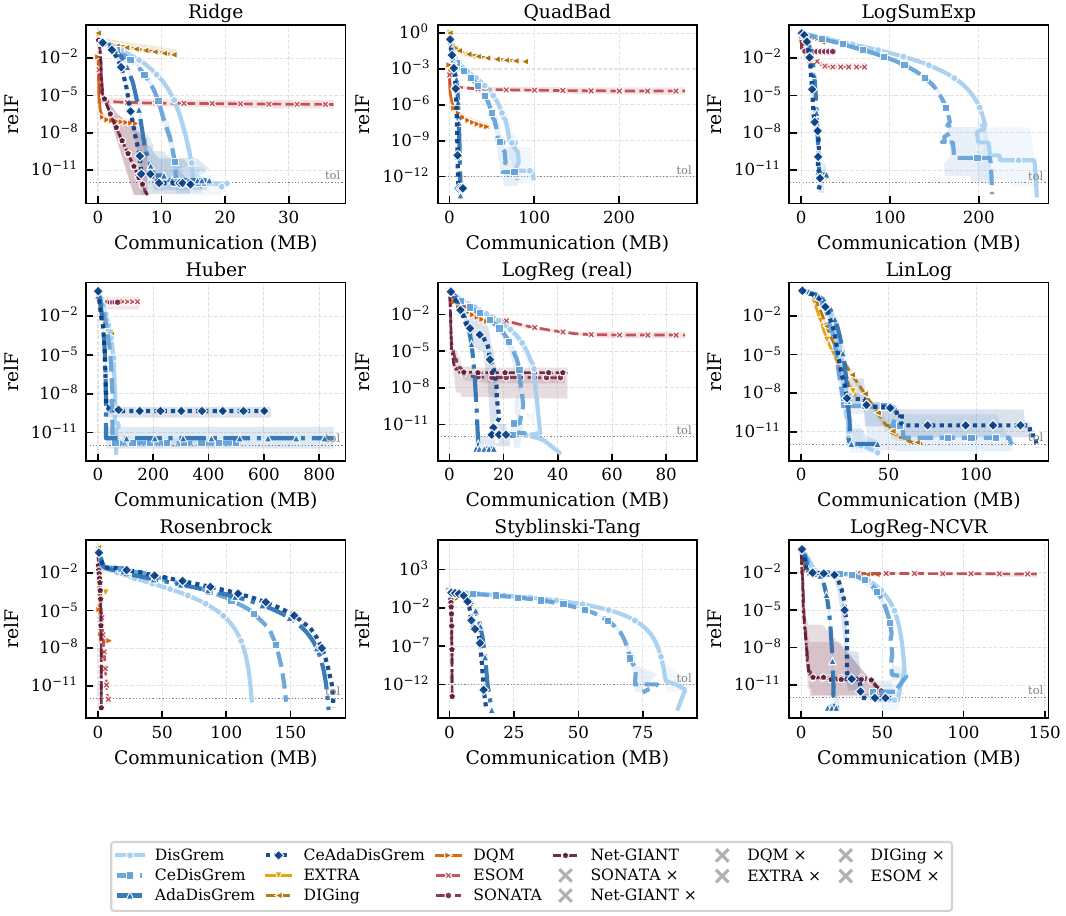}
		\caption{$\mathrm{relF}$ vs.\ cumulative communication cost (MB) for
			all nine functions.}
		\label{fig:app_relF_comm}
	\end{figure}
	
	\subsection{Communication ablation details}
	
	\begin{figure}[htbp]
		\centering
		\includegraphics[width=\textwidth]{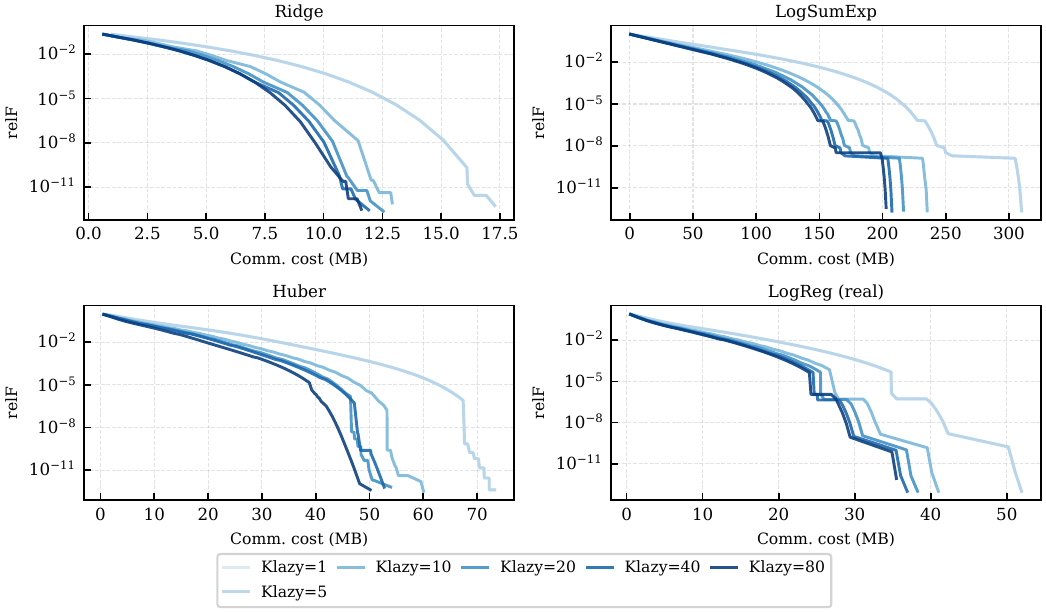}
		\caption{Effect of $K_{\mathrm{lazy}}$ on the
			communication--precision trade-off for \textsc{CeDisGrem}.
			Gradient-coloured curves from light to dark:
			$K_{\mathrm{lazy}}\in\{1,5,10,20,40,80\}$.}
		\label{fig:klazy}
	\end{figure}
	
	\begin{figure}[htbp]
		\centering
		\includegraphics[width=\textwidth]{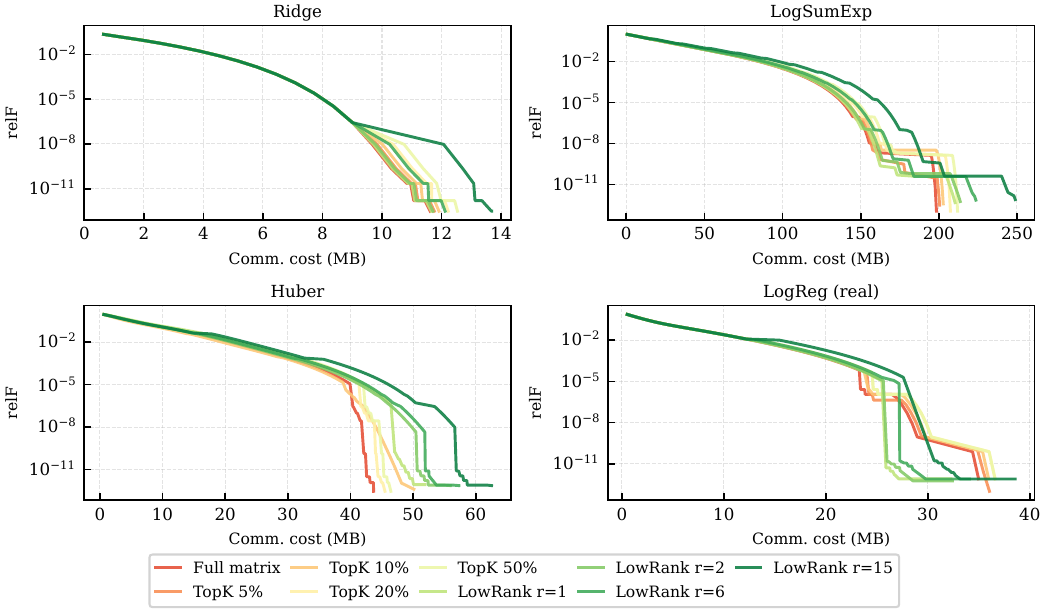}
		\caption{Effect of compression method and rank on the
			communication--precision trade-off for \textsc{CeDisGrem}.
			Configurations: full (no compression),
			Top-$k$ ($k\in\{5\%,10\%,20\%,50\%\}$),
			Low-Rank ($r\in\{1,2,d/5,d/2\}$).}
		\label{fig:compress}
	\end{figure}
	
	\subsection{Adaptive mechanism details}
	
	\begin{figure}[htbp]
		\centering
		\includegraphics[width=\textwidth]{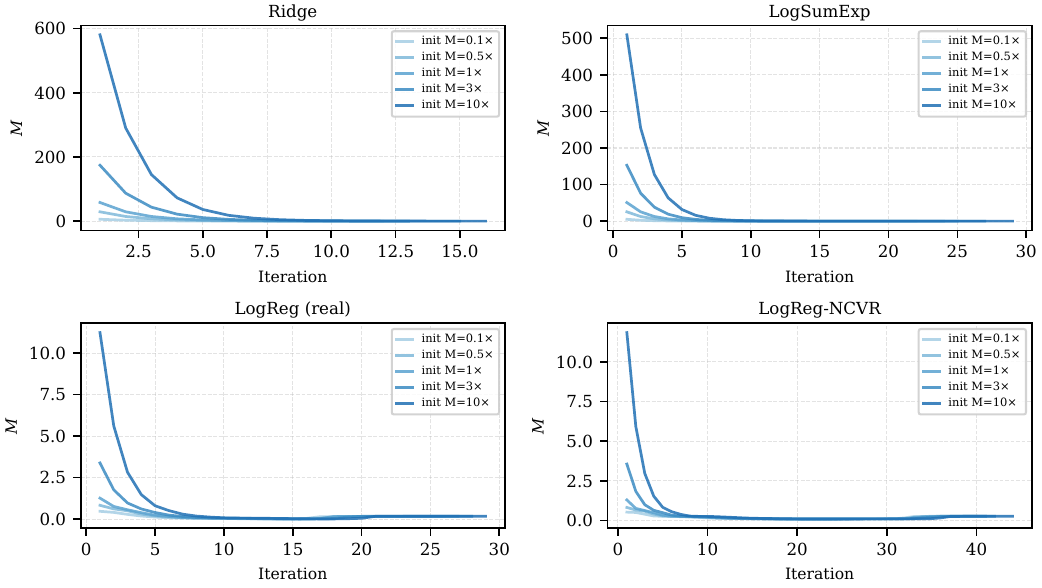}
		\caption{Robustness of \textsc{AdaDisGrem} to initial $M$.
			Initial $M_0\in\{0.1M^*,0.5M^*,M^*,3M^*,10M^*\}$.
			All initializations converge to similar trajectories within
			50--100 iterations.}
		\label{fig:ada_init_m}
	\end{figure}
	
	\subsection{Parameter sensitivity sweeps}
	
	\begin{figure}[htbp]
		\centering
		\includegraphics[width=\textwidth]{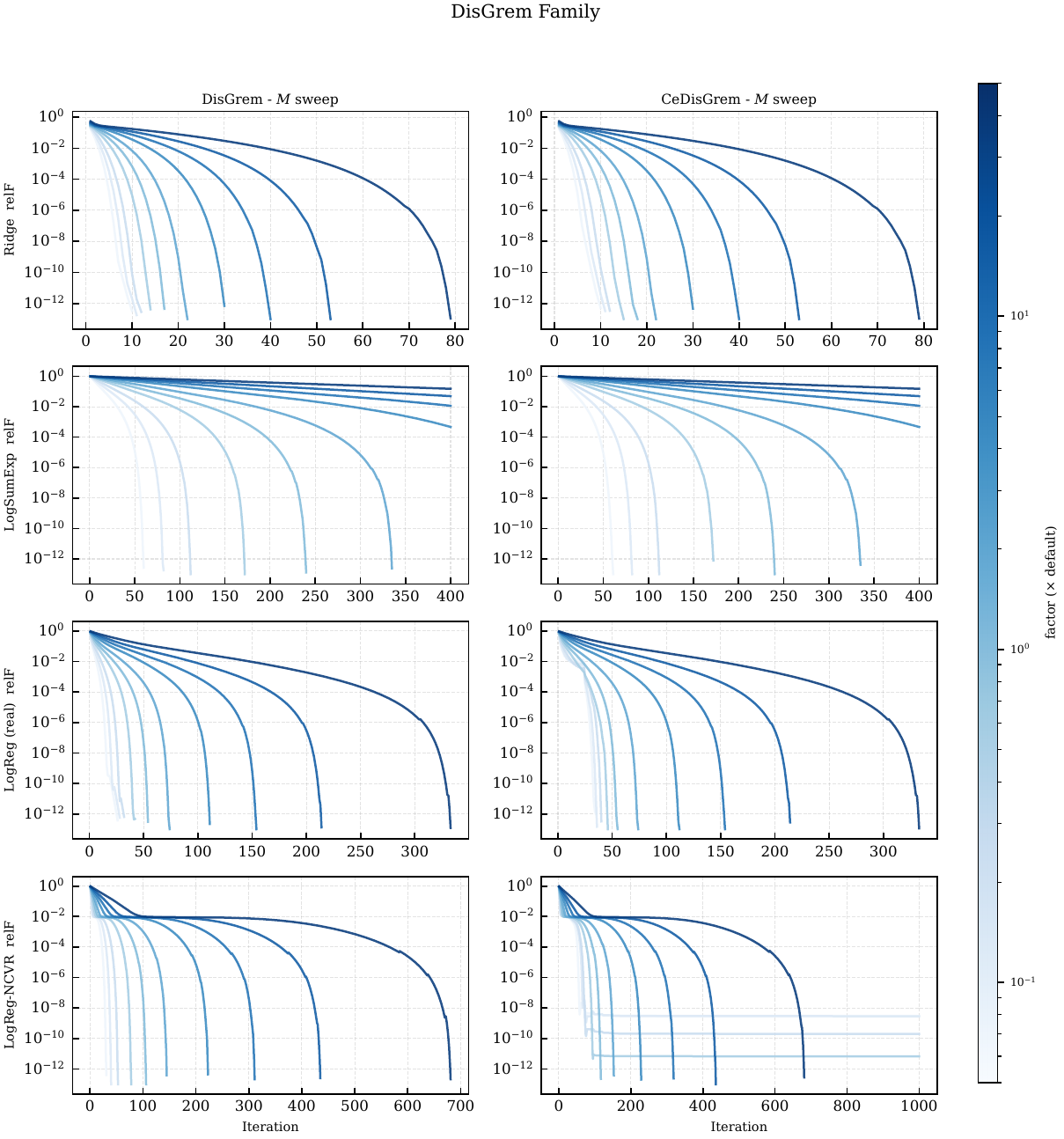}
		\caption{Parameter sweep for the \textsc{DisGrem} family
			on four functions.
			10 values of $M_{\mathrm{fac}}$ (light-to-dark for increasing
			$M_{\mathrm{fac}}$).
			Rows: functions; columns: algorithms.}
		\label{fig:sweep_disgrem}
	\end{figure}
	
	\begin{figure}[htbp]
		\centering
		\includegraphics[width=\textwidth]{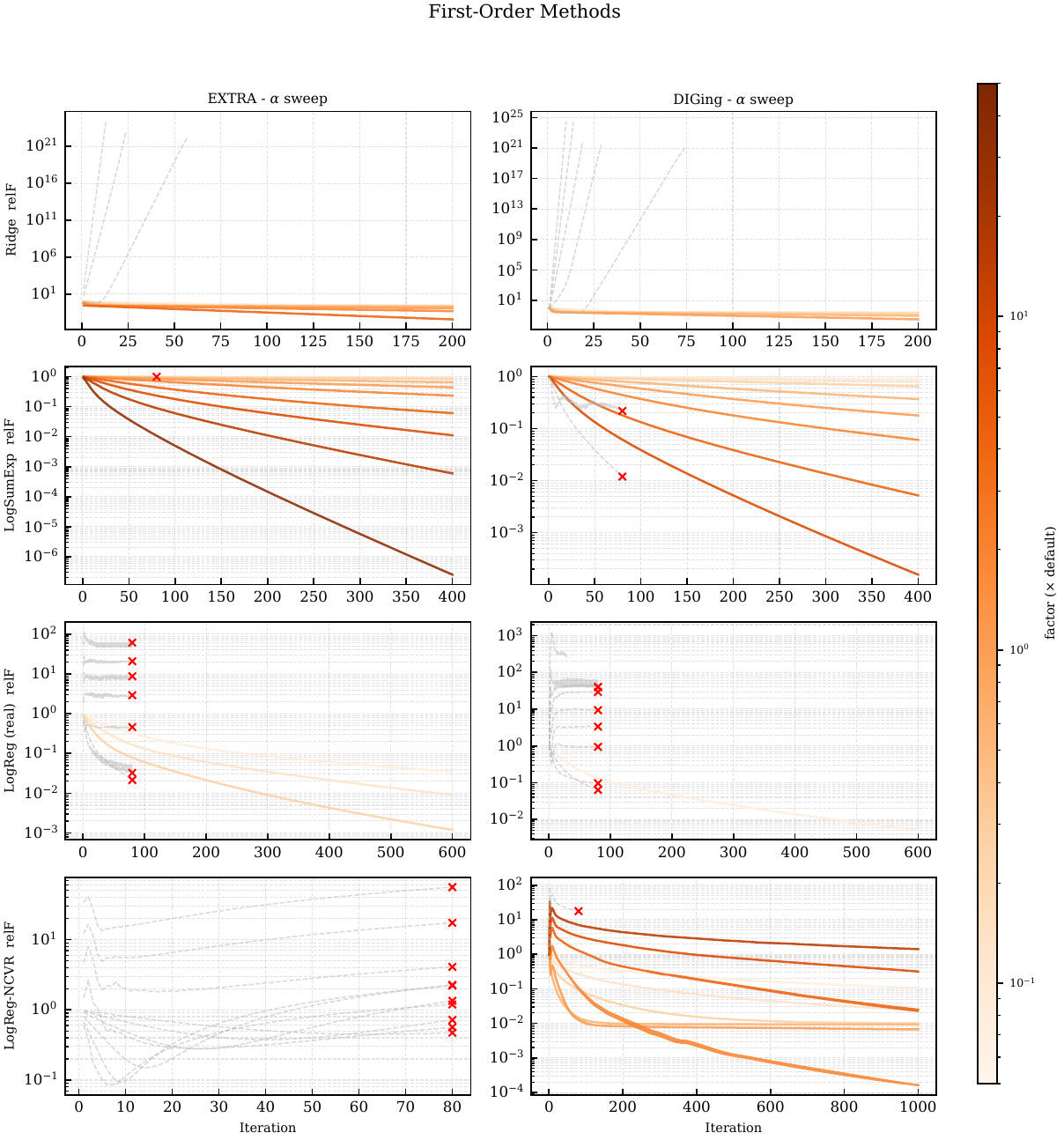}
		\caption{Stepsize sensitivity for first-order baselines
			(EXTRA and DIGing) on four functions.
			10 values of $\alpha$ (light-to-dark).
			Divergent runs are clipped at $\mathrm{relF}=10^4$.}
		\label{fig:sweep_first}
	\end{figure}
	
	\begin{figure}[htbp]
		\centering
		\includegraphics[width=\textwidth]{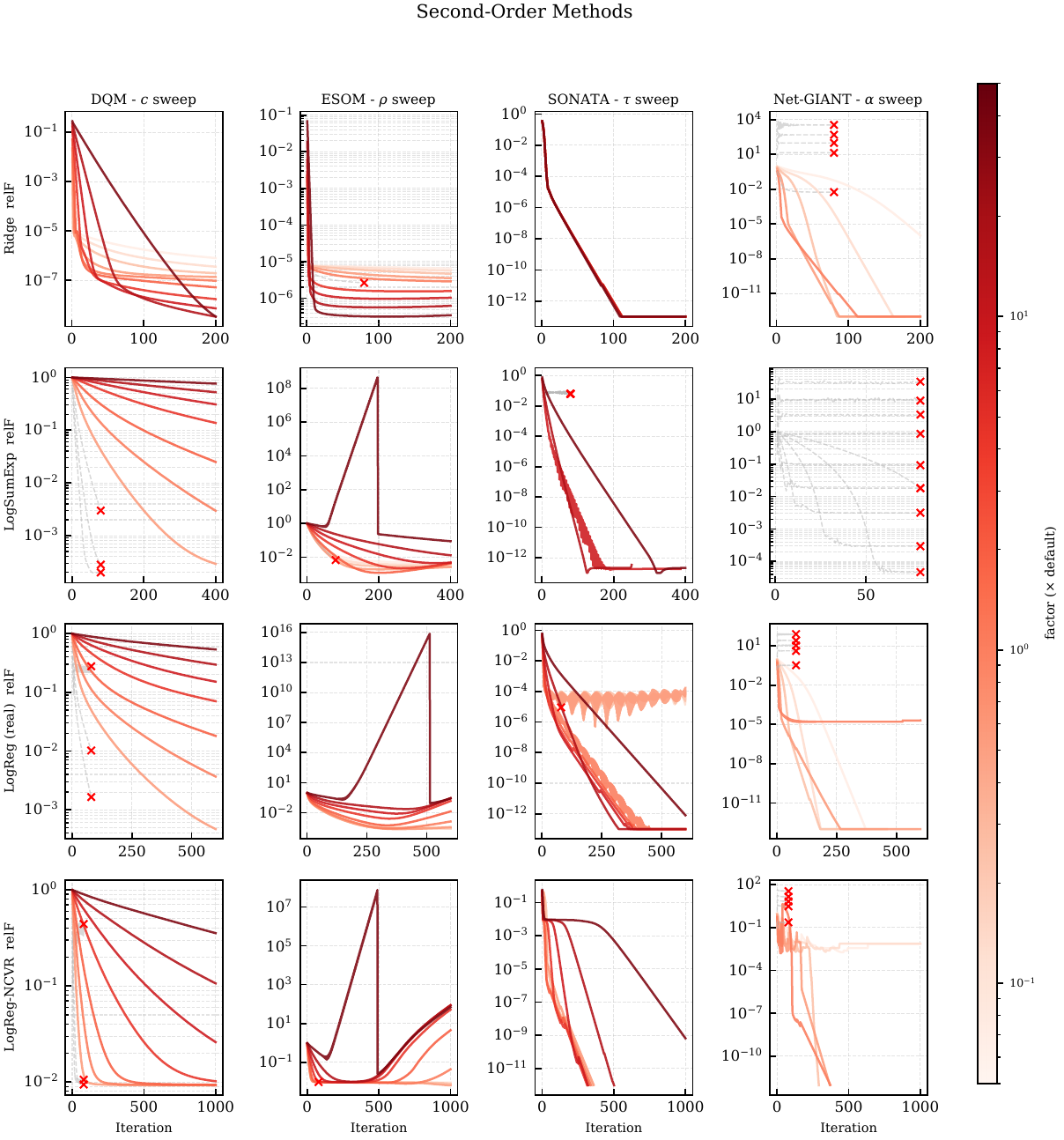}
		\caption{Key-parameter sensitivity for second-order baselines
			on four functions: penalty~$c$ (DQM), penalty~$\rho_{\mathrm{E}}$ (ESOM),
			regularization~$\tau$ (SONATA), stepsize~$\alpha$ (Net-GIANT).
			Entries marked by $\times$ indicate parameter settings with
			NaN, overflow, or divergence.}
		\label{fig:sweep_second}
	\end{figure}
	
\end{document}

%% file: preamble.tex
\usepackage[margin=1in]{geometry}
\usepackage{amsmath,amsthm}
\usepackage{amsfonts,amssymb}
\usepackage{booktabs}
\usepackage{float}
\usepackage{graphicx}
\usepackage{subcaption}
\usepackage{algorithm}
\usepackage{algpseudocode}
\usepackage[utf8]{inputenc}
\usepackage[T1]{fontenc}
\usepackage{lmodern}
\usepackage{microtype}
\usepackage{enumitem}
\usepackage{multirow}
\usepackage[section]{placeins}
\usepackage{xcolor}
\usepackage{ifpdf}
\usepackage[numbers,sort&compress]{natbib}
\usepackage[colorlinks=true,linkcolor=blue!70!black,citecolor=blue!70!black,urlcolor=blue!70!black]{hyperref}
\usepackage[nameinlink]{cleveref}

\usepackage{fancyhdr}
\fancypagestyle{plain}{\fancyhf{}\fancyfoot[C]{\thepage}}

\ifpdf
  \DeclareGraphicsExtensions{.pdf,.png,.jpg,.eps}
\else
  \DeclareGraphicsExtensions{.eps}
\fi

\graphicspath{
	{./}
	{figures/main/}
	{figures/appendix/}
}

\theoremstyle{plain}
\newtheorem{theorem}{Theorem}[section]
\newtheorem{lemma}[theorem]{Lemma}
\newtheorem{proposition}[theorem]{Proposition}

\newtheorem{assumption}[theorem]{Assumption}

\theoremstyle{definition}

\theoremstyle{remark}
\newtheorem{remark}[theorem]{Remark}

\crefname{assumption}{Assumption}{Assumptions}
\Crefname{assumption}{Assumption}{Assumptions}
\crefname{remark}{Remark}{Remarks}
\Crefname{remark}{Remark}{Remarks}

\newcommand{\email}[1]{\href{mailto:#1}{\texttt{#1}}}
\newenvironment{keywords}{\paragraph{Keywords.}}{}
\newenvironment{AMS}{\paragraph{AMS subject classifications.}}{}
\numberwithin{equation}{section}

\title{Distributed Gradient-Regularized Newton Method:
	Scheduled Consensus and \texorpdfstring{$\mathcal{O}(\varepsilon^{-1})$}{O(eps\^{}-1)} Global Iteration Complexity%
	\thanks{Submitted to the editors March~3, 2026.}}

\author{
	Wei Hu\thanks{
		LSEC, ICMSEC, Academy of Mathematics and Systems Science,
		Chinese Academy of Sciences, Beijing 100190, China;
		and University of Chinese Academy of Sciences, Beijing 100049, China
		(\email{huwei@amss.ac.cn}).
		Corresponding author.
	}
	\and
	Pengcheng Xie\thanks{
		Lawrence Berkeley National Laboratory,
		Berkeley, CA 94720, USA
		(\email{pxie@lbl.gov}, \email{pxie98@gmail.com}).
	}
	\and
	Ya-Xiang Yuan\thanks{
		LSEC, ICMSEC, Academy of Mathematics and Systems Science,
		Chinese Academy of Sciences, Beijing 100190, China
		(\email{yyx@lsec.cc.ac.cn}).
	}
	\and
	Li Zhang\thanks{
		LSEC, ICMSEC, Academy of Mathematics and Systems Science,
		Chinese Academy of Sciences, Beijing 100190, China
		(\email{zhangli2022@lsec.cc.ac.cn}).
	}
}

\newcommand{\R}{\mathbb{R}}
\newcommand{\one}{\mathbf{1}}
\newcommand{\ip}[2]{\left\langle #1,#2\right\rangle}
\newcommand{\norm}[1]{\left\|#1\right\|}
\newcommand{\normtwo}[1]{\left\|#1\right\|_2}
\newcommand{\normF}[1]{\left\|#1\right\|_F}
\newcommand{\cO}{\mathcal{O}}
\newcommand{\J}{\left(\frac{1}{N}\one\one^\top\right)}
\newcommand{\Pmat}{I-\J}
\DeclareMathOperator{\blkdiag}{blkdiag}

\allowdisplaybreaks

%% file: main.bbl
\begin{thebibliography}{99}
		
		\bibitem[Cartis et~al.(2011)Cartis, Gould, and Toint]{cartis2011arc}
		C.~Cartis, N.~I.~M. Gould, and Ph.~L. Toint.
		\newblock Adaptive cubic regularisation methods for unconstrained optimization. Part~I: motivation, convergence and numerical results.
		\newblock Mathematical Programming, 127:245--295, 2011.
		
		\bibitem[Chang and Lin(2011)]{chang2011libsvm}
		C.-C.~Chang and C.-J.~Lin.
		\newblock {LIBSVM}: a library for support vector machines.
		\newblock ACM Transactions on Intelligent Systems and Technology, 2(3):27:1--27:27, 2011.
		\newblock Software available at \url{https://www.csie.ntu.edu.tw/~cjlin/libsvmtools/datasets/}.
		
		\bibitem[Eisen et~al.(2017)Eisen, Mokhtari, and Ribeiro]{eisen2017dqn}
		M.~Eisen, A.~Mokhtari, and A.~Ribeiro.
		\newblock Decentralized quasi-{N}ewton methods.
		\newblock IEEE Transactions on Signal Processing, 65(10):2613--2628, 2017.
		
		\bibitem[Jakoveti{\'c} et~al.(2014)Jakoveti{\'c}, Xavier, and Moura]{jakovetic2014fast}
		D.~Jakoveti{\'c}, J.~Xavier, and J.~M.~F. Moura.
		\newblock Fast distributed gradient methods.
		\newblock IEEE Transactions on Automatic Control, 59(5):1131--1146, 2014.
		
		\bibitem[Mishchenko(2023)]{mishchenko2023rn}
		K.~Mishchenko.
		\newblock Regularized {N}ewton method with global $O(1/k^2)$ convergence.
		\newblock SIAM Journal on Optimization, 33(3):1440--1462, 2023.
		
		\bibitem[Mokhtari et~al.(2017)Mokhtari, Ling, and Ribeiro]{mokhtari2017networknewton}
		A.~Mokhtari, Q.~Ling, and A.~Ribeiro.
		\newblock Network {N}ewton distributed optimization methods.
		\newblock IEEE Transactions on Signal Processing, 65(1):146--161, 2017.
		
		\bibitem[Zhang et~al.(2021b)Zhang, Ling, and So]{zhang2021newtontracking}
		J.~Zhang, Q.~Ling, and A.~M.-C. So.
		\newblock A {N}ewton tracking algorithm with exact linear convergence for decentralized consensus optimization.
		\newblock IEEE Transactions on Signal and Information Processing over Networks, 7:346--358, 2021.
		
		\bibitem[Nedi{\'c} and Ozdaglar(2009)]{nedic2009subgradient}
		A.~Nedi{\'c} and A.~Ozdaglar.
		\newblock Distributed subgradient methods for multi-agent optimization.
		\newblock IEEE Transactions on Automatic Control, 54(1):48--61, 2009.
		
		\bibitem[Nedi{\'c} et~al.(2017)Nedi{\'c}, Olshevsky, and Shi]{nedic2017diging}
		A.~Nedi{\'c}, A.~Olshevsky, and W.~Shi.
		\newblock Achieving geometric convergence for distributed optimization over time-varying graphs.
		\newblock SIAM Journal on Optimization, 27(4):2597--2633, 2017.
		
		\bibitem[Nesterov and Polyak(2006)]{nesterov2006cubic}
		Y.~Nesterov and B.~T. Polyak.
		\newblock Cubic regularization of {N}ewton method and its global performance.
		\newblock Mathematical Programming, 108:177--205, 2006.
		
		\bibitem[Shi et~al.(2015)Shi, Ling, Wu, and Yin]{shi2015extra}
		W.~Shi, Q.~Ling, G.~Wu, and W.~Yin.
		\newblock {EXTRA}: an exact first-order algorithm for decentralized consensus optimization.
		\newblock SIAM Journal on Optimization, 25(2):944--966, 2015.
		
		\bibitem[Sun et~al.(2022)Sun, Scutari, and Palomar]{sun2022sonata}
		Y.~Sun, G.~Scutari, and D.~P. Palomar.
		\newblock Distributed nonconvex optimization and learning based on successive convex approximation.
		\newblock IEEE Transactions on Signal Processing, 70:5900--5915, 2022.
		
		\bibitem[Maritan et~al.(2023)Maritan, Sharma, Schenato, and Dey]{maritan2023networkgiant}
		A.~Maritan, G.~Sharma, L.~Schenato, and S.~Dey.
		\newblock Network-{GIANT}: fully distributed {N}ewton-type optimization via harmonic {H}essian consensus.
		\newblock arXiv preprint arXiv:2305.07898, 2023.
		
		\bibitem[Yuan et~al.(2019)Yuan, Ying, and Sayed]{yuan2019exactdiffusion}
		K.~Yuan, B.~Ying, and A.~H. Sayed.
		\newblock Exact diffusion for distributed optimization and learning---Part {I}: algorithm development.
		\newblock IEEE Transactions on Signal Processing, 67(3):708--723, 2019.
		
		
		\bibitem[Jakoveti{\'c} et~al.(2025)Jakoveti{\'c}, Kreji{\'c}, and Malaspina]{jakovetic2025dinas}
		D.~Jakoveti{\'c}, N.~Kreji{\'c}, and G.~Malaspina.
		\newblock {DINAS}: Distributed inexact {N}ewton method with adaptive step sizes.
		\newblock Computational Optimization and Applications, 91:683--715, 2025.
		
		\bibitem[Daneshmand et~al.(2021)Daneshmand, Scutari, Dvurechensky, and Gasnikov]{daneshmand2021newton}
		A.~Daneshmand, G.~Scutari, P.~Dvurechensky, and A.~Gasnikov.
		\newblock Newton method over networks is fast up to the statistical precision.
		\newblock In Proceedings of the 38th International Conference on Machine Learning (ICML), volume 139 of PMLR, pp.~2398--2409, 2021.
		
		\bibitem[Gratton et~al.(2023)Gratton, Jerad, and Toint]{gratton2023nonconvex}
		S.~Gratton, S.~Jerad, and Ph.~L. Toint.
		\newblock Yet another fast variant of {N}ewton's method for nonconvex optimization.
		\newblock arXiv preprint arXiv:2302.10065, 2023.
		
		\bibitem[Doikov and Nesterov(2024)]{doikov2024gradient}
		N.~Doikov and Y.~Nesterov.
		\newblock Gradient regularization of {N}ewton method with {B}regman distances.
		\newblock Mathematical Programming, 204:1--25, 2024.
		
		\bibitem[Yuan et~al.(2023b)Yuan, Li, and Ling]{yuan2023indo}
		G.~Yuan, X.~Li, and Q.~Ling.
		\newblock {INDO}: {IN}version-free {D}istributed second-{O}rder method for consensus optimization.
		\newblock Optimization Online preprint, 2022.
		
		\bibitem[Zhang et~al.(2024)Zhang, Che, Yang, and others]{zhang2024cubic}
		Z.~Zhang, K.~Che, S.~Yang, et~al.
		\newblock Communication-efficient distributed cubic {N}ewton with compressed lazy {H}essian.
		\newblock Neural Networks, 174:106212, 2024.
		
		\bibitem[Alghunaim et~al.(2021)Alghunaim, Ryu, Yuan, and Sayed]{alghunaim2021decentralized}
		S.~A. Alghunaim, E.~K. Ryu, K.~Yuan, and A.~H. Sayed.
		\newblock Decentralized proximal gradient algorithms with linear convergence rates.
		\newblock IEEE Transactions on Automatic Control, 66(6):2787--2794, 2021.
		
		\bibitem[Bajovi{\'c} et~al.(2017)Bajovi{\'c}, Jakoveti{\'c}, Kre\-ji{\'c}, and Krklec~Jerinki{\'c}]{bajovic2017newton}
		D.~Bajovi{\'c}, D.~Jakoveti{\'c}, N.~Kre\-ji{\'c}, and N.~Krklec~Jerinki{\'c}.
		\newblock Newton-like method with diagonal correction for distributed optimization.
		\newblock SIAM Journal on Optimization, 27(2):1171--1203, 2017.
		
		\bibitem[Beznosikov et~al.(2022)Beznosikov, Richt{\\'a}rik, Diskin, and others]{beznosikov2022biased}
		A.~Beznosikov, P.~Richt{\'a}rik, M.~Diskin, et~al.
		\newblock Distributed methods with compressed communication for solving variational inequalities, with theoretical guarantees.
		\newblock In Advances in Neural Information Processing Systems (NeurIPS), 35:14013--14029, 2022.
		
		\bibitem[Doikov and Nesterov(2022)]{doikov2022tensor}
		N.~Doikov and Y.~Nesterov.
		\newblock Local convergence of tensor methods.
		\newblock Mathematical Programming, 193:315--336, 2022.
		
		\bibitem[Doikov et~al.(2023)Doikov, Chayti, and Jaggi]{doikov2023lazy}
		N.~Doikov, E.~M.~Chayti, and M.~Jaggi.
		\newblock Second-order optimization with lazy {H}essians.
		\newblock In Proceedings of the International Conference on Machine Learning (ICML), 2023.
		
		\bibitem[Koloskova et~al.(2019)Koloskova, Stich, and Jaggi]{koloskova2019decentralized}
		A.~Koloskova, S.~U. Stich, and M.~Jaggi.
		\newblock Decentralized stochastic optimization and gossip algorithms with compressed communication.
		\newblock In Proceedings of the International Conference on Machine Learning (ICML), pp.~3478--3487, 2019.
		
		\bibitem[Li and Lin(2024)]{li2024decentralized}
		H.~Li and Z.~Lin.
		\newblock Accelerated gradient tracking over time-varying graphs for decentralized optimization.
		\newblock Journal of Machine Learning Research, 25(274):1--52, 2024.
		
		\bibitem[Mokhtari et~al.(2016)Mokhtari, Shi, Ling, and Ribeiro]{mokhtari2016esom}
		A.~Mokhtari, W.~Shi, Q.~Ling, and A.~Ribeiro.
		\newblock {ESOM}: A second-order method for exact decentralized optimization over networks.
		\newblock IEEE Transactions on Signal and Information Processing over Networks, 2(4):507--522, 2016.
		
		\bibitem[Pu and Nedi{\'c}(2021)]{pu2021distributed}
		S.~Pu and A.~Nedi{\'c}.
		\newblock Distributed stochastic gradient tracking methods.
		\newblock Mathematical Programming, 187:409--457, 2021.
		
		\bibitem[Xin and Khan(2020)]{xin2020distributed}
		R.~Xin and U.~A. Khan.
		\newblock Distributed heavy-ball: a generalization and acceleration of first-order methods with gradient tracking.
		\newblock IEEE Transactions on Automatic Control, 65(6):2627--2633, 2020.
		
		\bibitem[Li et~al.(2020)Li, Cen, Chen, and others]{li2020communication}
		B.~Li, S.~Cen, Y.~Chen, et~al.
		\newblock Communication-efficient distributed optimization in networks with gradient tracking and variance reduction.
		\newblock Journal of Machine Learning Research, 21(180):1--51, 2020.
		
		\bibitem[Nesterov(2005)]{nesterov2005smooth}
		Y.~Nesterov.
		\newblock Smooth minimization of non-smooth functions.
		\newblock Mathematical Programming, 103(1):127--152, 2005.
		
		\bibitem[Charbonnier et~al.(1997)Charbonnier, Blanc-F{\'e}raud, Aubert, and Barlaud]{charbonnier1997huber}
		P.~Charbonnier, L.~Blanc-F{\'e}raud, G.~Aubert, and M.~Barlaud.
		\newblock Deterministic edge-preserving regularization in computed imaging.
		\newblock IEEE Transactions on Image Processing, 6(2):298--311, 1997.
		
		\bibitem[Rosenbrock(1960)]{rosenbrock1960automatic}
		H.~H.~Rosenbrock.
		\newblock An automatic method for finding the greatest or least value of a function.
		\newblock The Computer Journal, 3(3):175--184, 1960.
		
		\bibitem[Styblinski and Tang(1990)]{styblinski1990experiments}
		M.~A.~Styblinski and T.~S.~Tang.
		\newblock Experiments in nonconvex optimization: Stochastic approximation with function smoothing and simulated annealing.
		\newblock Neural Networks, 3(4):467--483, 1990.
		
		\bibitem[Geman and Yang(1995)]{geman1995nonlinear}
		D.~Geman and C.~Yang.
		\newblock Nonlinear image recovery with half-quadratic regularization.
		\newblock IEEE Transactions on Image Processing, 4(7):932--946, 1995.

        \bibitem[Xie and Yuan(2023)]{xie2023linesearch}
P.~Xie and Y.~Yuan.
\newblock A derivative-free optimization algorithm combining line-search and trust-region techniques.
\newblock {\em Chinese Annals of Mathematics, Series B}, 44(5):719--734, 2023.

\bibitem[Xie and Yuan(2025)]{xie2023dfoto}
P.~Xie and Y.~Yuan.
\newblock Derivative-free optimization with transformed objective functions {(DFOTO)} and the algorithm based on the least {Frobenius} norm updating quadratic model.
\newblock {\em Journal of the Operations Research Society of China}, 13:327--363, 2025.

\bibitem[Xie and Yuan(2025)]{xieyuannew}
P.~Xie and Y.~Yuan.
\newblock A derivative-free method using a new underdetermined quadratic interpolation model.
\newblock {\em SIAM Journal on Optimization}, 35(2):1110--1133, 2025.

\bibitem[Xie and Yuan(2026)]{10.1093/imanum/drae106}
P.~Xie and Y.~Yuan.
\newblock Least {$H^2$} norm updating of quadratic interpolation models for derivative-free trust-region algorithms.
\newblock {\em IMA Journal of Numerical Analysis}, 46(1):21--50, 2026.

\bibitem[Xie and Yuan(2026)]{xie2023twodimensional}
P.~Xie and Y.~Yuan.
\newblock A new two-dimensional model-based subspace method for large-scale unconstrained derivative-free optimization: {2D-MoSub}.
\newblock {\em Optimization Methods and Software}, 41(1):118--150, 2026.

\bibitem[Xie and Wild(2025)]{xie2025remuregionalminimalupdating}
P.~Xie and S.~M. Wild.
\newblock {ReMU}: Regional minimal updating for model-based derivative-free optimization.
\newblock arXiv:2504.03606, 2025.

\bibitem[He and Xie(2025)]{he2025modeldrivensubspaceslargescaleoptimization}
Y.~He and P.~Xie.
\newblock Model-driven subspaces for large-scale optimization with local approximation strategy.
\newblock arXiv:2509.08256, 2025.

\bibitem[Xie(2023)]{xie2023ellipsoid}
P.~Xie.
\newblock A derivative-free trust-region method for optimization on the ellipsoid.
\newblock {\em Journal of Physics: Conference Series}, 2620:012007, 2023.

\bibitem[Xie(2025)]{XIE2025116146}
P.~Xie.
\newblock Sufficient conditions for error distance reduction in the $\ell^2$-norm trust region between minimizers of local nonconvex multivariate quadratic approximates.
\newblock {\em Journal of Computational and Applied Mathematics}, 453:116146, 2025.


\bibitem[Xie et~al.(2025)]{xie2025objectivevaluechangeshapebased}
P.~Xie, Z.~Zhou, and Z.~Zhou.
\newblock Objective value change and shape-based accelerated optimization for the neural network approximation.
\newblock arXiv:2508.20290, 2025.

\bibitem[Xie and Wild(2024)]{xie2024bary}
P.~Xie and S.~M. Wild.
\newblock Barycenter of weight coefficient region of least weighted {$H^2$} norm updating quadratic models with vanishing trust-region radius.
\newblock {\em SIAM NCC 2024, Early Career Travel Award}, 2024.

\bibitem[Xie(2024)]{zhang2024relationshiplambdapoisednessderivativefreeoptimization}
P.~Xie.
\newblock On the relationship between $\Lambda$-poisedness in derivative-free optimization and outliers in local outlier factor.
\newblock arXiv:2407.17529, 2024.

\bibitem[Li et~al.(2025)]{li2025novelnumericalmethodtailored}
L.~Li, P.~Xie, and L.~Zhang.
\newblock A novel numerical method tailored for unconstrained optimization problems.
\newblock arXiv:2504.02832, 2025.

\bibitem[Xie(2024)]{2023multiple}
P.~Xie.
\newblock An efficient derivative-free method for finding multiple solutions.
\newblock To be posted on arXiv, 2024.

\bibitem[Li et~al.(2025)]{li2025spectrallevenbergmarquardtdeflationmethodmultiple}
L.~Li, Y.~Zhou, P.~Xie, and H.~Li.
\newblock A spectral {Levenberg--Marquardt--Deflation} method for multiple solutions of semilinear elliptic systems.
\newblock {\em Journal of Computational and Applied Mathematics}, 2025.

\bibitem[Ye et~al.(2025)]{ye2025improvedadaptiveorthogonalbasis}
Y.~Ye, L.~Li, P.~Xie, and H.~Yu.
\newblock An improved adaptive orthogonal basis deflation method for multiple solutions with applications to nonlinear elliptic equations in varying domains.
\newblock {\em Journal of Computational Mathematics}, 2025.

\bibitem[Xie(2025)]{xie2025privacypreservingblackboxoptimizationpbbo}
P.~Xie.
\newblock Privacy-preserving black-box optimization {(PBBO)}: Theory and the model-based algorithm {DFOp}.
\newblock arXiv:2601.11570, 2025.

\bibitem[Xie et~al.(2024)]{xie2024lchange}
P.~Xie et~al.
\newblock A novel local analysis of objectives approximated by neural network: {L-Change}.
\newblock International Conference on Mathematical Theory of Deep Learning {(MTDL)}, 2024.

\bibitem[Dzahini et~al.(2025)]{postionpaper2025optimization}
K.~J. Dzahini, S.~M. Wild, and P.~Xie.
\newblock Optimization approaches for solving inverse problems must account for uncertainty in both data and downstream decisions.
\newblock Position paper, Inverse Methods for Complex Systems under Uncertainty Workshop, Sponsored by the U.S. Department of Energy, Office of Science, Advanced Scientific Computing Research, 2025.

\bibitem[Xie and Tao(2019)]{xie2019acc}
P.~Xie and M.~Tao.
\newblock Parametric resonant control of macroscopic behaviors of multiple oscillators.
\newblock In {\em 2019 American Control Conference (ACC)}, pages 1898--1905, 2019.

\bibitem[Xie(2024)]{xie2023invariant}
P.~Xie.
\newblock A note on the invariant distribution of a stochastic dynamical system.
\newblock 2024.

\bibitem[Xie(2024)]{10.11648/j.acm.20241305.13}
P.~Xie.
\newblock The modeling and optimization of a multi-dam system.
\newblock {\em Applied and Computational Mathematics}, 13(5):140--152, 2024.


		
	\end{thebibliography}
